\documentclass[reqno]{amsart}

\usepackage[english]{babel}
\usepackage{amsmath}
\usepackage{amsthm}
\usepackage{amsfonts}
\usepackage{amssymb}
\usepackage{anysize}
\usepackage{hyperref}
\usepackage{appendix}
\usepackage{graphicx}

\newtheorem{defi}{Definition}[section]
\newtheorem{lem}[defi]{Lemma}
\newtheorem{theo}[defi]{Theorem}
\newtheorem{cor}[defi]{Corollary}
\newtheorem{pro}[defi]{Proposition}

\newtheorem{rem}[defi]{Remark}

\DeclareMathOperator{\divop}{div}

\DeclareMathOperator{\jac}{Jac}
\DeclareMathOperator{\trace}{Tr}

\DeclareMathOperator{\dist}{dist}

\DeclareMathOperator{\Sym}{Sym}

\DeclareMathOperator{\RR}{\mathbb{R}}

\allowdisplaybreaks 

\title[Euler-type equations  and hyperbolic limits of kinetic Cucker-Smale models]{Euler-type equations and commutators in singular and hyperbolic limits of kinetic Cucker-Smale models}
\author{David Poyato}
\address{Departamento de Matem\'atica Aplicada, Universidad de Granada, 18071 Granada, Spain}
\email{davidpoyato@ugr.es}

\author{Juan Soler}
\address{Departamento de Matem\'atica Aplicada, Universidad de Granada, 18071 Granada, Spain}
\email{jsoler@ugr.es}
{\thanks{This work has been partially supported by the MINECO-Feder (Spain) research grant number MTM2014-53406-R, the Junta de Andaluc\'ia (Spain) Project FQM 954, and the MECD (Spain) research grant FPU2014/06304 (D.P.).
}}

\begin{document}

\begin{abstract}
This paper deals with the derivation and analysis of a compressible Euler--type equation with singular commutator, which is derived from a hyperbolic limit of the kinetic description to the Cucker--Smale model of interacting individuals.
\end{abstract}

\maketitle

\section{Introduction}\label{Introduction.section}

The aim of this paper is to rigorously derive and analyze the following system of equations
of Euler-type with singular commutators arising from hyperbolic limits of kinetic Cucker--Smale models
\begin{equation}\label{Euler-limit.intro}
\left\{
\begin{array}{ll}
\displaystyle\frac{\partial \rho}{\partial t}+\divop (\rho u)=0, & t\geq 0,\,x\in\RR^N\\
\nabla\psi + \mu u = \phi_0*(\rho u)-(\phi_0*\rho)u, & t\geq 0,\,x\in\RR^N,
\end{array}
\right.
\end{equation} 
where $\rho=\rho(t,x)$ stands for the density, $u=u(t,x)$ is the velocity field of the population at time $t$ and position $x$ and $N$ is the space dimension (the physically meaningful cases are $N=2$ and $N=3$). The velocity and the density are coupled through the second equation in \eqref{Euler-limit.intro}, where some linear friction with constant coefficient $\mu\geq 0$ is considered, some external or internal force $F=-\nabla\psi$ associated with a potential $\psi=\psi(t,x)$ is assumed to act on the system and the above singular kernel in the commutator takes the form
\begin{equation}\label{influencefunction.singular.form}
\phi_0(r)= \frac{1}{c_\lambda^\lambda}\frac{1}{r^{2 \lambda}}.
\end{equation}
Then, the commutator can be written in the following way
\begin{eqnarray}\label{lc}
\phi_0*(\rho u)-(\phi_0*\rho)u=-\frac{1}{c_\lambda^\lambda}[M_u,I_{N-2\lambda}]\rho=-\frac{1}{c_\lambda^\lambda}\int_{\RR^N}\frac{u(t,x)-u(t,y)}{\vert x-y\vert^{2\lambda}}\rho(t,y)\,dy,
\end{eqnarray}
where $M_u$ is the multiplication operator by $u$ and $I_{\alpha}f$ denotes the Riesz potential of order $\alpha\in (0,N)$ generated by a scalar measurable function $f$ on $\RR^N$. For the sake of simplicity, we forget about the standard constant arising from the fundamental solution of the fractional Laplacian $(-\Delta)^{\alpha/2}$ and we simply write 
$$(I_\alpha f)(x)=\left(\frac{1}{\vert \cdot\vert^{N-\alpha}}\ast f\right)(x)=\int_{\RR^N}\frac{f(y)}{\vert x-y\vert^{N-\alpha}}\,dy,$$
see \cite[Chapter V, Section 1]{Stein} for more details.

The coupling equation in (\ref{Euler-limit.intro}) collects the nonlocal alignment effects of the velocity of the system in terms of the friction and the potential $\psi$ that, as we will see later, inherits the behavior of the microscopic system from which \eqref{Euler-limit.intro} is a macroscopic derivation.  As we have mentioned before, $\psi$ can be considered external or internal (i.e., self-generated). Hence, in the former case one restricts to linear terms whilst the latter case does not only cover local nonlinearities but also nonlocal ones of the type $\psi=k \ast \rho$. Such choices open our scope of work and connect \eqref{Euler-limit.intro}-\eqref{influencefunction.singular.form} with other swarming, aggregation, generalized surface quasi-geostrophic (gSQG) or Euler-type models arising in the literature \cite{BertozziVerdera,CastroCordobaGomezSerrano,delaHozHassainiaHmidi,TopazBertozzi1,TopazBertozzi2}.

The connection between microscopic and macroscopic scales of description of the dynamical properties of a large system of particles intends for the achieved models to collect the interaction of individuals towards the self-organization of a complex system. This is a currently challenging issue in different areas of Science steering towards the attainment of more complex rules beyond the Newtonian-type laws and it constitutes a particularly active area in the biological description of emergent processes in complex populations. Depending on the biological context, some of the emerging dynamics are called \textit{swarming}, \textit{schooling} or \textit{flocking}, where Cucker--Smale model \cite{CuckerSmale1,CuckerSmale2,HaTadmor} is relevant. For different models in this context we refer, for instance, to \cite{[CCG10B],BS12,BDT99,Couzin,DM08,HildenbrandtCarereHemelrijk,Marche,MotschTadmor2,Reynolds,Viscek} and the references therein, while in the same spirit but in the framework of social and cellular interactions, we can mention \cite{BBNS07,BBNS10B,BCKS,traffic,BKS13,BPRS98,GSR,HFV00,Marche,NOW06,PT13} among others. From the biological viewpoint, the main interest is to learn the convoluted set of rules stating how agents in an ecosystem self-organize. From such \textit{local} organization or \textit{cooperation} between the individuals in neighboring areas, one would expect to recover some kind of \textit{global emergent behavior} as a complex mechanism leading to a dynamics driven by the collective motion of the rest of agents. Some well known examples in the nature are the formation of a flock of birds, a school of fish or a swarm. In this sense, the individuals in the ecosystem are self-driven by the collective global behavior of the whole population, and, in turns, it is determined by the local dynamics itself.

Note that the above system  \eqref{Euler-limit.intro}-\eqref{influencefunction.singular.form} can be compared with the recently proposed \textit{compressible Euler equations with nonlocal alignment} arising in the setting of Cucker--Smale dynamics of flocking
\begin{equation}\label{Euler-alignment.eq}
\left\{
\begin{array}{l}
\displaystyle\frac{\partial\rho}{\partial t}+\divop(\rho u)=0,\\
\displaystyle\frac{\partial}{\partial t}\left(\rho	u\right)+\divop(\rho u\otimes u)=-\nabla \rho-\int_{\RR^N}\phi(\vert x-y\vert)(u(t,x)-u(t,y))\rho(t,x)\rho(t,y)\,dy,
\end{array}
\right.
\end{equation}
see \cite{CarrilloChoiTadmorTan,TadmorTan} for some properties of such models. Specifically, when inertia effect (i.e., acceleration) is neglected in such hydrodynamic nonlocal alignment model but linear friction and some force $F=-\nabla\psi$ are assumed to act on the system, then we are formally led to the above general system (\ref{Euler-limit.intro}). Nevertheless, the rigorous derivation of the  Euler system with nonlocal alignment effects (\ref{Euler-alignment.eq}) has been mainly obtained from mesoscopic models in a formal way through a method of closure of the hierarchy of velocity moments based on the choice of a monokinetic distribution of particles \cite{CarrilloChoiTadmorTan,CarrilloFornasierToscaniVecil,ChuangDorsognaMarthalerBertozziChayes,HaTadmor,MotschTadmor2,TadmorTan}. Recently, some attempts to derive such models have been addessed via rigorous hydrodynamic limits in \cite{CarrilloChoiKarper,KangVasseur,KarperMelletTrivisa}. The first two references deal with what authors call \textit{local alignment force} that can be understood as a ``viscous force'' with respect to the mean velocity (it can also be derived from a Dirac delta limit in the the Mostch--Tadmor kinetic alignment term) whilst the latter also focuses on including Cucker--Smale type alignment operators in such a way that they are not lost in the hydrodynamic limit (as it happens in \cite{CarrilloChoiKarper,KangVasseur} due to the symmetries of the local alignment force).  

To our best knowledge, no rigorous derivation has been achieved so far, apart from \cite{KarperMelletTrivisa}, where regular influence functions $\phi=\phi(r)$ are considered. Specifically, the underlying choice of the influence function is the classical by Cucker and Smale \cite{CuckerSmale1,CuckerSmale2}. It is an interaction kernel that decreases as interparticle distance $r$ increases, it is bounded near the origin (i.e., at $r=0$) and takes the form:
\begin{eqnarray}\label{blc}
\phi(r)=\frac{1}{(1+r^2)^\lambda},
\end{eqnarray}
where $\lambda>0$ is some parameter that measures the fall-off of the interactions between individuals separated by large distances. Naturally, the lack of singularity at $r=0$ (like in the framework of Newtonian interactions) simplifies the mathematical handling of the kernel and nonlinear terms appearing in the Euler equation with nonlocal alignment. 
From a mathematical point of view, some analysis has been done in the singular case (i.e., $\phi_0(r)$ given by (\ref{influencefunction.singular.form})) for values of the parameter $\lambda\in (0,1/2)$. See \cite{CarrilloChoiHauray,HaLiu,Peszek,Peszek2}, where the authors are concerned with the local in-time well posedness of the kinetic model, particle asymptotics, existence of piecewise weak solutions and existence and uniqueness of $W^{1,1}$ strong solutions to the particle system. As discussed in \cite{MuchaPeszek}, regular influence functions aims to prevent distant interparticle interactions. Nevertheless, in order to augment the local interactions, a possible solution might be the one considered in \cite{MotschTadmor}  (assuming asymmetric interactions) or the model proposed in \cite{Peszek} (with singular interactions). Indeed, including such singular kernels leads to new asymptotics in the modeling of flocking behavior. Namely, \cite{Peszek} shows some examples where (\ref{influencefunction.singular.form}) with $\lambda\in (0,1/2)$ entails finite time collisions (sticking indeed) of trajectories.

In this work, we will be interested in rigorously obtaining measure-valued solutions to \eqref{Euler-limit.intro}-\eqref{influencefunction.singular.form} through a singular limit in the classical choice of the influence function devoted to Cucker and Smale. It will be done via the derivation of a hydrodynamic singular limit on a hyperbolic scaling for the kinetic Cucker--Smale system of flocking with friction and diffusion effects in the velocity variable $v$, namely
\begin{equation}\label{VPFPCS-dimensionless-hyperbolicscaling.Intro.form}
\varepsilon\frac{\partial f_\varepsilon}{\partial t}+\varepsilon v\cdot\nabla_x f_\varepsilon-\nabla_x\psi_\varepsilon\cdot \nabla_v f_\varepsilon=\divop_v\left(f_\varepsilon v+\nabla_v f_\varepsilon+Q_{CS}^{\phi_\varepsilon}(f_\varepsilon,f_\varepsilon)\right),
\end{equation}
where the scaled Cucker--Smale operator reads 
$$Q_{CS}^{\phi_\varepsilon}(f_\varepsilon,f_\varepsilon)(t,x,v)=f_\varepsilon(t,x,v) \int_{\RR^N}\int_{\RR^N}\phi_\varepsilon (\vert x-y\vert)(v-w)f_\varepsilon(t,y,w)\,dy\,dw.$$
The interaction potential will be scaled as follows
\begin{equation}\label{influencefunction.regular.form}
\phi_\varepsilon(r)=\frac{1}{(\varepsilon^2+c_\lambda r^2)^\lambda},
\end{equation}
for some $\lambda$-dependent coefficient $c_\lambda$. As it is apparent now, the formal limit $\varepsilon\rightarrow 0$ amounts to singular interactions modeled by the Riesz kernel $\phi_0$ of order $N-2\lambda$. 
In a general context, we can see \eqref{VPFPCS-dimensionless-hyperbolicscaling.Intro.form}-\eqref{influencefunction.regular.form} as a way to obtain an approximate sequence of solutions whose macroscopic quantities converge towards a measure-valued solution of \eqref{Euler-limit.intro}-\eqref{influencefunction.singular.form} in a weak sense. Apart from that hyperbolic scaling, we might wonder about other relevant scalings that have been considered in the literature (e.g. parabolic and intermediate scaling). For instance, see \cite{BellouquidCalvoNietoSoler-Intermedio,NietoPoupaudSoler-Hiperbolico,PoupaudSoler-Parabolico} for a comprehensive study of the intermediate, hyperbolic and parabolic limits in the Vlasov--Poisson--Fokker--Planck system. We will extend our results to an appropriately chosen intermediate scaling where the velocity diffusion effects (that are inherited from the randomness in the microscopic equations) does not dominate, but disappears in the limit. 

The main difficulty both in the passage to the limit and in analyzing the limiting hydrodynamic systems obviously relies on handling with commutators. The commutator in the hydrodynamic limit reads
$$
(\phi_\varepsilon*j_\varepsilon)\rho_\varepsilon-(\phi_\varepsilon*\rho_\varepsilon)j_\varepsilon=\int_{\RR^{N}}\frac{1}{(\varepsilon^2+c_\lambda\vert x-y\vert^2)^\lambda}\left(\rho_\varepsilon(t,x)j_\varepsilon(t,y)-\rho_\varepsilon(t,y)j_\varepsilon(t,x)\right)\,dy,
$$
where $\rho_\varepsilon$ and $j_\varepsilon$ stand for the density and current of particles associated with the particles distribution function $f_\varepsilon$, i.e.,
$$\rho_\varepsilon(t,x):=\int_{\RR^N}f_\varepsilon(t,x,v)\,dv,\hspace{0.5cm}j_\varepsilon(t,x):=\int_{\RR^N} vf_\varepsilon(t,x,v)\,dv.$$
Naturally, we have some difficulties when trying to give some sense to the corresponding commutator when $\varepsilon\rightarrow 0$ since we do not expect $\rho_\varepsilon$ and $j_\varepsilon$ to converge stronger than in the weak-star sense of finite Radon measures. This is due to the fact that the distribution function and some of its functional moments with respect to velocity will be only bounded in $L^1$, independently of $\varepsilon$. The main result of this paper then reads as follows (for a more precise statement see Section \ref{Section-Hyperbolic}).

\begin{theo}\label{Convergence.aproximate.solutions-intro.teo}
Under appropriate hypothesis on the initial data $f_\varepsilon^0$ and the external forces $-\nabla\psi_\varepsilon$, let $f_\varepsilon$ be the smooth solutions to (\ref{VPFPCS-dimensionless-hyperbolicscaling.Intro.form}) with $\lambda\in (0,1/2)$. Then, the macroscopic quantities $\rho_\varepsilon$ and $j_\varepsilon$ converge in a weak sense to some finite Radon measure $\rho$ and $j$ that solve the Cauchy problem associated with the following Euler-type system in the distributional sense
\begin{equation}\label{Euler-limit-j.intro}
\left\{
\begin{array}{ll}
\displaystyle \partial_t\rho +\divop j=0, & x\in\RR^N,\,t\in [0,T),\\
\displaystyle \rho\,\nabla\psi+j=(\phi_0*j)\rho-(\phi_0*\rho)j, & x\in\RR^N,\, t\in [0,T)\\
\displaystyle \rho(0,\cdot)=\rho^0, & x\in \RR^N.
\end{array}
\right.
\end{equation}
\end{theo}

The most sketchy part in the above result shows how to take limits in the product $\rho_\varepsilon (t,x) j_\varepsilon(t,y)$ and rigorously pass to the limit in the commutator. Like in other related systems, one can solve such issue for $\lambda\in (0,1/2)$ by considering the above term in weak form and cancelling the singularity through the use of appropriate test functions. Nevertheless, the same cannot be directly ensured for the endpoint case $\lambda=1/2$ (compare with the \textit{2D Euler equations in vorticity formulation} \cite{Schochet}), that indeed agrees with the critical value of the parameter $\lambda$ over which one cannot expect unconditional flocking in the Cucker--Smale system (see \cite{Carrillo-Improve,HaTadmor}). In the 2D Euler system in vorticity formulation, some sort of logarithmic Morrey space estimate for the sequence of approximated vorticities that holds true for nonnegative $H^{-1}(\RR^2)$ initial vorticities is needed in order to prevent the system from concentrations on the diagonal and to pass to the limit (see \cite{Delort,Majda,VecchiWu}). Note that another elemental difference with such case is the type of symmetries of the kernel: while the Biot--Savart kernel in the 2D Euler equation is odd, our Riesz kernels are even and do not admit similar cancellations. In our case, our scaled systems do not enjoy any additional estimate with respect to Morrey--type norms apart from a bound on what we will call the \textit{dissipation of the kinetic energy due to alignment interactions} between particles, i.e.,
$$\int_0^T\int_{\RR^N}\int_{\RR^N}\int_{\RR^N}\int_{\RR^N}\vert v-w\vert^2\phi_\varepsilon(\vert x-y\vert)f_\varepsilon(t,x,v)f_\varepsilon(t,y,w)\,dx\,dy\,dv\,dw\,dt\leq C,\ \ \forall\,\varepsilon>0.$$
Such bound will help us on proving that the singular limit $\varepsilon\rightarrow 0$ involves no concentration in the limiting nonlinear term for all the values of the parameter $\lambda\in (0,N/2)$ in the same spirit as in \cite{Schochet} or the 2D-case in \cite{NietoPoupaudSoler-Hiperbolico}. Despite the aforementioned lack of concentrations, we will show that the singular limit $\varepsilon\rightarrow 0$ in the kernel $\phi_\varepsilon$ will only allow rigorously passing to the limit in the nonlinear term for the restricted interval of the parameter $\lambda\in (0,1/2]$ as a consequence of the lack of extra estimates for the density and current of particles.  

The above quantity also appears in some preceding results \cite{CarrilloChoiKarper,KangVasseur,KarperMelletTrivisa} as a dissipation term of a ``generalized entropy'' that is composed of the Boltzmann entropy, the kinetic energy and the second order moment with respect to position. The symmetries of the aforementioned local alignment forces along with the boundedness of the influence function \eqref{blc} near $r=0$ and an appropriate choice of the scaling in the mesoscopic system ensure the presence of such entropy inequality that supports a relative entropy method when passing to the macroscopic system. However, as it will be checked later, we cannot ensure that our scaled system (\ref{VPFPCS-dimensionless-hyperbolicscaling.Intro.form}) do enjoy such kind of entropy results. It stands to reason that it is due to the singular scaling for $\phi_\varepsilon$ and the lack of $L^p$ norms for the macroscopic density $\rho=\rho(t,x)$. Let us consequently advert that in our hyperbolic and intermediate limits, the presence of the velocity friction term $\divop_v(f_\varepsilon v)$ in the Fokker--Planck differential operator is essential to derive some of the only possible estimates for the current of particles that can be obtained from such scaled systems, in contrast with the scaling in \cite{KangVasseur}.

Once Theorem \ref{Convergence.aproximate.solutions-intro.teo} is proved, one can formally cancel a factor $\rho$ in the limiting equation (\ref{Euler-limit-j.intro}) by considering $j=\rho\,u$. Then, we are formally led to Equation (\ref{Euler-limit.intro}) and the commutator in (\ref{lc}). This is somehow surprising because such commutator have not a clear sense for $\rho \in L^1(\RR^N)$. For such limiting Equation (\ref{Euler-limit.intro}), that couples the velocity field and density function, you can deal with the existence of regular strong solutions for the Cauchy problem associated with the system \eqref{Euler-limit.intro}-\eqref{influencefunction.singular.form}. To this end, one has to obtain accurate bounds for the commutator operator $[M_u,I_{N-2\lambda}]$ of a Riesz potential and a multiplication operator. The attainment of such bounds in $L^p$ spaces is a classical topic in harmonic analysis (see \cite{Chanillo,CoifmanRochbergWeiss,CruzUribe,Perez}). In particular, it is known, by the \textit{Hardy--Littlewood--Sobolev theorem of fractional integrals} \cite[Theorem 1.2.1]{Stein}, that the Riesz potential of order $N-2\lambda$ is a bounded linear operator between the next spaces
$$I_{N-2\lambda}:L^p(\RR^N)\longrightarrow L^q(\RR^N),$$
where $1<p<q<\infty$ and $1/p-(N-2\lambda)/N=1/q$. The boundedness when $p=1$ (then $q=N/2\lambda$) is not true in the $L^p-L^q$ sense but in the $L^p-L^{q,\infty}$ way, where $L^{q,\infty}(\RR^N)$ stands for the weak Lebesgue space (or Lorentz spaces). Regarding the commutator, one recovers the boundedness of the next linear operator
$$[M_u,I_{N-2\lambda}]:L^p(\RR^N)\longrightarrow L^q(\RR^N,\RR^N),$$
if, and only if, $u\in BMO(\RR^N,\RR^N)$ \cite{Chanillo}. Nevertheless, the endpoint case $p=1$ cannot be recovered in the $L^p-L^{q,\infty}$ sense \cite{CruzUribe}, which is the natural framework according to the lack of estimates for $\rho$. The main difficulty when dealing with the Equation (\ref{Euler-limit.intro}) is that it is not clear whether the operator $u\longmapsto [M_u,I_{N-2\lambda}]\rho$ is bounded from some Banach space to itself. Note that it is a key property to be checked in order for the classical existence techniques involving fixed-point theorems to work. We will show that this is the case for some well chosen Banach space of function of Lipschitz type enjoying some summability properties. Namely, such normed space will be denoted by $L^1(0,T;W^{1,kp_1,kp_2}(\RR^N,\RR^N))$, where
\begin{equation}\label{normed.space.W.form}
W^{k,p,q}(\RR^N):=(W^{k-1,p}\cap W^{k-1,q}\cap W^{k,\infty})(\RR^N),
\end{equation}
that is a Banach space when endowed with the norm
$$\Vert f\Vert_{W^{k,p,q}(\RR^N)}:=\Vert f\Vert_{W^{k-1,p}(\RR^N)}+\Vert f\Vert_{W^{k-1,q}(\RR^N)}+\Vert f\Vert_{W^{k,\infty}(\RR^N)}.$$
Such commutator estimates constitute the cornerstone to construct our strong solutions to (\ref{Euler-limit.intro}) via fixed point arguments. This will be the purpose of our second main result in this paper.

\begin{theo}\label{Euler.limit.eq.existence.intro.theo}
Let $\lambda$ be any exponent in $(0,N/2)$ and $1\leq p_1<p_2\leq \infty$ such that
$$\frac{1}{p_2}<1-\frac{2\lambda}{N}<\frac{1}{p_1}.$$
Consider any (large enough) positive integer $k$ so that
$$k>\max\left\{\frac{N}{2\lambda}-1,\frac{N}{2\lambda p_1}\right\}.$$
Then, for any positive radius $R$ there exists some positive constant $\delta_R$ depending on $R$ such that the system (\ref{Euler-limit.intro}) admits one, and only one solution $u\in L^1(0,T;W^{1,kp_1,kp_2}(\RR^N,\RR^N))$, $\rho\in L^\infty(0,T;L^{p_1}(\RR^N)\cap L^{p_2}(\RR^N))$ with $\Vert u\Vert_{L^1(0,T;W^{1,kp_1,kp_2}(\RR^N,\RR^N))}\leq R$, as long as the external force $F=-\nabla\psi$ and the initial datum $\rho^0$ are taken so that they fulfil the next conditions
$$\Vert \nabla\psi\Vert_{L^1(0,T;W^{1,kp_1,kp_2}(\RR^N,\RR^N))}<R\ \mbox{ and }\ \Vert \rho^0\Vert_{W^{k,p_1,p_2}(\RR^N)}<\delta_R.$$
\end{theo}

Such result then entails that the friction term tends to slow down the particles up to a state of absence of any motion unless some force $F=-\nabla \psi$ (inducing a shear flow in the hydrodynamic description) is assumed to act on the system. On the one hand, the velocity friction is natural in the sense that it is imposed by the conditions of the medium. On the other hand, the presence of a force $F$ can also be justified in certain biological settings like in the \textit{soaring flight} of some flocks of big birds like pelicans or seagulls. Soaring birds avoids flapping in the presence of a thermal, since they can use wind currents to self-propel. See \cite[Chapter 5]{Norberg} for a comprehensive classification of the kinds and mechanisms of soaring flights along with some pictures (e.g., slope, thermal, gust, frontal, wave and dynamic soaring). Another reference where such biological feature is considered (i.e., some kind of ``external'' force is assumed to propel birds) is the coupled macroscopic--mesoscopic system proposed in \cite{CarrilloChoiKarper} consisting in a Vlasov--Fokker--Planck equation with local alignment effects describing the motion of birds in the core of a shear flow associated with a fluid obeying the Navier--Stokes equation of Fluid Mechanics. See also \cite{BaeChoiHaKang} for a similar analysis of global existence of strong solutions of such coupled systems.

The paper is organized as follows. Section \ref{Section-FlockingModels} is devoted to introduce the model we are interested in at the different scales of description that have been studied in the literature, namely, microscopic, mesoscopic and macroscopic. Along this paper we will focus on the relationship between the mesoscopic and macroscopic ones via hydrodynamic limits. This is the content of Section \ref{Section-Hyperbolic} and \ref{Section-OtherLimits}, where we will focus on different type of scalings where the rigorous convergence in the hydrodynamic limit can be performed. Specifically, in Section \ref{Section-Hyperbolic} we will concentrate on a hyperbolic and singular limit of the kinetic Cucker--Smale model whilst Section \ref{Section-OtherLimits} shows that such techniques can be adapted to other relevant hydrodynamic limits of intermediate type. We will also address the frictionless case along with a hydrodynamic limit linking the Rayleigh--Helmholtz friction with the standard linear friction. Section \ref{Section-AnalisisLimit} presents the existence and uniqueness of global in time solutions of higher regularity for the limiting macrocopic systems arising from the preceding hydrodynamic and singular limits. To conclude, we derive in Appendix \ref{Appendix-Scaling} the preceding scalings of the kinetic equation through a nondimensionalization procedure and an appropriate choice of the dimensionless parameters. 

\section{Different scales of description in flocking dynamics models}\label{Section-FlockingModels}

Once we have introduced the main targets in this paper, let us sketch a brief overview about the main models of flocking dynamics and the basic ideas underlying the kinetic Cucker--Smale model. From a biological outlook, it has been deeply studied in the context of starlings. For instance, \cite{HildenbrandtCarereHemelrijk} shows a study of the social behavior of tens of thousands of European Starlings (\textit{Sturnus vulgaris}) in Roma, the rules that guide them into an ordered motion culminating with the formation of a flock, which propagates and modifies its density and shape through the amazing aerial maneuvers taking place in winter in the period before migration. Such study, focuses on the internal natural laws and on explaining how can they lead to the formation of convoluted \textit{patterns} in flocking. Computer simulation of complex systems is also a key ingredient to test whether the basic rules of organization yields the observed patterns. The first computer simulation of flocking dynamics was pioneered by Reynodls \cite{Reynolds} with his simulation program \textit{Boids}. 

The original model of flocking proposed by Reynolds relied in three different basic rules: \textit{separation}, \textit{alignment} and \textit{cohesion}. Roughly speaking, there exists some type of repulsive short range interaction (separation) that prevents two individuals in the crowd from getting too close. In the same way, some sort of attractive long range interaction (cohesion) must exist in order for the population to form clusters that eventually flock. Finally, an intermediate third kind of interaction (alignment), which is the actual mechanism governing the flocking behavior, should exist. Such interaction may be described in terms of a set of rules steering towards the average heading of neighbors. 

\subsection{Microscopic systems}

Based on such idea, Cucker and Smale designed a model by means of a system of ODE for $n$ interacting particles under the influence of self-alignment effects. Specifically, consider a collection of $n$ birds with identical masses $m$ labelled by its positions $x_1(t),\ldots, x_n(t)\in\RR^3$ and its velocities $v_1(t),\ldots,v_n(t)\in\RR^3$. The idea is that each bird should modify its velocity to mimic some average relative velocity of the remaining $n-1$ birds. For instance, the weights may depend decreasingly on the distance between each bird (the closer the birds, the stronger the influence). Then, we are led to the following kind of $n$-particles model
\begin{equation}\label{CuckerSmale_model.eq}
\left\{
\begin{array}{ll}
\hspace{0.35cm}\displaystyle \frac{d x_i}{dt}=v_i, & t\geq 0,\\
\displaystyle m\,\frac{d v_i}{dt}=\frac{1}{n}\sum_{j=1}^n\phi(\vert x_i-x_j\vert) (v_j-v_i), & t\geq 0,
\end{array}
\right.
\end{equation}
for every index $i\in\{1,\ldots,n\}$. Here, $\phi=\phi(r)$ is called the interaction kernel or \textit{influence function} and it is usually normalized as follows
$$\phi(r)=Ka(r),$$
for some $a=a(r)\in (0,1]$ such that $a(0)=1$ (finite maximum influence) and some positive constant $K$ standing for the strength of the influence. Here $a(r)$ should be assumed a decreasing function of the distance $r$. The classical choice of the weight function $a(r)$ is devoted to Cucker and Smale and reads
\begin{equation}\label{Weight_influence_function.form}
a(r)=\frac{1}{(1+r^2)^\lambda},\ \ r\geq 0,
\end{equation}
for some positive exponent $\lambda$ measuring the asymptotic fall-off of the long range interactions. One of the crucial features of the Cucker--Smale model (\ref{CuckerSmale_model.eq}) lies in the symmetry of the interactions. This may be interpreted as the absence of leaders in the flock: all the birds enjoy the same role and are influenced by the remaining ones in exactly the same manner. Regarding the analysis of (\ref{CuckerSmale_model.eq}) (see e.g., \cite{HaTadmor}), the presence of such symmetric interactions leads to some conserved and controlled quantities such as the linear momentum of the $n$ agents (which remains constant), the center of mass (which propagates at constant velocity) or the global kinetic energy (which is monotonically decreasing but lower bounded away from zero). Notice that in this model, neither the separation nor the cohesion dynamic rules are obvious. The only clues supporting such laws are the fall-off of the influence functions at large distances and the finiteness of the interaction kernel at $r=0$. However, there is some analysis supporting the effectiveness of the alignment effects in the particular case of the Cucker--Smale choice (\ref{Weight_influence_function.form}) for $a$. Namely, when the fall-off of the tail of the influence function is sufficiently slow in the sense that
$$\int^{\infty}a(r)\,dr=\infty,$$
then, the conservation of linear momentum and the decrease of kinetic energy entail ``unconditional flocking''. In the particular case of the Cucker--Smale influence function, it amounts to choosing $\lambda\in (0,1/2)$. The terms \textit{unconditional flocking} mean that, independently on the choice of the initial configurations $x_1(0),\ldots,x_n(0)$ and $v_1(0),\ldots,v_n(0)$, the maximum distance between any couple of agents is uniformly bounded at all times, and the velocities align towards a common velocity at infinity, i.e.,
$$\sup_{t\geq 0}\max_{i\neq j}\vert x_i(t)-x_j(t)\vert<+\infty\ \mbox{ and }\ \lim_{t\rightarrow +\infty}\max_{i\neq j}\vert v_i(t)-v_j(t)\vert=0.$$
Indeed, the Cucker--Smale case is much simpler since all the velocities align towards the conserved mean velocity, i.e.,
$$\frac{1}{n}\sum_{i=1}^n v_i(t)=\frac{1}{n}\sum_{i=1}^n v_i(0),$$
which can be easily prescribed through the initial configuration. The flocking dynamics is called \textit{conditional flocking} when the above two properties only hold for some specific well chosen initial configurations. This is the case when $\lambda>1/2$ (see \cite{CuckerSmale2} for some specific profiles).  The particular case $\lambda=1/2$ is dominated again by unconditional flocking although it has to be studied separately through a bootstrap argument.

Note that the above model (\ref{CuckerSmale_model.eq}) assumes symmetry of the interactions. Nevertheless, asymmetric interactions in populations with leaders is a very common situation in Biology that cannot be explained by the classical Cucker--Smale model. Furthermore, the presence of symmetry in the interactions entails the next characteristics of the model. First, the final bulk velocity of the population does not depend on the evolution, but only on the initial configuration of the $n$ birds since it agrees with the mean velocity, which remains constant. Second, the changes on the velocity of each agent are controlled by the (weighted) mean relative velocity of all the remaining agents. Consequently, two distinguished groups in the same population $G_1, G_2$ separated by a large enough distance ($G_2$ much crowded than $G_1$) would evolve as follows: the presence of the large group $G_2$, although far away from the smaller one, would almost halt the dynamic of the small group $G_1$.

A possible modification of the model to avoid the above features was proposed in \cite{MotschTadmor}. Roughly speaking, the gist is to normalized the pairwise interactions $\phi(\vert x_i-x_j\vert)$ between agents in terms of a relative influence, involving a loss of symmetry. This gives rise to the \textit{Motsch--Tadmor} $n$-particles system:
\begin{equation}\label{MotschTadmor_model.eq}
\left\{
\begin{array}{ll}
\hspace{0.35cm}\displaystyle \frac{dx_i}{dt}=v_i & t\geq 0,\\
\displaystyle m\,\frac{dv_i}{dt}=\frac{K}{\sum_{k=1}^N a(\vert x_i-x_k\vert)} \sum_{j=1}^N a(\vert x_i-x_j\vert)(v_j-v_i), & t\geq 0,
\end{array}
\right.
\end{equation}
Notice that when the agents are concentrated near the same area, then (\ref{MotschTadmor_model.eq}) essentially describes the same dynamics as (\ref{CuckerSmale_model.eq}). On the contrary, the above issue with the spatially inhomogeneous configuration consisting of two distinguished groups $G_1,G_2$ has being corrected since the dynamics of the agents in $G_1$ is dominated by the remaining agents in $G_1$, not $G_2$ (see \cite{MotschTadmor} for a detailed discussion). Despite the loss of the above-mentioned conserved quantities, it is possible to do some analysis yet. Indeed, unconditional flocking in the above Motsch--Tadmor model (\ref{MotschTadmor_model.eq}) is granted as long as the fall-off of the influence function $a(r)$ is slow enough. Specifically,
$$\int^{\infty}a(r)^2\,dr=\infty,$$
which is a bit more restrictive than in the Cucker--Smale case. Furthermore, the bulk velocity $v_\infty$ cannot be recovered just from the mean velocity (now it is not constant), but it depend on the evolution itself. The same properties hold for more general asymmetric models with leaders as it is analyzed in \cite{MotschTadmor}.

\subsection{Mesoscopic systems}

Apart from the above microscopic systems of $n$ particles, the mesoscopic description for $n$ large enough has been studied by means of kinetic-type equation arising from Physical Statistics. Note that when $n$ is high, the $n$-particles systems (\ref{CuckerSmale_model.eq}) and (\ref{MotschTadmor_model.eq}) become unmanageable both from the theoretical and numerical points of view and there is some need for simplified models which aim to preserve the above dynamical properties such as the flocking formation. Let us sketch the classical formal procedure based on the BBGKY hierarchy of particle distributions to achieve mesoscopic systems from microscale descriptions (see \cite{HaTadmor,MotschTadmor}) either in a low density or low interactions regimes. Consider the $n$-particles distribution function $f^n(t,x_1,v_1,\ldots,x_n,v_n),$
which according to the \textit{BBGKY hierarchy} verifies the next Liouville-type PDE
$$\frac{\partial f^n}{\partial t}+\sum_{i=1}^n v_i\cdot \nabla_{x_i}f^n+\frac{K}{n}\sum_{i=1}^n\divop_{v_i}\left(\sum_{j=1}^n a(\vert x_i-x_j\vert)(v_j-v_i)f^n\right)=0.$$
Assume symmetry of $f^n$ with respect to particle's positions and velocities $(x_i,v_i)$.
Considering the martingale distribution function for one particle, one can obtain its evolution equation by integrating with respect to the remaining $n-1$ particles and closing it in terms of the one-point martingale distribution itself and the the two-point martingale distribution function in the interaction term. Then, one formally pass to the limit $n\rightarrow +\infty$ and assumes \textit{molecular chaos} (i.e., independence of the two-point particles distribution, that can be physically attainable both in the low density and low interaction regimes). Consequently, the one-particle distribution function can be shown to verify the \textit{kinetic Cucker--Smale} model or \textit{Ha--Tadmor} model
\begin{equation}\label{HaTadmor_3_dimensions.eq}
\frac{\partial f}{\partial t}+v\cdot \nabla_xf=\divop_v\left(Q_{CS}(f,f)\right), \hspace{0.25cm} (t,x,v)\in \RR_0^+\times \RR^3\times\RR^3,
\end{equation}
where the integral operator $Q_{CS}$ is known as the alignment operator
$$Q_{CS}(f,f)(t,x,v):=f(t,x,v)\int_{\RR^3}\int_{\RR^3}\phi(\vert x-y\vert)(v-w)f(t,y,w)\,dy\,dw.$$
The same formal reasoning yields the \textit{kinetic Motsch--Tadmor} model with $Q_{CS}$ replaced with the Motsch--Tadmor alignment operator
\begin{equation}\label{MostchTadmor.alignment.eq}
Q_{MT}(f,f)(t,x,v)=f(t,x,v)\frac{K\int_{\RR^3}\int_{\RR^3}a(\vert x-y\vert)(v-w)f(t,y,w)\,dy\,dw}{\int_{\RR^3}\int_{\RR^3}a(\vert x-y\vert)f(t,y,w)\,dy\,dw}.
\end{equation}
To our best knowledge, the above mean field limits when $n\rightarrow +\infty$ have only been rigorously established in the Cucker--Smale case in \cite{HaLiu} (regular influence functions) and \cite{MuchaPeszek} (singular influence functions). As far as we know, the mean field limit of the Motsch--Tadmor individual based system remains open as depicted in \cite{MotschTadmor} because asymmetric interactions are involved in such system.

A similar analysis to that of the discrete system (\ref{CuckerSmale_model.eq}) involving symmetry of the interactions once more shows that we also expect unconditional flocking behavior in the kinetic model. First, \cite{HaTadmor} shows the existence of classical solutions of the kinetic Ha--Tadmor model with compactly supported initial data. The corresponding existence theory for measure-valued solutions of the Cauchy problem associated with (\ref{HaTadmor_3_dimensions.eq}) was developed in \cite{HaLiu}. Compared with the discrete system. the flocking analysis again relies on the study of the fluctuations of kinetic energy, i.e.,
$$\Lambda(f)(t):=\int_{\RR^N}\int_{\RR^N}\vert v-u_c\vert^2f(t,x,v)\,dx\,dv,$$
where $u_c(t)=u_c(0)$ stands for the mean velocity of the distribution of particles, which agrees with the velocity of the center of mass. The main idea is to show that the Lyapunov functional $\Lambda(f)(t)$ vanishes exponentially fast when $t\rightarrow +\infty$ under the assumption $\lambda\in (0,1/4]$. This was done through some suboptimal estimates of the propagation of the support that led to such apparent difference with the dynamic behavior of the discrete model (\ref{CuckerSmale_model.eq}), see \cite{HaLiu,HaTadmor}. Nevertheless, such estimates for the propagation of the support were sharped in \cite{Carrillo-Improve} in order to show that exponential unconditional flocking in the above sense actually holds for each exponent $\lambda$ in the whole range $(0,1/2]$.

In what follows, our mesoscopic models will include some other effects such as the presence of a random white noise and a linear friction in the $v$ variables. In addition, as we showed in the Section \ref{Introduction.section}, one may want our system to be exposed to some external (or self-generated) force $F=-\nabla\psi$ for some potential $\psi=\psi(t,x)$. As it is well known, the evolution of the microscopic system of $n$ interacting particles under the above assumptions is governed by the \textit{Langevin equations}:
\begin{equation}\label{CuckerSmale_model.noise.friction.eq}
\left\{\begin{array}{ll}
\displaystyle\hspace{0.26cm}\frac{dx_i}{dt}=v_i, & t\geq 0,\\
\displaystyle m\frac{dv_i}{dt}=-\mu v_i+\frac{1}{n}\sum_{j=1}^n\phi(\vert x_i-x_j\vert)(v_j-v_i)-\nabla_x\psi(t,x_i)+\sqrt{2D}\xi_i(t), & t\geq 0,
\end{array}\right.
\end{equation}
where the stochastic forces $\xi_i(t)$ stand for $\delta$-correlated Gaussian probability distributions
$$\left<\xi_i(t_1)\xi_j(t_2)\right>=\delta_{ij}\delta(t_1-t_2).$$
A similar formal procedure based on the BBGKY formalism amounts to the following \textit{kinetic Cucker--Smale system with Gaussian noise, friction and external forces}
\begin{equation}\label{kinetic.CuckerSmale.frictiondiffusion.eq}
\frac{\partial f}{\partial t}+v\cdot\nabla_x f-\nabla_x\psi\cdot \nabla_v f=\divop_v\left(vf+\nabla_v f+Q_{CS}(f,f)\right).
\end{equation}
where we have intentionally avoided physical units issues for the sake of simplicity. See Appendix \ref{Appendix-Scaling} for a dimensional analysis of the model and the rigorous derivation of the hyperbolic and intermediate scalings that we will be concerned with. The analysis of existence for solutions to this system can be deduced, for instance, from the combination of results given in \cite{B,CS,HaTadmor}.

\subsection{Macroscopic systems}
Before ending this section, let us briefly introduce the above-mentioned macroscopic systems \cite{CarrilloChoiKarper,KangVasseur,KarperMelletTrivisa} of flocking dynamics that have been derived in the literature by means of rigorous hydrodynamic limits. In addition, we will sketch the main results concerning the flocking behavior at large times in some of those models and will compare the kind of scaling proposed by the preceding authors with those considered in this paper and those in \cite{BellouquidCalvoNietoSoler-Intermedio,NietoPoupaudSoler-Hiperbolico,PoupaudSoler-Parabolico} for the Vlasov--Poisson--Fokker--Planck system.

Recall that the choice of a Dirac delta in the Mostch--Tadmor alignment operator (\ref{MostchTadmor.alignment.eq}) leads to what authors call \textit{strong local alignment}, that might also be understood as a relaxation term towards the mean velocity field, i.e., a linear friction term centered at $u$. Specifically,
$$Q_{MT}(f,f)\rightarrow Kf(v-u)\ \mbox{ when }\ a\rightarrow \delta_0$$ where $u$ is the mean velocity field, that can be recovered from the current of particles $j$ via the relation $j=\rho\,u$. See \cite{KarperMelletTrivisa3} for a rigorous derivation of the above property. One of the main features of such term is that despite modelling some kind of alignment in the mesoscopic system, by definition it does provides no effect on the first order moment equation by virtue of its cancellations. Specifically,
$$\int_{\RR^N}v \divop_v(f(v-u))\,dv=-\int_{\RR^N}f(v-u)\,dv=j-\rho\,u=0.$$
Apart form such cancellation in the first order moment equation, it also provides certain structure on the system that helps on closing an equation for the macroscopic entropy functionals and on obtaining closed relative entropy bounds which ensures the strong passage to the limit $\varepsilon\rightarrow 0$ (see the aforementioned references). All the above references include such strong local alignment term at a high order of the parameter $\varepsilon$. 

First, \cite{KarperMelletTrivisa} focuses on the next scaling
$$\varepsilon\frac{\partial f_\varepsilon}{\partial t}+\varepsilon v\cdot\nabla_x f_\varepsilon=\varepsilon\divop_v Q_{CS}^\phi(f_\varepsilon,f_\varepsilon)+\divop_v(\nabla_v f_\varepsilon+f_\varepsilon(v-u_\varepsilon)).$$
Note that it is a regime with high noise and local alignment but low nonlocal alignment of Cucker--Smale type. In contrast with our scaling, the influence function $\phi$ is not singular, but bounded near $r=0$. By virtue of the above property of the strong local alignment, the authors are led to the next limiting system when $\varepsilon\rightarrow 0$
$$\left\{
\begin{array}{l}
\displaystyle \frac{\partial \rho}{\partial t}+\divop(\rho u)=0,\\
\displaystyle\frac{\partial}{\partial t}(\rho u)+\divop(\rho u\otimes u)+\nabla\rho=-\int_{\RR^N}\phi(\vert x-y\vert)(u(t,x)-u(t,y))\rho(t,x)\rho(t,y)\,dy.
\end{array}
\right.$$
Note that such models maintain nonlocal alignment effects but does not include neither a friction term nor a strong alignment one. Second, \cite{KangVasseur} neglects noise effects but considers both linear friction and strong local alignment, i.e.,
$$\varepsilon\frac{\partial f_\varepsilon}{\partial t}+\varepsilon v\cdot\nabla_x f_\varepsilon=\varepsilon\divop(f_\varepsilon v)+\divop_v\left(f_\varepsilon(v-u_\varepsilon)\right).$$
In this case, the limiting system takes the form
$$\left\{
\begin{array}{l}
\displaystyle \frac{\partial \rho}{\partial t}+\divop(\rho u)=0,\\
\displaystyle\frac{\partial}{\partial t}(\rho u)+\divop(\rho u\otimes u)=-u
\end{array}
\right.$$
Hence, although strong local alignment is lost in the macroscopic system, a linear friction term has been recovered in the limit. Finally, let us consider the coupled system in \cite{CarrilloChoiKarper} given by
$$\frac{\partial f_\varepsilon}{\partial t}+v\cdot \nabla_x f_\varepsilon=\varepsilon\divop_v(f_\varepsilon(v-U_\varepsilon))+\divop_v(\nabla_v f_\varepsilon+f_\varepsilon(v-u_\varepsilon)).$$
Recall that $u_\varepsilon$ stands for the macroscopic velocity field of the particles distributions $f_\varepsilon$. Here, $U_\varepsilon$ is the velocity field of a fluid, where the particles move, which evolves according to the incompressible Navier--Stokes system, i.e.,
$$\left\{
\begin{array}{l}
\displaystyle\frac{\partial U_\varepsilon}{\partial t}+(U_\varepsilon\cdot \nabla)U_\varepsilon=-\nabla p_\varepsilon+\nu\Delta U_\varepsilon+\int_{\RR^N}f_\varepsilon(v-U_\varepsilon)\,dv,\\
\displaystyle\divop U_\varepsilon=0,
\end{array}
\right.$$
where $\nu\geq 0$ is the viscosity and $p_\varepsilon$ stands for the pressure of the fluid. In such paper, an entropy method in the spirit of \cite{KarperMelletTrivisa} was derived to pass to the limit $\varepsilon\rightarrow 0$ and the following limiting macroscopic system was obtained
$$\left\{
\begin{array}{l}
\displaystyle \frac{\partial \rho}{\partial t}+\divop(\rho u)=0,\\
\displaystyle\frac{\partial}{\partial t}(\rho u)+\divop(\rho u\otimes u)+\nabla\rho=\rho(U-u),
\end{array}
\right.$$
coupled with the limiting Navier--Stokes system
$$\left\{
\begin{array}{l}
\displaystyle\frac{\partial U}{\partial t}+(U\cdot \nabla)U=-\nabla p+\nu\Delta U+\rho(u-U),\\
\displaystyle\divop U=0,
\end{array}
\right.$$
As it is apparent now, the crucial differences between the above three hydrodynamic models and those we propose in this paper are the presence of inertial effects and the absence of singularity in the influence function at $r=0$. Our system is comparable to those in \cite{BellouquidCalvoNietoSoler-Intermedio,NietoPoupaudSoler-Hiperbolico,PoupaudSoler-Parabolico,GNPS}, where the inertial effects where also lost in the macroscopic limit. Specifically, the kind of limiting systems for the Vlasov--Poisson--Fokker--Planck system takes the form
$$\left\{
\begin{array}{l}
\displaystyle\frac{\partial \rho}{\partial t}+\divop(\rho u)=0,\\
\displaystyle u=-\nabla\psi,
\end{array}
\right.$$
where the potential $\psi=\psi(t,x)$ can be recovered from the density $\rho$ through the Poisson equation
$$\Delta\psi=\theta \rho$$
and $\theta =1$ o $\theta =-1$ depending on the attractive of repulsive character of the Newtonian interactions. In the parabolic case, the limiting system changes the velocity field from $u=-\nabla\psi$ to $u=-\nabla \psi+\frac{\nabla\rho}{\rho}$, that includes viscosity on the continuity equation for $\rho$. It stands to reason that the main difficulty on analyzing our limiting system (\ref{Euler-limit.intro}) will be that the velocity field $u$ cannot be explicitly written in terms of the density $\rho$, but it solves a commutator-type implicit equation.

In \cite{TadmorTan} the pressure-less Euler system is analyzed. As it happens in the microscopic and mesoscopic descriptions, any strong solution must flock. The analysis of global existence of solution is there analyzed for such model in one and two dimensions. In particular, global regularity for subcritical initial data takes place. Indeed, the authors show the existence of critical thresholds in the phase space of the initial data distinguishing between a global regularity regime and a finite-time blow-up regime. If one adds attractive or repulsive effects in the system, similar critical thresholds appear as depicted in \cite{CarrilloChoiTadmorTan} in the $1D$ case.

\section{Hyperbolic and singular hydrodynamic limit}\label{Section-Hyperbolic}
In this section we focus on the hyperbolic scaling for the kinetic Cucker--Smale model with friction and diffusion effects in the velocity variable, see  (\ref{VPFPCS-dimensionless-hyperbolicscaling.Intro.form}), where the regular influence functions $\phi_\varepsilon$ converge towards the singular one $\phi_0$ (see  (\ref{influencefunction.singular.form}) and (\ref{influencefunction.regular.form})). Such scaling is obtained in Appendix \ref{Appendix-Scaling}, where a study of the physical constants of the model and a dimensionless analysis amounts to consider the system
\begin{equation}\label{VPFPCS-dimensionless-hyperbolicscaling.form}
\varepsilon\frac{\partial f_\varepsilon}{\partial t}+\varepsilon v\cdot\nabla_x f_\varepsilon-\nabla_x\psi_\varepsilon\cdot \nabla_v f_\varepsilon=\divop_v\left(f_\varepsilon v+\nabla_v f_\varepsilon+Q_{CS}^{\phi_\varepsilon}(f_\varepsilon,f_\varepsilon)\right).
\end{equation}
under appropriate assumptions on the scaling of the scaled mean thermal velocity and mean free path, along with the scaled effective range and strength of the interactions. See Appendix \ref{Appendix-Scaling} for the details.

Although we will be interested in (\ref{VPFPCS-dimensionless-hyperbolicscaling.form}) most of the time, we will  show that our technique also yields some insight into the understanding of the rigorous hydrodynamic limit in an intermediate hyperbolic scaling for the frictional case and a hyperbolic scaling in the frictionless case that have been introduced in the Appendix \ref{Appendix-Scaling}. This will be done in Section \ref{Section-OtherLimits}. 

\subsection{Hierarchy of moments and limiting distribution}
Before going into the heart of the matter, let us introduce some notation that will be used along the paper and show which is the expected limiting distribution along with the first equations of the hierarchy of velocity moments. We will focus on the next ones
\begin{align*}
\mbox{Density:}&\hspace{0.5cm}\rho_\varepsilon(t,x):=\int_{\RR^N}f_\varepsilon(t,x,v)\,dv,\\
\mbox{Current:}&\hspace{0.5cm}j_\varepsilon(t,x):=\int_{\RR^N}v\,f_\varepsilon(t,x,v)\,dv,\\
\mbox{Velocity field:}&\hspace{0.5cm}u_\varepsilon(t,x):=\frac{j_\varepsilon(t,x)}{\rho_\varepsilon(t,x)},\\
\mbox{Stress tensor:}&\hspace{0.5cm}\mathcal{S}_\varepsilon(t,x):=\int_{\RR^N}v\otimes v\, f_\varepsilon(t,x,v)\,dv,\\
\mbox{Stress flux tensor:}&\hspace{0.5cm}\mathcal{T}_\varepsilon(t,x):=\int_{\RR^N}(v\otimes v)\otimes v f_\varepsilon(t,x,v)\,dv,\\
\end{align*}
First of all, let us note that the average kinetic energy consists of the macroscopic kinetic energy together with the internal energy and it can be obtained as a contraction of the stress tensor. Specifically, take the trace of the stress tensor to obtain
\begin{align*}
E_\varepsilon(t,x):=\frac{1}{2}\trace(\mathcal{S}_\varepsilon(t,x))&=\frac{1}{2}\int_{\RR^N}\vert v\vert^2 f_\varepsilon(t,x,v)\,dv\\
&=\frac{1}{2}\rho_\varepsilon +\frac{1}{2}\int_{\RR^N}\vert v-u_\varepsilon(t,x)\vert^2 f_\varepsilon(t,x,v)\,dv\\
&=: E^{kin}_\varepsilon(t,x)+E^{int}_\varepsilon(t,x).
\end{align*}
Another classical kinetic quantity is the well known energy flux that can be obtained as a new contraction of the stress flux tensor. Namely,
$$Q_\varepsilon(t,x):=\frac{1}{2}\trace(\mathcal{T}_\varepsilon(t,x))=\frac{1}{2}\int_{\RR^N}v\vert v\vert^2 f_\varepsilon(t,x,v)\,dv.$$
Note that (\ref{VPFPCS-dimensionless-hyperbolicscaling.form}) can be restated as
\begin{equation}\label{VPFPCS-dimensionless-hyperbolicscaling.restated.form}
\varepsilon\frac{\partial f_\varepsilon}{\partial t}+\varepsilon v\cdot\nabla_x f_\varepsilon-\nabla_x\psi_\varepsilon\cdot \nabla_v f_\varepsilon=\divop_v\left(f_\varepsilon v+\varepsilon\nabla_v f_\varepsilon+f_\varepsilon \phi_\varepsilon*\rho_\varepsilon v-f_\varepsilon\phi_\varepsilon*j_\varepsilon\right)
\end{equation}
Here on we will assume that the initial data $f_\varepsilon^0$ satisfy the next hypothesis
\begin{equation}\tag{$H_1$}\label{hypothesis.initial.data.f.form}
\left\{
\begin{array}{l}
\displaystyle f_\varepsilon^0=f_\varepsilon^0(x,v)\geq 0\ \mbox{ and }\ f_\varepsilon^0\in C^\infty_c(\RR^N\times \RR^N),\\
\displaystyle \Vert f_\varepsilon^0\Vert_{L^1(\RR^{2N})}\leq M_0\ \mbox{ and } \rho_\varepsilon^0 \overset{*}{\rightharpoonup} \rho^0\mbox{ in }\mathcal{M}(\RR^N),\\
\displaystyle \Vert \vert v\vert^2 f_\varepsilon^0\Vert_{L^1(\RR^2N)}\leq E_0.
\end{array}
\right.
\end{equation}
for every $\varepsilon>0$, where $M_0,E_0>0$ are $\varepsilon$-independent constants. Regarding the external forces $-\nabla\psi_\varepsilon$, we will assume that 
\begin{equation}\tag{$H_2$}\label{hypothesis.initial.data.F.form}
\left\{
\begin{array}{l}
\displaystyle \psi_\varepsilon\in L^2(0,T;W^{1,\infty}(\RR^N))\mbox{ and }\psi\in L^2(0,T;W^{1,\infty}(\RR^N)),\\
\displaystyle \nabla \psi_\varepsilon(t,\cdot)\in C_0(\RR^N,\RR^N))\mbox{ and }\nabla\psi(t,\cdot)\in C_0(\RR^N,\RR^N),\ \mbox{a.e. }t\in [0,T),\\
\Vert \nabla\psi_\varepsilon\Vert_{L^2(0,T;L^\infty(\RR^N,\RR^N))}\leq F_0\ \mbox{ and }\ \nabla\psi_\varepsilon \rightarrow \nabla\psi\ \mbox{ in }L^1(0,T;C_0(\RR^N)),
\end{array}
\right.
\end{equation}
for every $\varepsilon>0$ and some $\varepsilon$-independent $F_0>0$. Under such assumptions (see \cite{B,CS,HaTadmor}) there exists a unique global in time strong solution to (\ref{VPFPCS-dimensionless-hyperbolicscaling.restated.form}) such that $f_\varepsilon(0,\cdot,\cdot)=f_\varepsilon^0$ for every $\varepsilon>0$. Respectively multiplying Equation (\ref{VPFPCS-dimensionless-hyperbolicscaling.restated.form}) by $1$, $v$ and $v\otimes v$ (Kronecker product) and integrating with respect to $v$ by parts one can obtain the next hierarchy of moment equations by virtue of the fall-off at infinity of the $\varepsilon$-approximate solutions $f_\varepsilon(t,x,v)$:
\subsubsection*{Mass conservation}
\begin{equation}\label{MomentumEquation-mass.form}
\frac{\partial \rho_\varepsilon}{\partial t}+\divop_x j_\varepsilon=0.
\end{equation}
\subsubsection*{Current balance}
\begin{equation}\label{MomentumEquation-current.form}
\varepsilon\frac{\partial j_\varepsilon}{\partial t}+\varepsilon\divop_x \mathcal{S}_\varepsilon+\rho_\varepsilon\nabla_x\psi_\varepsilon+\left(1+\phi_\varepsilon*\rho_\varepsilon\right)j_\varepsilon-\left(\phi_\varepsilon*j_\varepsilon\right)\rho_\varepsilon=0.
\end{equation}
\subsubsection*{Stress tensor balance}
\begin{equation}\label{MomentumEquation-energy.form}
\varepsilon\frac{\partial \mathcal{S_\varepsilon}}{\partial t}+\varepsilon\divop_x\mathcal{T}_\varepsilon+2\Sym(j_\varepsilon\otimes \nabla_x\psi_\varepsilon)+2\left(\left(1+\phi_\varepsilon*\rho_\varepsilon\right)\mathcal{S}_\varepsilon-\rho_\varepsilon I\right)-2\Sym((\phi_\varepsilon*j_\varepsilon)\otimes j_\varepsilon)=0.
\end{equation}
Here, $\Sym(M)$ stands for the symmetric part of a square matrix $M$, i.e., $\Sym(M):=\frac{1}{2}(M+M^\mathsf{T})$. Note that the spectral norm of the stress tensor can be entirely controlled by the average kinetic energy
$$\vert \mathcal{S}_\varepsilon(t,x)\vert\leq 2 E_\varepsilon(t,x),\ \ t\geq 0,\,x\in\RR^N.$$
Consequently, despite the stress tensor containing all the information about the average kinetic energy of the system of particles, one might be interested in a explicit balance law for the latter. It can be straightforwardly deduced from  (\ref{MomentumEquation-energy.form}) by taking traces
\begin{equation}\label{MomentumEquation-energy-average.form}
\varepsilon\frac{\partial E_\varepsilon}{\partial t}+\varepsilon\divop_x Q_\varepsilon+j_\varepsilon\cdot \nabla_x\psi_\varepsilon+2\left((1+\phi_\varepsilon*\rho_\varepsilon)E_\varepsilon-\frac{N}{2}\rho_\varepsilon\right)-(\phi_\varepsilon*j_\varepsilon)\cdot j_\varepsilon=0.
\end{equation}
As it often happens in most of the kinetic equations, the hierarchy of velocity moments is not a closed system. Nevertheless, as it will be shown later, our hyperbolic hydrodynamic limit will close the system (\ref{MomentumEquation-mass.form})--(\ref{MomentumEquation-current.form}) when $\varepsilon\rightarrow 0$ as stated in Equation (\ref{Euler-limit.intro}). 

\subsection{A priori bounds}
Let us show first some useful bounds for the local interaction kernel $\phi_\varepsilon$.

\begin{lem}\label{Phi_epsilon-estimates.lem}
Let us fix any nonnegative number $\varepsilon$ and any exponent $\lambda\in (0,1/2)$. Then, the next estimates hold for every $r>0$:
\begin{enumerate}
\item $\displaystyle\phi_\varepsilon(r)\leq \frac{1}{c_\lambda^\lambda r^{2\lambda}}$,
\item $\displaystyle\phi_\varepsilon(r)\leq \frac{1}{\varepsilon^{2\lambda}}$,
\item $\displaystyle\vert \phi_\varepsilon(r)-\phi_0(r)\vert\leq C_\lambda \frac{\varepsilon^{1-2\lambda}}{r}$,
\end{enumerate}
where $C_\lambda>0$ is some constant depending on $\lambda$ but not depending neither on $\varepsilon$ nor in $r$.
\end{lem}
\begin{proof}
The first two estimates follows from the definition of $\phi_\varepsilon$. Then, let us focus on the last property. The fundamental lemma of calculus leads to the next expression
$$\phi_\varepsilon(r)-\phi_0(r)=\int_0^\varepsilon\frac{d}{d\theta}\left[\frac{1}{(\theta^2+c_\lambda r^2)^\lambda}\right]\,d\theta=-\lambda\int_0^\varepsilon\frac{2\theta}{(\theta^2+c_\lambda r^2)^{\lambda+1}}\,d\theta.$$
Young's inequality for real numbers allows obtaining the next bound
$$2\theta=\frac{1}{\sqrt{c_\lambda}r}2\theta\sqrt{c_\lambda}r\leq \frac{1}{\sqrt{c_\lambda} r}(\theta^2+c_\lambda r^2).$$
Then, taking absolute values on the preceding identity we arrive at
\begin{align*}
\vert \phi_\varepsilon(r)-\phi_0(r)\vert&\leq \frac{\lambda}{\sqrt{c_\lambda}}\frac{1}{r}\int_0^\varepsilon\frac{1}{(\theta^2+c_\lambda r^2)^\lambda}\,d\theta\\
&\leq \frac{\lambda}{\sqrt{c_\lambda}}\frac{1}{r}\int_0^\varepsilon \theta^{-2\lambda}\,d\theta=\frac{\lambda}{(1-2\lambda)\sqrt{c_\lambda}}\frac{\varepsilon^{1-2\lambda}}{r}.
\end{align*}
\end{proof}
Let us know discuss some a priori bounds for the density function $\rho_\varepsilon$ and the current $j_\varepsilon$.
\begin{pro}\label{Momenta_equations.pro}
Let the initial distribution functions $f_\varepsilon^0$ verify (\ref{hypothesis.initial.data.f.form}) and the external forces $-\nabla\psi_\varepsilon$ fulfill (\ref{hypothesis.initial.data.F.form}). Consider the strong global in time solution $f_\varepsilon$ to (\ref{VPFPCS-dimensionless-hyperbolicscaling.restated.form}) with initial data $f_\varepsilon^0$ and fix any nonnegative integer $k$. Then, the $k$-th order moments in $x$ and $v$ of $f_\varepsilon$ obey the following equations,
\begin{align}
\varepsilon\frac{\partial}{\partial t}\int_{\RR^N}\int_{\RR^N}\vert v\vert^k f_\varepsilon\,dx\,dv=&-k\int_{\RR^N}\int_{\RR^N}\vert v\vert^{k-2}v\cdot\nabla_x\psi_\varepsilon f_\varepsilon\,dx\,dv\nonumber\\
&-k\int_{\RR^N}\int_{\RR^N}\vert v\vert^k f_\varepsilon\,dx\,dv+k(N+k-2)\int_{\RR^N}\int_{\RR^N}\vert v\vert^{k-2}f_\varepsilon\,dx\,dv\nonumber\\
&-k\int_{\RR^N}\int_{\RR^N}\vert v\vert^{k-1}\left((\phi_\varepsilon*\rho_\varepsilon) \vert v\vert-(\phi_\varepsilon*j_\varepsilon)\cdot\frac{v}{\vert v\vert}\right)f_\varepsilon\,dx\,dv\label{Momenta_equation_v.form}\\
\frac{\partial}{\partial t}\int_{\RR^N}\int_{\RR^N}\vert x\vert^k f_\varepsilon\,dx\,dv=&\hspace{0.35cm}k\int_{\RR^N}\int_{\RR^N}\vert x\vert^{k-2}x\cdot v f_\varepsilon\,dx\,dv\label{Momenta_equation_x.form}
\end{align}
Note that in particular the total mass remains constant, i.e.,
\begin{equation}\label{Totalmass_constant.form}
\frac{d}{dt}\int_{\RR^N} \rho_\varepsilon(t,x)\,dx=0.
\end{equation}
\end{pro}

\begin{cor}\label{Moments_bounds.cor}
Under the hypothesis in Proposition \ref{Momenta_equations.pro},
\begin{multline*}
k\left\Vert \vert v\vert^k f_\varepsilon\right\Vert_{L^1(0,T;L^1(\RR^{2N}))}\\
+\frac{k}{2}\int_0^T\int_{\RR^{4N}}(\vert v\vert^{k-2}v-\vert w\vert^{k-2}w)\cdot (v-w)\phi_\varepsilon(\vert x-y\vert)f_\varepsilon(t,x,v)f_\varepsilon(t,y,w)\,dx\,dt\,dv\,dw\,dt\\
\leq \varepsilon\Vert\vert v\vert^k f_\varepsilon(0)\Vert_{L^1(\RR^{2N})}+k(N+k-2)\Vert \vert v\vert^{k-2}f_\varepsilon\Vert_{L^1(0,T;L^1(\RR^{2N}))}+k\left\Vert \vert v\vert^{k-1}f_\varepsilon \nabla_x\psi_\varepsilon\right\Vert_{L^1(0,T;L^1(\RR^{2N}))}
\end{multline*}
\end{cor}
\begin{proof}
Integrate Equation (\ref{Momenta_equation_v.form}) with respect to time and neglect the term arising from the Barrow's law associated with the endpoint $t=T$ of the time interval.
\end{proof}

\begin{rem}\label{Moments_bounds.rem}
Note that the estimate in Corollary \ref{Moments_bounds.cor} leads to a real estimate of the $k$-the order velocity moment as long as the second term in the LHS is nonnegative. In turns, such term can be lower bounded through simple computations involving the Cauchy--Schwartz inequality as follows
\begin{multline*}
k\int_0^T\int_{\RR^{4N}}(\vert v\vert^{k-1}-\vert w\vert^{k-1})(\vert v\vert-\vert w\vert)\phi_\varepsilon(\vert x-y\vert)f_\varepsilon(t,x,v)f_\varepsilon(t,y,w)\,dx\,dy\,dv\,dw\\
\leq \frac{k}{2}\int_0^T\int_{\RR^{4N}}(\vert v\vert^{k-2}v-\vert w\vert^{k-2}w)\cdot (v-w)f_\varepsilon(t,x,v)f_\varepsilon(t,y,w)\,dx\,dy\,dv\,dw\,dt.
\end{multline*}
Note that when $k\geq 1$, then
$$(\vert v\vert^{k-1}-\vert w\vert^{k-1})(\vert v\vert-\vert w\vert)\geq 0\hspace{0.5cm}v,w\in\RR^N.$$
Consequently, for each $k\geq 1$ the $k$-th order moment can be bounded as follows
\begin{multline}\label{Momenta_equation_v.form2}
k\left\Vert \vert v\vert^k f_\varepsilon\right\Vert_{L^1(0,T;L^1(\RR^{2N}))}\\
\leq \varepsilon\Vert\vert v\vert^k f_\varepsilon(0)\Vert_{L^1(\RR^{2N})}+k(N+k-2)\Vert \vert v\vert^{k-2}f_\varepsilon\Vert_{L^1(0,T;L^1(\RR^{2N}))}+k\left\Vert \vert v\vert^{k-1}f_\varepsilon \nabla_x\psi_\varepsilon\right\Vert_{L^1(0,T;L^1(\RR^{2N}))}
\end{multline}
and the next term can be bounded similarly
\begin{multline}\label{Momenta_equation_v.form3}
\frac{k}{2}\int_0^T\int_{\RR^{4N}}(\vert v\vert^{k-2}v-\vert w\vert^{k-2}w)\cdot (v-w)\phi_\varepsilon(\vert x-y\vert)f_\varepsilon(t,x,v)f_\varepsilon(t,y,w)\,dx\,dt\,dv\,dw\,dt\\
\leq \varepsilon\Vert\vert v\vert^k f_\varepsilon(0)\Vert_{L^1(\RR^{2N})}+k(N+k-2)\Vert \vert v\vert^{k-2}f_\varepsilon\Vert_{L^1(0,T;L^1(\RR^{2N}))}+k\left\Vert \vert v\vert^{k-1}f_\varepsilon \nabla_x\psi_\varepsilon\right\Vert_{L^1(0,T;L^1(\RR^{2N}))}
\end{multline}
\end{rem}

In particular, the case $k=2$ yields the next bounds of the average kinetic energy, the current of particles and the \textit{dissipation of kinetic energy due to alignment interaction between particles}.
\begin{cor}\label{Moments_bounds_k=2.cor}
Under the hypothesis in Proposition \ref{Momenta_equations.pro},
\begin{align}
\left\Vert \vert v\vert f_\varepsilon\right\Vert_{L^2(0,T;L^1(\RR^{2N}))}&\leq M_0^{1/2}\left(\left\Vert \vert v\vert^2 f_\varepsilon\right\Vert_{L^1(0,T;L^1(\RR^{2N}))}\right)^{1/2},\label{Moment_v_firstorder.form}\\
\left\Vert \vert v\vert^2 f_\varepsilon\right\Vert_{L^1(0,T;L^1(\RR^{2N}))}&\leq 2\varepsilon E_0+\left(2NT+F_0^2\right)M_0.\label{Moment_v_secondorder.form}
\end{align}
In addition, the next estimate holds
\begin{equation}\label{Moment_v_cuadratic.form}
\int_0^T\int_{\RR^{4N}}\phi_\varepsilon(\vert x-y\vert)\vert v-w\vert^2f_\varepsilon(t,x,v)f_\varepsilon(t,y,w)\,dx\,dy\,dv\,dw\,dt\leq 2\varepsilon E_0+\left(2NT+F_0^2\right)M_0.
\end{equation} 
\end{cor}
\begin{proof}
By Corollary \ref{Moments_bounds.cor}, the next bound follows for $k=2$
\begin{align*}
2\Vert \vert v\vert^2 &f_\varepsilon\Vert_{L^1(0,T;L^1(\RR^{2N}))}+\int_0^T\int_{\RR^{4N}}\vert v-w\vert^2\phi_\varepsilon(\vert x-y\vert)f_\varepsilon(t,x,v)f_\varepsilon(t,y,w)\,dx\,dt\,dv\,dw\,dt\\
&\leq \varepsilon\Vert\vert v\vert^2 f_\varepsilon(0)\Vert_{L^1(\RR^{2N})}+2NT\Vert f_\varepsilon(0)\Vert_{L^1(\RR^{2N})}+2\left\Vert \vert v\vert f_\varepsilon \nabla_x\psi_\varepsilon\right\Vert_{L^1(0,T;L^1(\RR^{2N}))}\\
&\leq \varepsilon\Vert \vert v\vert^2 f_\varepsilon(0)\Vert_{L^1(\RR^{2N})}+2NT\Vert f_\varepsilon(0)\Vert_{L^1(\RR^{2N})}+2\Vert \vert v\vert f_\varepsilon\Vert_{L^2(0,T;L^1(\RR^{2N}))}\Vert \nabla\psi_\varepsilon\Vert_{L^2(0,T;L^\infty(\RR^{2N}))}
\end{align*}
The Cauchy--Schwartz inequality leads to
$$
\left\Vert \vert v\vert f_\varepsilon\right\Vert_{L^2(0,T;L^1(\RR^{2N}))}\leq \left(\Vert f_\varepsilon(0)\Vert_{L^1(\RR^{2N})}\right)^{1/2}\left(\left\Vert \vert v\vert^2 f_\varepsilon\right\Vert_{L^1(0,T;L^1(\RR^{2N}))}\right)^{1/2},
$$
and Young's inequality for real numbers yields
\begin{multline*}
2\left\Vert \vert v\vert f_\varepsilon\right\Vert_{L^2(0,T;L^1(\RR^{2N}))}\Vert \nabla_x\psi_\varepsilon\Vert_{L^2(0,T;L^\infty(\RR^N))}\\
\leq \Vert \nabla_x\psi_\varepsilon\Vert_{L^2(0,T;L^\infty(\RR^{2N}))}^2\Vert f_\varepsilon(0)\Vert_{L^1(\RR^{2N})}+\left\Vert\vert v\vert^2 f_\varepsilon\right\Vert_{L^1(0,T;L^1(\RR^{2N}))},
\end{multline*}
and this ends the proof.
\end{proof}

\subsection{Passing to the limit}
In this subsection, we will concentrate on the limiting procedure underlying the proof of Theorem \ref{Convergence.aproximate.solutions-intro.teo}. For the sake of completeness, let us specify the content and technical hypothesis that we will need in the proof.
\begin{theo}\label{Convergence.aproximate.solutions.teo}
Let $f_\varepsilon^0$ and $\nabla\psi_\varepsilon$ satisfy hypothesis (\ref{hypothesis.initial.data.f.form})-(\ref{hypothesis.initial.data.F.form}) and consider a sequence $f_\varepsilon$ of smooth solutions to (\ref{VPFPCS-dimensionless-hyperbolicscaling.Intro.form}) with $\lambda\in (0,1/2)$. Then, the macroscopic quantities $\rho_\varepsilon$ and $j_\varepsilon$ satisfy
\begin{align*}
\rho_\varepsilon&\rightarrow\rho, \hspace{0.25cm} \mbox{in }C([0,T],\mathcal{M}(\RR^N)-weak\,*),\\
j_\varepsilon &\overset{*}{\rightharpoonup} j, \hspace{0.25cm} \mbox{in }L^2(0,T;\mathcal{M}(\RR^N))^N.
\end{align*}
when $\varepsilon\rightarrow 0$, for some finite Radon measure $\rho$ and $j$ and some subsequence of $\{\rho_\varepsilon\}_{\varepsilon>0}$ and $\{j_\varepsilon\}_{\varepsilon>0}$ that we denote in the same way. In addition $(\rho,j)$ is a weak measure-valued solution to the Cauchy problem associated with the following Euler-type system
\begin{equation}\label{Euler-limit-j}
\left\{
\begin{array}{ll}
\displaystyle \partial_t\rho +\divop j=0, & x\in\RR^N,\,t\in [0,T),\\
\displaystyle \rho\,\nabla\psi+j=(\phi_0*j)\rho-(\phi_0*\rho)j, & x\in\RR^N,\, t\in [0,T)\\
\displaystyle \rho(0,\cdot)=\rho^0, & x\in \RR^N.
\end{array}
\right.
\end{equation}
\end{theo}
This subsection will be divided into three distinguished parts. The first step will collect the necessary compactness properties of the sequences $\{\rho_\varepsilon\}_{\varepsilon>0}$ and $\{j_\varepsilon\}_{\varepsilon>0}$ that can be inferred from the preceding part of the section, whilst the second step will try to give some insight into what does the dissipation of kinetic energy due to alignment interactions provide to our system. The last step will show how to pass to the limit in all the terms. It will be done by paying a special attention to the nonlinear term and on the admissible range of the parameter $\lambda$ where the singular limit succeeds.

\subsubsection*{First step: Compactness}
Let the initial distribution functions $f_\varepsilon^0$ verify (\ref{hypothesis.initial.data.f.form}), the external forces $-\nabla\psi_\varepsilon$ fulfil (\ref{hypothesis.initial.data.F.form}) and consider the strong global in time solution $f_\varepsilon$ to (\ref{VPFPCS-dimensionless-hyperbolicscaling.restated.form}) with initial data $f_\varepsilon^0$ and  $\lambda\in (0,N/2)$. Notice that by Proposition \ref{Momenta_equations.pro}, Corollary \ref{Moments_bounds_k=2.cor} and the Banach--Alaoglu theorem one has
\begin{align}
\rho_\varepsilon&\overset{*}{\rightharpoonup}\rho, \hspace{0.25cm} \mbox{in }L^\infty(0,T;\mathcal{M}(\RR^N)),\label{Mass_convergence_weak_star.form}\\
j_\varepsilon&\overset{*}{\rightharpoonup} j, \hspace{0.25cm} \mbox{in }L^2(0,T;\mathcal{M}(\RR^N))^N.\label{Current_convergence_weak_star.form}
\end{align}
when $\varepsilon\searrow 0$. Consequently, it is straightforward to pass to the limit in the continuity equation (\ref{MomentumEquation-mass.form}) and to obtain
$$\frac{\partial \rho}{\partial t}+\divop j=0,$$
in the distributional sense. Regarding the equation of balance of linear momentum (\ref{MomentumEquation-current.form})
$$
\varepsilon\frac{\partial j_\varepsilon}{\partial t}+\varepsilon\divop_x \mathcal{S}_\varepsilon+\rho_\varepsilon\nabla_x\psi_\varepsilon+\left(1+\phi_\varepsilon*\rho_\varepsilon\right)j_\varepsilon-\left(\phi_\varepsilon*j_\varepsilon\right)\rho_\varepsilon=0.
$$
Notice that the first two terms converge towards zero in the distributional sense by virtue of (\ref{Moment_v_firstorder.form})--(\ref{Moment_v_secondorder.form}). Furthermore, (\ref{Mass_convergence_weak_star.form})--(\ref{Current_convergence_weak_star.form}) entails the next convergence in the sense of distributions
\begin{equation}\label{Formal_limit_current.form}
\lim_{\varepsilon\rightarrow 0}j_\varepsilon=\lim_{\varepsilon\rightarrow 0}\left\{\left(\phi_\varepsilon*j_\varepsilon\right)\rho_\varepsilon-(\phi_\varepsilon*\rho_\varepsilon)j_\varepsilon-\rho_\varepsilon\nabla_x\psi_\varepsilon\right\}.
\end{equation}
However, we would like to identify some limit equation for $j$ through the equation (\ref{Formal_limit_current.form}). In weak form, we have the next equation
\begin{align*}
\int_0^T\int_{\RR^N}j\cdot\varphi\,dx\,dt=&\lim_{\varepsilon\rightarrow 0}\left\{\int_0^T\int_{\RR^N}\int_{\RR^N}\phi_\varepsilon(\vert x-y\vert)\left(\rho_\varepsilon(t,x)j_\varepsilon(t,y)-\rho_\varepsilon(t,y)j_\varepsilon(t,x)\right)\cdot \varphi(t,x)\,dx\,dy\,dt\right.\\
&\hspace{0.6cm}-\left.\int_0^T\int_{\RR^N}\rho_\varepsilon(t,x)\nabla_x\psi_\varepsilon(t,x)\cdot\varphi(t,x)\,dx\,dt,\right\}
\end{align*}
for any vector-valued test function $\varphi\in C^\infty_c([0,T)\times \RR^N,\RR^N)$. In order to show that such nonlinear term makes sense, let us use the kindness of the commutator therein along with the symmetries of the influence function. An easy change of variables that interchanges $x$ with $y$ then yields the next expression:
\begin{align}\label{Current_equation_weak.form}
\begin{split}
\int_0^T\int_{\RR^N}j\cdot\varphi\,dx\,dt&=\lim_{\varepsilon\rightarrow 0}\left\{\frac{1}{2}\int_0^T\int_{\RR^N}H_\varphi^{\lambda,\varepsilon}(t,x,y)\cdot (\rho_\varepsilon(t,x)j_\varepsilon(t,y)-\rho_\varepsilon(t,y)j_\varepsilon(t,x))\,dx\,dy\,dt\right.\\
&\hspace{1.2cm}-\left.\int_0^T\int_{\RR^N}\rho_\varepsilon(t,x)\nabla_x\psi_\varepsilon(t,x)\cdot\varphi(t,x)\,dx\,dt\right\},
\end{split}
\end{align}
where the integral kernel $H_\varphi^{\lambda,\varepsilon}$ takes the form
$$H_\varphi^{\lambda.\varepsilon}(t,x,y):=\phi_\varepsilon(\vert x-y\vert)(\varphi(t,x)-\varphi(t,y))=\frac{\varphi(t,x)-\varphi(t,y)}{(\varepsilon^2+c_\lambda \vert x-y\vert^2)^\lambda}.$$

Combining the mean value theorem along with Lemma \ref{Phi_epsilon-estimates.lem} we arrive at
\begin{lem}\label{H_estimates.lem}
Let us fix $\varepsilon>0$, $\lambda\in (0,1/2)$ and a test function $\varphi\in C^\infty_c([0,T)\times \RR^N)$. Then, the next estimates hold for every couple $(x,y)\in \RR^{2N}\setminus \Delta$ and $t\in [0,T)$:
\begin{enumerate}
\item $\displaystyle\left\vert H_\varphi^{\lambda,\varepsilon}(t,x,y)\right\vert\leq \frac{1}{c_\lambda^\lambda}\Vert \nabla_x\varphi\Vert_{C([0,T)\times \RR^N)}\vert x-y\vert^{1-2\lambda},$
\item $\displaystyle \vert H_\varphi^{\lambda,\varepsilon}(t,x,y)\vert\leq \frac{1}{\varepsilon^{2\lambda}}\Vert \nabla_x\varphi\Vert_{C([0,T)\times \RR^N)}\vert x-y\vert,$
\item $\displaystyle \vert H_\varphi^{\lambda,\varepsilon}(t,x,y)-H_\varphi^{\lambda,0}(t,x,y)\vert\leq \varepsilon^{1-2\lambda} C_\lambda\Vert \nabla_x\varphi\Vert_{C([0,T)\times \RR^N)},$
\end{enumerate}
where $\Delta$ stands for the diagonal of $\RR^{2N}$ i.e., $\Delta:=\{(x,y)\in\RR^{2N}:\,x=y\}$.
\end{lem}

Let us first note some difference between the cases $\lambda\in (0,1/2)$ and $\lambda\geq 1/2$. In the former case, not only does $H_\varphi^{\lambda,\varepsilon}$ belong to $C([0,T],C_0(\RR^N,\RR^N))$ for every positive $\varepsilon$ but it also belongs to such space when $\varepsilon=0$. This is no longer true when $\lambda=1/2$ because despite the continuity be granted for positive $\varepsilon$, the corresponding kernel with $\varepsilon=0$ takes the form
$$H_\varphi^{\frac{1}{2},0}(t,x,y):=\frac{\varphi(t,x)-\varphi(t,y)}{\sqrt{3}\vert x-y\vert}.$$
and, although bounded, it clearly loses the continuity at the diagonal points $x=y$ regardless of the modulus of continuity of $\varphi(t,x)$. This is a common restriction in many other classical systems such as the 2D Euler equations in vorticity formulation (see e.g. \cite{Schochet}) or the 2D Vlasov--Poisson--Fokker--Planck system (see e.g. \cite{NietoPoupaudSoler-Hiperbolico}). Similarly, in the latter case $\lambda>1/2$ the kernel is discontinuous along the diagonal points and it is not necessarily bounded at $x=y$ since it can even blow up. Hence, one cannot expect to pass to the limit for $\lambda>1/2$, but at most for $\lambda\in (0,1/2]$.

In this part we will focus on $\lambda\in (0,1/2)$, where the continuity along the diagonal point is granted. In such case, Lemma \ref{H_estimates.lem} shows that
\begin{equation}\label{Kernel_uniform_convergence.form}
\lim_{\varepsilon\rightarrow 0} \Vert H_\varphi^{\lambda,\varepsilon}- H_\varphi^{\lambda,0}\Vert_{C([0,T],C_0(\RR^N,\RR^N))}=0.
\end{equation}
The next step will be to show that,
$$\rho_\varepsilon\otimes j_\varepsilon\overset{*}{\rightharpoonup} \rho\otimes j\hspace{0.25cm}\mbox{in }L^2(0,T;\mathcal{M}(\RR^{2N}))^N,$$
which does not directly follows just from (\ref{Mass_convergence_weak_star.form})-(\ref{Current_convergence_weak_star.form}). In general some sort of time equicontinuity in certain space is needed (see \cite{NietoPoupaudSoler-Hiperbolico,PoupaudSoler-Parabolico,Schochet}).

\begin{theo}\label{Mass_convergence_strong_weak_star.theo}
Let the initial distribution functions $f_\varepsilon^0$ verify (\ref{hypothesis.initial.data.f.form}), the external forces $-\nabla\psi_\varepsilon$ fulfill (\ref{hypothesis.initial.data.F.form}) and consider the strong global in time solution $f_\varepsilon$ to (\ref{VPFPCS-dimensionless-hyperbolicscaling.restated.form}) with initial data $f_\varepsilon^0$ and $\lambda\in (0,N/2)$. Then, the sequence $\rho_\varepsilon$ does not only converges weakly star in $L^\infty(0,T;\mathcal{M}(\RR^N))$ to $\rho$ but also in $C([0,T];\mathcal{M}(\RR^N)-weak\,*)$, i.e.,
$$\lim_{n\rightarrow 0}\sup_{t\in [0,T]}\left\vert\int_{\RR^N}\phi(x)d_x(\rho_\varepsilon(t)-\rho(t))\right\vert=0,$$
for every continuous test function $\phi\in C_0(\RR^N)$. In particular, $\rho(0,\cdot)=\rho^0$.
\end{theo}
\begin{proof}
Consider any test function $\varphi(t,x)=\eta(t)\phi(x)$ where $\eta\in C^\infty_c(0,T)$ and $\phi\in C^\infty_c(\RR^N)$. Notice that (\ref{MomentumEquation-mass.form}) gives rise to the following equation in weak form
$$\int_0^T\frac{\partial\eta}{\partial t}(t)\left(\int_{\RR^N}\rho_\varepsilon(t,x)\phi(x)\,dx\right)\,dt=\int_0^T\eta(t)\left(\int_{\RR^N}j_\varepsilon(t,x)\cdot \nabla_x\phi(x)\,dx\right)\,dt.$$
Since the scaled current is bounded in $L^2(0,T;L^1(\RR^N))$, then
$$\left\vert\int_0^T\frac{\partial\eta}{\partial t}(t)\left(\int_{\RR^N}\rho_\varepsilon(t,x)\phi(x)\,dx\right)\,dt\right\vert\leq \Vert \eta\Vert_{L^2(0,T)}\left\Vert j_\varepsilon\right\Vert_{L^2(0,T;L^1(\RR^N))}\Vert \phi\Vert_{W^{1,\infty}(\RR^N)}.$$
A well know characterization of Sobolev space leads to
$$\left\Vert \int_{\RR^N}\rho_\varepsilon(\cdot ,x)\phi(x)\,dx\right\Vert_{H^1(0,T)}\leq \left\Vert j_\varepsilon\right\Vert_{L^2(0,T;L^1(\RR^N))}\Vert \phi\Vert_{W^{1,\infty}(\RR^N)}.$$
In particular,
$$\left\vert\int_{\RR^N}(\rho_\varepsilon(t_1,x)-\rho_\varepsilon(t_2,x))\phi(x)\,dx\right\vert\leq \left\Vert j_\varepsilon\right\Vert_{L^2(0,T;L^1(\RR^N))}\Vert \phi\Vert_{W^{1,\infty}(\RR^N)}\vert t_1-t_2\vert^{1/2},$$
for every $t_1,t_2\in [0,T)$. Consequently, the next bound holds
$$\Vert \rho_\varepsilon\Vert_{C^{0,\frac{1}{2}}([0,T];W^{-1,1}(\RR^N))}\leq \left\Vert j_\varepsilon\right\Vert_{L^2(0,T;L^1(\RR^N))}.$$
Notice that the boundedness of $\rho_\varepsilon$ in $C([0,T];W^{-1,1}(\RR^N))$ also follows from the analogous bound in $L^\infty(0,T;L^1(\RR^N))$ (see Proposition \ref{Momenta_equations.pro}) and the chain of continuous inclusions
$$L^1(\RR^N)\hookrightarrow\mathcal{M}(\RR^N)\hookrightarrow W^{-1,1}(\RR^N).$$
Then, the weak star form of the Ascoli--Arzel\`a theorem yields the convergence in $C([0,T];W^{-1,1}(\RR^N)-weak\,*)$ and by density of $W^{1,\infty}(\RR^N)$ in $C_0(\RR^N)$ along with the boundedness of $\rho_\varepsilon$ in $L^\infty(0,T;L^1(\RR^N))$
$$\rho_\varepsilon\rightarrow \rho\hspace{0.25cm}\mbox{ in }C([0,T];\mathcal{M}(\RR^N)-weak\,*).$$
To end the proof, let us note that in particular
$$\rho_\varepsilon(0,\cdot)\overset{*}{\rightharpoonup}\rho(0,\cdot)\ \mbox{ in }\mathcal{M}(\RR^N).$$
By hypothesis (\ref{hypothesis.initial.data.f.form}),
$$\rho_\varepsilon(0,\cdot)\overset{*}{\rightharpoonup}\rho^0\ \mbox{ in }\mathcal{M}(\RR^N).$$
Since the weak-star topology is Hausdorff, one concludes that both limits agrees, i.e., $\rho(0,\cdot)=\rho^0$.
\end{proof}

Let us conclude our assertion with the next result.

\begin{cor}\label{Joint_convergence.cor}
Let the initial distribution functions $f_\varepsilon^0$ verify (\ref{hypothesis.initial.data.f.form}), the external forces $-\nabla\psi_\varepsilon$ fulfil (\ref{hypothesis.initial.data.F.form}) and consider the strong global in time solution $f_\varepsilon$ to (\ref{VPFPCS-dimensionless-hyperbolicscaling.restated.form}) with initial data $f_\varepsilon^0$ and $\lambda\in (0,N/2)$. Then,
$$\rho_\varepsilon\otimes j_\varepsilon\overset{*}{\rightharpoonup} \rho\otimes j\ \mbox{ in }\ L^2(0,T;\mathcal{M}(\RR^{2N}))^N.$$
\end{cor}
\begin{proof}
Notice that the Riesz representation theorem for $L^p$ dual spaces and finite Radon measures leads to the next identification
$$L^2(0,T;\mathcal{M}(\RR^{2N}))\equiv L^2(0,T;C_0(\RR^{2N}))^*.$$
Let us then consider a test function $\varphi\in L^2(0,T;C_0(\RR^{2N}))$. By density, one can assume that $\varphi$ takes the form
$$\varphi(t,x,y):=\eta(t)\phi(x)\psi(y),$$
for $\eta\in L^2(0,T)$ and $\phi,\psi\in C_0(\RR^N)$. The objective is to show that
$$I_{i,\varepsilon}:=\int_0^T\int_{\RR^N}\int_{\RR^N}\varphi(t,x,y)\left(d_x(\rho_\varepsilon(t))d_y\left(j_\varepsilon(t)\right)^i-d_x(\rho(t))d_y(j(t))^i\right)\,dt\rightarrow 0,$$
as $\varepsilon\rightarrow 0$. The superscript $i\in \{1,\ldots,N\}$ stands for each component of the above current vectors. $I_{i,\varepsilon}$ can be split as follows
$$I_{i,\varepsilon}=I_{i,\varepsilon}^1+I_{i,\varepsilon}^2+I_{i,\varepsilon}^3,$$
where each term reads
\begin{align*}
I_{i,\varepsilon}^1&:=\int_0^T\eta(t)\left(\int_{\RR^N}\phi(x)d_x(\rho_\varepsilon(t)-\rho(t))\right)\left(\int_{\RR^N}\psi(y)d_y\left(j\varepsilon(t)-j(t)\right)^i\right)\,dt,\\
I_{i,\varepsilon}^2&:=\int_0^T\eta(t)\left(\int_{\RR^N}\phi(x)d_x(\rho_\varepsilon(t)-\rho(t))\right)\left(\int_{\RR^N}\psi(y)d_y\left(j(t)\right)^i\right)\,dt,\\
I_{i,\varepsilon}^3&:=\int_0^T\eta(t)\left(\int_{\RR^N}\phi(x)d_x(\rho(t))\right)\left(\int_{\RR^N}\psi(y)d_y\left(j_\varepsilon(t)-j(t)\right)^i\right)\,dt.
\end{align*}
We will restrict to the first term, $I_{i,\varepsilon}^1$, since the reasoning in the remaining two terms is similar. Let us define the sequence of scalar functions given by
$$R_\varepsilon(t):=\int_{\RR^N}\phi(x)d_x\left(\rho_\varepsilon(t)-\rho(t)\right),\hspace{0.5cm}\eta_\varepsilon(t):=\eta(t)R_\varepsilon(t).$$
By Theorem \ref{Mass_convergence_strong_weak_star.theo} one has that
\begin{align*}
R_\varepsilon&\rightarrow 0\hspace{0.25cm}\mbox{ in }L^\infty(0,T),\\
\eta_\varepsilon&\rightarrow 0\hspace{0.25cm}\mbox{ in }L^2(0,T).
\end{align*}
Then, $I_{i,\varepsilon}^1$ can be restated as
$$I_{i,\varepsilon}^1=\int_0^T\int_{\RR^N}\eta_\varepsilon(t)\psi(y)\,d_y\left(j_\varepsilon(t)-j(t)\right)^i\,dt.$$
Therefore, such term becomes zero when $\varepsilon\rightarrow 0$ because the test functions $\eta_\varepsilon\otimes \psi$ strongly converges to zero in $L^2(0,T;C_0(\RR^N))$ and $j_\varepsilon-j$ also converges to zero weakly star in $L^2(0,T;\mathcal{M}(\RR^N))$ by (\ref{Current_convergence_weak_star.form}).
\end{proof}

\subsubsection*{Second step: Non-concentration}
Recall that the limiting kernels $H_\varphi^{\lambda,\varepsilon}$ are continuous except at most at the diagonal points $x=y$. Then, in order for the nonlinear term in weak form
$$\frac{1}{2}\int_0^T\int_{\RR^{N}}\int_{\RR^N}H_\varphi^{\lambda,\varepsilon}(t,x,y)\cdot (\rho_\varepsilon(t,x)j_\varepsilon(t,y)-\rho_\varepsilon(t,y)j_\varepsilon(t,x))\,dx\,dy,$$
to pass to the limit, we have to ensure that the limiting terms $\rho\otimes j-j\otimes \rho$ do not concentrate on the set of points of discontinuity of $H_\varepsilon^{\lambda,0}$. This is the content of the next result, that holds for $\lambda$ belonging to the whole range $(0,N/2)$.

\begin{lem}\label{non-concentration.lem}
Let the initial distribution functions $f_\varepsilon^0$ verify (\ref{hypothesis.initial.data.f.form}), the external forces $-\nabla\psi_\varepsilon$ fulfil (\ref{hypothesis.initial.data.F.form}) and consider the strong global in time solution $f_\varepsilon$ to (\ref{VPFPCS-dimensionless-hyperbolicscaling.restated.form}) with initial data $f_\varepsilon^0$ and $\lambda\in (0,N/2)$. Then 
$$\liminf_{R,\varepsilon\rightarrow 0}\vert \rho_\varepsilon(t,\cdot)\otimes j_\varepsilon(t,\cdot)-j_\varepsilon(t,\cdot)\otimes \rho_\varepsilon(t,\cdot)\vert (\Omega_R)=0,\ \mbox{ a.e. }t\in [0,T).$$
where $\Omega_R$ denotes the augmented diagonal of $\RR^{N}\times \RR^N$ with radius $R>0$, i.e., 
$$\Omega_R:=\{(x,y)\in \RR^{2N}:\,\vert x-y\vert<R\}.$$
\end{lem}
\begin{proof}
Consider any $\varepsilon>0$, fix any radius $R>0$ and note that
\begin{align*}
\vert\rho_\varepsilon(t,\cdot)&\otimes j_\varepsilon(t,\cdot)-j_\varepsilon(t,\cdot)\otimes \rho_\varepsilon(t,\cdot)\vert (\Omega_R)\\
&= \iint_{\vert x-y\vert<R}\vert\rho_\varepsilon(t,x)\otimes j_\varepsilon(t,y)-j_\varepsilon(t,x)\otimes \rho_\varepsilon(t,y)\vert\,dx\,dy\\
&\leq (\varepsilon^2+R^2)^{\lambda/2}\int_{\RR^{4N}}\vert v-w\vert \phi_\varepsilon(\vert x-y\vert)^{1/2}f_\varepsilon(t,x,v)f_\varepsilon(t,y,w)\,dx\,dy\,dv\,dw.
\end{align*}
Then, the Cauchy--Schwartz inequality along with the estimate (\ref{Moment_v_cuadratic.form}) in Corollary \ref{Moments_bounds_k=2.cor} yield
$$\left(\int_0^T\vert\rho_\varepsilon(t,\cdot)\otimes j_\varepsilon(t,\cdot)-j_\varepsilon(t,\cdot)\otimes \rho_\varepsilon(t,\cdot)\vert (\Omega_R)^2\,dt\right)^{1/2}\leq (\varepsilon^2+R^2)^{\lambda/2}(2\varepsilon E_0+(2NT+F_0^2)M_0)^{1/2}M_0,$$
and consequently, taking limits when $\varepsilon$ and $R$ become zero
$$\liminf_{\varepsilon,R\rightarrow 0}\left(\int_0^T\vert\rho_\varepsilon(t,\cdot)\otimes j_\varepsilon(t,\cdot)-j_\varepsilon(t,\cdot)\otimes \rho_\varepsilon(t,\cdot)\vert (\Omega_R)^2\,dt\right)^{1/2}=0.$$
Then, the proof is done by virtue of Fatou's lemma.
\end{proof}

Now, let us recall the next result that may be used to show how to pass to the limit on a sequence of finite Radon measures (without distinguished sign) that do not exhibit concentrations on a set, against a bounded function that is discontinuous, at most, on such set of concentrations. It is a well known result for positive measures and it is the cornerstone in many other frameworks (e.g. 2D Euler with signed vortex sheet initial data \cite{Delort,Schochet} or the 2D Vlasov--Poisson--Fokker--Planck system \cite{NietoPoupaudSoler-Hiperbolico}). The proof for positive measures can be found in \cite[Lemma 2.1]{Poupaud}, \cite{Schochet} or \cite[Theorems 62-63, Chapter IV]{Schwartz}.

\begin{pro}\label{limitetestdiscontinua.pro}
Let $\{\mu_\varepsilon\}_{\varepsilon>0}\subseteq \mathcal{M}(\RR^N)$ be a sequence of finite Radon measures (without distinguished sign) and assume that
\begin{enumerate}
\item $\mu_\varepsilon\overset{*}{\rightharpoonup}\mu$ in $\mathcal{M}(\RR^N)$ when $\varepsilon\rightarrow 0$ for some $\mu \in \mathcal{M}(\RR^N)$,
\item there exists some closed set $C\subseteq \RR^N$ such that
$$\liminf_{\varepsilon,R\rightarrow 0}\vert \mu_\varepsilon\vert (\Omega_R)=0,$$
where $\Omega_R$ stands for the augmented $C$ with radius $R>0$, i.e., 
$$\Omega_R:=C+B_R(0)=\{x\in\RR^N:\dist(x,C)<R\}.$$
Then,
\begin{equation}\label{limitetestdiscontinua.form}
\lim_{\varepsilon\rightarrow 0}\int_{\RR^N}\varphi\,d\mu_\varepsilon=\int_{\RR^N}\varphi\,d\mu,
\end{equation}
for every measurable and bounded function $\varphi:\RR^N\longrightarrow\RR$ whose discontinuity points lie in $C$ and such that $\varphi$ falls off at infinity.
\end{enumerate}
\end{pro}
\begin{proof}
Let us decompose $\mu_\varepsilon$ into its positive and negative part, according to Hahn's decomposition theorem, i.e., $\mu_\varepsilon=\mu_\varepsilon^+-\mu_\varepsilon^-$. We can assume, without loss of generality, that
$$\mu_\varepsilon^{\pm}\overset{*}{\rightharpoonup}\mu^{\pm}\ \mbox{ in }\ \mathcal{M}(\RR^N),$$
where $\mu^{\pm}\in\mathcal{M}(\RR^N)$ do not necessarily agree with the Jordan decomposition of $\mu=\mu^+-\mu^-$.
By hypothesis, it is clear that
$$\liminf_{\varepsilon,R\rightarrow 0}\mu_\varepsilon^{\pm}(\Omega_R)=0.$$
As a consequence of the weak-star lower semicontinuity of the total variation norm, one obtains
$$\mu^{\pm}(\Omega_R)\leq\liminf_{\varepsilon\rightarrow 0} \mu_\varepsilon^{\pm}(\Omega_R).$$
Then, by monotonicity $\mu^{\pm}(C)=0$ and one can then apply Lemma 2.1 in \cite{Poupaud} to the positive measures $\mu_\varepsilon^+$ and $\mu_\varepsilon^-$ to end the proof of the theorem.
\end{proof}

\begin{rem}
An important fact to be remarked is that as stated in \cite{Poupaud}, the above result is not true if one only assumes that $\vert \mu\vert(C)=0$. The example there exhibited is
$$\mu_\varepsilon=\delta_{\varepsilon}-\delta_{-\varepsilon},\ \ \mu=0,$$
that obviously verify $\vert \mu\vert(\{0\})=0$ but fails (\ref{limitetestdiscontinua.form}) (take $\varphi$ any cutoff function at infinity of the sign function). There is no contradiction on that since the non-concentration property in the above result also fails
$$\liminf_{\varepsilon,R\rightarrow}\vert \mu_\varepsilon\vert(-R,R)=\liminf_{\varepsilon,R\rightarrow 0}(\delta_\varepsilon+\delta_{-\varepsilon})(-R,R)=2\neq 0.$$
\end{rem}

\subsubsection*{Third step: Convergence}
To end this section, let us show that the above results allow us to take limits in (\ref{Current_equation_weak.form}) as $\varepsilon\rightarrow 0$ for the restricted range of the parameter $\lambda\in (0,1/2]$. For $\lambda\in (0,1/2)$, it is a direct consequence of the Corollary \ref{Joint_convergence.cor} and the uniform convergence (\ref{Kernel_uniform_convergence.form}). The case $\lambda=1/2$ requires a special analysis since (\ref{Kernel_uniform_convergence.form}) does not hold. However, the estimate (\ref{Moment_v_cuadratic.form}) of the dissipation of kinetic energy due to alignment interactions will suffices to reinforce the lack of uniformity in the convergence of the approximate kernels $H_\varphi^{\lambda,\varepsilon}$.
\begin{cor}\label{Commutator_forces_limits.cor}
Let the initial distribution functions $f_\varepsilon^0$ verify (\ref{hypothesis.initial.data.f.form}), the external forces $-\nabla\psi_\varepsilon$ fulfil (\ref{hypothesis.initial.data.F.form}) and consider the strong global in time solution $f_\varepsilon$ to (\ref{VPFPCS-dimensionless-hyperbolicscaling.restated.form}) with initial data $f_\varepsilon^0$ and $\lambda\in (0,1/2]$.
Then, both terms in (\ref{Current_equation_weak.form}) pass to the limits as $\varepsilon\rightarrow 0$. Specifically,
\begin{align*}
\lim_{\varepsilon\rightarrow 0}\frac{1}{2}\int_0^T\int_{\RR^N}\int_{\RR^N}H_\varphi^{\lambda,\varepsilon}(\rho_\varepsilon\otimes j_\varepsilon-j_\varepsilon\otimes \rho_\varepsilon)\,dx\,dy\,dt&=\frac{1}{2}\int_0^T\int_{\RR^N}\int_{\RR^N}H_\varphi^{\lambda,0}(\rho\otimes j-j\otimes \rho)\,dx\,dy\,dt,\\
\lim_{\varepsilon\rightarrow 0}\int_0^T\int_{\RR^N}\rho_\varepsilon\,\nabla\psi_\varepsilon\cdot \varphi\,dx\,dt&=\int_0^T\int_{\RR^N}\rho\,\nabla\psi\cdot \varphi\,dx\,dt,
\end{align*}
for every test function $\varphi\in C^\infty_c([0,T)\times \RR^N,\RR^N)$.
\end{cor}
\begin{proof}
Since the proof for $\lambda\in (0,1/2)$ is a direct consequence of Corollary \ref{Joint_convergence.cor} and the uniform convergence (\ref{Kernel_uniform_convergence.form}), then we restrict to the endpoint critical case $\lambda=1/2$. To this end, let us show that
$$I_\varepsilon:=\frac{1}{2}\int_0^T\int_{\RR^N}\int_{\RR^N}\left[H_\varphi^{\frac{1}{2},\varepsilon}(\rho_\varepsilon\otimes j_\varepsilon-j_\varepsilon\otimes \rho_\varepsilon)-H_\varphi^{\frac{1}{2},0}(\rho\otimes j-j\otimes \rho)\right]\,dx\,dy\,dt,$$
becomes zero when $\varepsilon\rightarrow 0$. For simplicity, we decompose such integral $I_\varepsilon$ into $I_\varepsilon=II_\varepsilon+III_\varepsilon$, where
\begin{align*}
II_\varepsilon&:=\frac{1}{2}\int _0^T\int_{\RR^N}\int_{\RR^N}(H_\varphi^{\frac{1}{2},\varepsilon}-H_\varphi^{\frac{1}{2},0})(\rho_\varepsilon\otimes j_\varepsilon-j_\varepsilon\otimes \rho_\varepsilon)\,dx\,dy\,dt,\\
III_\varepsilon&:=\frac{1}{2}\int_0^T\int_{\RR^N}\int_{\RR^N}H_\varphi^{\frac{1}{2},0}\left[(\rho_\varepsilon\otimes j_\varepsilon-j_\varepsilon\otimes \rho_\varepsilon)-(\rho\otimes j-j\otimes \rho)\right]\,dx\,dy\,dt.
\end{align*}
First, since $H_\varphi^{\frac{1}{2},0}$ is discontinuous at most at the diagonal $\Delta$ of $\RR^N\times \RR^N$, then Lemma \ref{non-concentration.lem} and Proposition \ref{limitetestdiscontinua.pro} show that $III_\varepsilon\rightarrow 0$ when $\varepsilon\rightarrow 0$. Regarding $II_\varepsilon$, note that
$$\vert \phi_\varepsilon(r)-\phi_0(r)\vert=\frac{(\varepsilon^2+3 r^2)^{1/2}-\sqrt{3}r}{\sqrt{3}r(\varepsilon^2+3 r^2)^{1/2}}\leq \frac{\varepsilon^{1/2}}{\sqrt{3}r}\phi_\varepsilon(r)^{1/2},\ \ r>0.$$
Consequently, there exists come constant $C>0$ depending on the text function $\varphi\in C^\infty_c([0,T)\times \RR^N,\RR^N)$ such that
$$\vert II_\varepsilon\vert\leq C\varepsilon^{1/2}\int_0^T\int_{\RR^{4N}}\vert v-w\vert\phi_\varepsilon(\vert x-y\vert)f_\varepsilon(t,x,v)f_\varepsilon(t,y,w)\,dx\,dy\,dv\,dw\,dt,$$
and by the Cauchy--Schwartz inequality
$$\vert II_\varepsilon\vert\leq CT^{1/2}(2\varepsilon E_0+(2NT+F_0^2)M_0)^{1/2}M_0\varepsilon^{1/2},$$
that obviously implies $II_\varepsilon\rightarrow 0$ when $\varepsilon\rightarrow 0$.

The second part is clear by virtue of the weak-star convergence of $\rho_\varepsilon$ and our hypothesis (\ref{hypothesis.initial.data.F.form}) on the external forces $-\nabla\psi_\varepsilon$.
\end{proof}

\begin{rem}
Let us note that the passage to the limit in the second term in Corollary \ref{Commutator_forces_limits.cor} directly follows from the imposed hypothesis (\ref{hypothesis.initial.data.F.form}) on the external forces $-\nabla\psi_\varepsilon$. Indeed, since $\rho_\varepsilon$ is only known to converge in $L^\infty(0,T;\mathcal{M}(\RR^N))$ according to our a priori estimates, a duality argument then steers us toward considering forces that fall off at infinity, i.e., $-\nabla\psi_\varepsilon\in L^1(0,T;C_0(\RR^N,\RR^N))$ since
$$L^1(0,T;C_0(\RR^N))^*\equiv L^\infty(0,T;\mathcal{M}(\RR^N)),$$
by the Riesz representation theorem. This might seem a very imposing hypothesis on the forces since generally, forces will not be even bounded. A slight improvement of our result (that was not considered in the previous reasoning for the sake of simplicity) is that one can replace the hypothesis $-\nabla\psi_\varepsilon(t,\cdot)\in C_0(\RR^N,\RR^N)$ with $-\nabla\psi_\varepsilon(t,\cdot)\in C(\RR^N,\RR^N)$, i.e., it is not compulsory that forces decay at infinity but only that they remain bounded and continuous. The passage to the limit in the second term in Corollary \ref{Commutator_forces_limits.cor} then goes as follows. Notice that
$$L^1(0,T;\mathcal{B}(\RR^N))^*=L^\infty(0,T;\mathcal{M}(\RR^N)-narrow),$$
where $\mathcal{B}(\RR^N)$ is the space of bounded and continuous functions endowed with the uniform norm and $\mathcal{M}(\RR^N)-narrow$ means the space of finite Radon measures endowed with narrow topology, i.e., the weak topology associated with the duality $\left<\mathcal{B}(\RR^N),\mathcal{M}(\RR^N)\right>$ given by
$$\left<\varphi,\mu\right>:=\int_{\RR^N}\varphi\,d\mu,\ \ \varphi\in \mathcal{B}(\RR^N),\,\mu\in\mathcal{M}(\RR^N).$$
In our case, note that (\ref{Momenta_equation_x.form}) in Proposition \ref{Momenta_equations.pro} with $k=1$ and $k=2$ along with Corollary \ref{Moments_bounds_k=2.cor} yield
\begin{align*}
\Vert \vert x\vert \rho_\varepsilon\Vert_{L^\infty(0,T;L^1(\RR^N))}&\leq \Vert \vert x\vert \rho_\varepsilon^0\Vert_{L^1(\RR^N)}+T^{1/2}M_0^{1/2}\left(2\varepsilon E_0+(2NT+F_0^2)M_0\right)^{1/2},\\
\Vert \vert x\vert^2 \rho_\varepsilon\Vert_{L^\infty(0,T;L^1(\RR^N))}&\leq \Vert \vert x\vert^2 \rho_\varepsilon^0\Vert_{\RR^N}+2\left(T\Vert \vert x\vert \rho_\varepsilon\Vert_{L^\infty(0,T;L^1(\RR^N))}\right)^{1/2}T^{1/2}\left(M_0^{1/2}(2\varepsilon E_0+(2NT+F_0^2)M_0)\right)^{1/4},\\
\Vert \vert x\vert j_\varepsilon\Vert_{L^1(0,T;L^1(\RR^N))}&\leq \frac{1}{2}\left(T\Vert \vert x\vert^2 \rho_\varepsilon\Vert_{L^\infty(0,T;L^1(\RR^N))}+(2E_0+(2NT+F_0^2)M_0)\right).
\end{align*}
Thus, if apart from the hypothesis (\ref{hypothesis.initial.data.f.form}) one also assumes that
$$\int_{\RR^N}\vert x\vert^2 \rho_\varepsilon^0(x)\,dx\leq S_0,$$
then $\rho_\varepsilon\otimes j_\varepsilon$ has some uniformly bounded moment and Prokhorov's compactness theorem entails
$$\rho_\varepsilon\otimes j_\varepsilon \overset{*}{\rightarrow}\rho\otimes j\ \mbox{ in }\ L^2(0,T;\mathcal{M}(\RR^{2N})^N-narrow),$$
then supporting the desired convergence result under less restricted hypothesis on the decay of the forces.

Indeed, as mentioned in the Introduction, one might have considered $-\nabla\psi_\varepsilon$ to be self-generated. For instance, recall that the forces in the Vlasov--Poisson--Fokker--Planck system come from the Poisson law $\Delta\psi_\varepsilon=\theta \rho_\varepsilon$ ($\theta=\pm 1$ suggesting the attractive or repulsive character of the interactions), i.e., 
$$-\nabla\psi_\varepsilon=\theta(\nabla \Gamma_N)*\rho_\varepsilon,$$
where $\Gamma_N$ stands for the fundamental solution of $-\Delta$ in $\RR^N$. Since they are nonlinear, they need to be handled like our nonlinear commutator. Note that when the density only enjoys $L^1$ bounds, then the force $-\nabla\psi_\varepsilon$ can be shown to be bounded at most in the $N=1$ case, but it does not generally falls off at infinity (hence, the above remark is relevant in such setting). The case $N=2$ is even harder and resembles our case $\lambda=1/2$, that required a non-concentration argument of the limiting density \cite{NietoPoupaudSoler-Hiperbolico}. For general $N$, relative entropy methods (that do not work in our case) have been used \cite{GNPS}. 

Consequently, despite having derived a method to deal with bounded continuous forces, each nonlocal force given by a potential $\psi_\varepsilon=k*\rho_\varepsilon$ and some kernel $k$ needs to be treated separately through similar ideas to those developed here and in the related papers \cite{BellouquidCalvoNietoSoler-Intermedio,GNPS,NietoPoupaudSoler-Hiperbolico,PoupaudSoler-Parabolico}. Also, the kernels appearing in fluid models like those in the aggregation equation, the gSQG model or other Euler-type systems (see \cite{BertozziVerdera,CastroCordobaGomezSerrano,delaHozHassainiaHmidi,TopazBertozzi1,TopazBertozzi2}) can be approached via this techniques.
\end{rem}

\section{Analysis of the limiting equations}\label{Section-AnalisisLimit}
In this section we will focus on analyzing the macroscopic system arising in the singular and hydrodynamic limit of the preceding section for the kinetic Cucker--Smale model. For simplicity, let us define the velocity field $u=j/\rho$ and note that such limiting system can be restated as
 \begin{equation}\label{Euler.limit.eq}
 \left\{
 \begin{array}{ll}
 \partial_t\rho+\divop(\rho u)=0, & x\in \RR^N,\,t\geq 0,\\
 \rho(0,x)=\rho^0(x),& x\in\RR^N,\\
 u=\phi_0*(\rho u)-(\phi_0*\rho)u-\nabla\psi, & x\in \RR^N,\,t\geq 0,
 \end{array}
 \right.
 \end{equation}
Notice that such systems consists of $N+1$ unknowns ($\rho$ and $u=(u_1\ldots,u_N)$), $N+1$ equations (a scalar conservation law for $\rho$ and a vector-valued equation for $u$) and a Cauchy datum at $t=0$. 
First, we will revisit the existence theory of weak solutions in $L^p$ for scalar conservation laws in conservative form driven by $W^{1,\infty}$ velocity fields. Second, we will analyze the properties of the above commutator
$$\phi_0*(\rho u)-(\phi_0*\rho)u=-\frac{1}{c_\lambda^\lambda}[M_u,I_{N-{2\lambda}}]\rho=-\int_{\RR^N}\phi_0(\vert x-y\vert)(u(t,x)-u(t,y))\rho(t,y)\,dy,$$
for $u(t,\cdot)\in W^{1,\infty}(\RR^N,\RR^N)$ and $\rho(t,\cdot)\in L^p(\RR^N)$. Finally, we will arrange the above properties about the scalar conservation law and the commutator to obtain a solution to (\ref{Euler.limit.eq}) by means of a fixed point procedure.

\subsection{Linear transport equations in conservative form with Lipschitz continuous field}

\begin{theo}\label{Scalar.conservation.law.theo}
Set a velocity field $u\in L^1(0,T;W^{1,\infty}(\RR^N,\RR^N))$ and consider the next Cauchy problem for the density $\rho=\rho(t,x)$
\begin{equation}\label{Scalar.conservation.law.eq}
\left\{
\begin{array}{ll}
\displaystyle \partial_t \rho+\divop(\rho u)=0, & x\in \RR^N,\,t\geq 0,\\
\displaystyle \rho(0,x)=\rho^0(x), & x\in \RR^N.
\end{array}
\right.
\end{equation}
where the initial data $\rho^0\geq 0$ is taken in $L^p(\RR^N)$ for some $1\leq p\leq \infty$. Then there exists a weak solution $\rho$ to (\ref{Scalar.conservation.law.eq}) such that $\rho\in L^\infty(0,T;L^p(\RR^N))$. Indeed, when $1<p\leq \infty$ then
$$
\Vert \rho(t,\cdot)\Vert_{L^p(\RR^N)}\leq \exp\left(\frac{1}{p'}\Vert u\Vert_{L^1(0,T;W^{1,\infty}(\RR^N))}\right)\Vert\rho^0\Vert_{L^p(\RR^N)}
$$
for any $t\in [0,T)$, where $p'$ stands for the conjugated exponent of $p$ i.e., $p'=\frac{p}{p-1}$. When $p=1$ we recover the conservation of the total mass, namely,
$$
\Vert \rho(t,\cdot)\Vert_{L^1(\RR^N)}=\Vert\rho^0\Vert_{L^1(\RR^N)}
$$
for any $t\in [0,T)$. When $u\in C^1([0,T);C^1(\RR^N,\RR^N))$, then the above weak solution is indeed unique.
\end{theo}

First, let us recall that $\rho\in L^1_{loc}(0,T;\mathcal{M}(\RR^N))$ is said to be a weak solution to (\ref{Scalar.conservation.law.eq}) when
\begin{equation}\label{Scalar.conservation.law.weak.formulation.eq}
\int_0^T\int_{\RR^N}\left(\frac{\partial\varphi}{\partial t}+u\cdot \nabla \varphi\right)\rho\,dx\,dt=-\int_{\RR^N}\rho^0(x)\varphi(0,x)\,dx,
\end{equation}
for every test function $\varphi\in C^1_c([0,T)\times\RR^N)$.
\begin{proof}
The assumed hypothesis on the transport field $u$ yields the existence of an uniquely defined flux of $u$ in the extended sense of Caratheodory (see \cite[Theorem 1.1, Chapter 2]{CoddingtonLevinson}). Specifically, since the three Caratheodory conditions (continuity in $x$, measurability in $t$ and boundedness both in $t$ and $x$ by a $t$-dependent integrable function) fulfill, then the characteristic system
$$\left\{
\begin{array}{ll}
\displaystyle\frac{dX}{dt}=u(t,X), & t\in [0,T)\\
\displaystyle X(t_0)=x_0, &
\end{array}
\right.$$
has an unique locally absolutely continuous global in time solution $X(t;t_0,x_0)$ for every $0\leq t_0<T$ and $x_0\in\RR^N$ such that the ODE holds in the a.e. sense (notice that absolutely continuous functions are a.e. differentiable). The Lipschitz continuity of $u$ does not only ensures the global definition by the Gronwall lemma but also its uniqueness. Indeed, $X(t;t_0,\cdot)$ is a bi-Lipschitz diffeomorphism uniformly in $t,t_0$. Specifically, note that the Lipschitz continuity with respect to $x$ shows that
\begin{align*}
\frac{d}{dt}\vert X(t;t_0,x_1)-X(t;t_0,x_2)\vert^2&=2\left(X(t;t_0,x_1)-X(t;t_0,x_2)\right)\cdot \left(u(t,X(t;t_0,x_1))-u(t,X(t;t_0,x_2))\right)\\
&\leq2\Vert u(t,\cdot) \Vert_{W^{1,\infty}(\RR^N,\RR^N)}\vert X(t;t_0,x_1)-X(t;t_0,x_2)\vert^2,
\end{align*}
and consequently,
$$
\vert X(t;t_0,x_1)-X(t;t_0,x_2)\vert \leq\exp\left(\Vert u\Vert_{L^1(0,T;W^{1,\infty}(\RR^N,\RR^N))}\right)\vert x_1-x_2\vert,
$$
for every $0\leq t,t_0<T$ and $x_1,x_2\in \RR^N$.
In particular, one can define its Jacobian determinant a.e. because Lipschitz functions are a.e. differentiable by the Rademacher theorem.
$$J(t;t_0,x_0):=\det\left(\jac_{x_0}X(t;t_0,x_0)\right).$$
If $u$ were smooth, then the following equation (Liouville's theorem) would hold
$$\frac{d}{dt}J(t;t_0,x_0)=\divop u(t,X(t;t_0,x_0))\,J(t;t_0,x_0),$$
and consequently,
$$J(t;t_0,x_0)=\exp\left(\int_{t_0}^t\divop u(s,X(s;t_0,x_0))\,ds\right).$$
In particular, one would obtain the next upper and lower bound of $J$
\begin{equation}\label{J.upper.lower.estimate.ineq}
\exp\left(-\Vert u\Vert_{L^1(0,T;W^{1,\infty})}\right)\leq J(t;t_0,x_0)\leq \exp\left(\Vert u\Vert_{L^1(0,T;W^{1,\infty})}\right).
\end{equation}
Nevertheless, it is not compulsory to assume that $u$ is smooth. Indeed, if $u\in L^1(0,T;W^{1,\infty}(\RR^N,\RR^N))$ then (\ref{J.upper.lower.estimate.ineq}) also holds by a density argument (see \cite[Section 2]{Ambrosio}).

The corresponding theory for classical solutions with smooth initial data $\rho^0$ provides the following candidate of weak solution
$$\rho(t,x)=\rho^0(X(0;t,x))\,J(0;t,x)=\left.\left(\frac{\rho^0(y)}{J(t;0,y)}\right)\right\vert_{y=X(0,t,x)}.$$
First of all, let us show that although measurable functions might not have sense at some points (because they are defined almost everywhere) the above definition makes sense and does not depend on the representative. To this end, note that $X(0;t,\cdot)$ has its Jacobian determinant lower and upper bounded by positive constants by virtue of (\ref{J.upper.lower.estimate.ineq}) and, consequently, the theorem of change of variables shows that the flux $X(0;t,\cdot)$ preserves the negligible sets in $\RR^N$. Thus, no matter the chosen representative of the measurable function, the above composition yields functions that agree almost everywhere (hence, representing the same equivalence class).

On the one hand, the desired bound in $L^\infty(0,T;L^p(\RR^N))$ when $1<p\leq \infty$ is clear by the change of variables theorem for bi-Lipschitz vector fields and the above bound (\ref{J.upper.lower.estimate.ineq}). Similarly, the $L^\infty(0,T;L^1(\RR^N))$ estimate is also apparent when $p=1$ through the same argument since the $L^1(\RR^N)$ norm is indeed constant for all time. Let us now show that, so defined, $\rho$ gives rise to a weak solution to (\ref{Scalar.conservation.law.eq}).
We will face first the case $1<p\leq\infty$. Fix any test function $\varphi\in C^1_c([0,T)\times\RR^N)$ and notice that
\begin{align*}
\int_0^T\int_{\RR^N}\left(\frac{\partial\varphi}{\partial t}+u\cdot \nabla\varphi\right)\rho\,dx\,dt&=\int_0^T\int_{\RR^N}\left.\left(\frac{\partial\varphi}{\partial t}+u\cdot\nabla\varphi\right)\right\vert_{(t,X(t;0,x))} \rho^0(x)\,dx\\
&=\int_0^T\int_{\RR^N}\frac{d}{dt}\left[\varphi(t,X(t;0,x))\right]\rho^0(x)\,dx\\
&=-\int_{\RR^N}\varphi(0,x)\rho^0(x)\,d,
\end{align*}
where we have used the fundamental lemma of calculus in the last equality to arrive at the weak formulation (\ref{Scalar.conservation.law.weak.formulation.eq}) for every $\varphi\in C_c(\RR^N)$ (notice that locally absolutely continuous functions also verify such result). A similar argument also proves the weak formulation when $p=1$.

Finally, let us prove the uniqueness of weak solution in the sense of (\ref{Scalar.conservation.law.weak.formulation.eq}) under the assumption that $u$ smooth. Consider two weak solutions $\rho_1(t,x)$ and $\rho_2(t,x)$ to the same Cauchy problem (\ref{Scalar.conservation.law.eq}). Then, it is straightforward that $\rho=\rho_1-\rho_2$ is a weak solution to (\ref{Scalar.conservation.law.eq}) with $\rho^0\equiv 0$, i.e.,
$$\int_0^T\int_{\RR^N}\left(\frac{\partial\varphi}{\partial t}+u\cdot \nabla \varphi\right)\rho\,dx\,dt=0,$$
for every test function $\varphi\in C^1_c([0,T),\RR^N)$. Note that given any test function $\psi\in C_c^1((0,T)\times\RR^N)$ the classical theory also allows solving the Cauchy problem
$$
\left\{
\begin{array}{ll}
\partial_t\varphi+u\cdot \nabla \varphi=\psi, & t\in [0,T]\\
\varphi(T)=0, & 
\end{array}
\right.
$$
then leading to a test function $\varphi\in C^1([0,T),C^1(\RR^N,\RR^N))$. Since $\psi$ is arbitrary, the fundamental lemma of the calculus of variations shows that $\rho\equiv 0$, i.e., $\rho_1\equiv \rho_2$.
\end{proof}

\subsection{Commutator estimates and existence theorem for the limiting macroscopic system}
In this section, we wonder about the regularity properties of the commutator appearing in the limit equation (\ref{Euler.limit.eq}). This is the content of the next
\begin{theo}\label{Commutator.estimate.theo}
Consider $1\leq p_1<p_2\leq \infty$ such that
$$\frac{1}{p_2}<1-\frac{2\lambda}{N}<\frac{1}{p_1}.$$
Then, there exists some positive constant $C=C(p_1,p_2,\lambda,N)$ such that
\begin{align*}
\Vert ([u,I_{N-2\lambda}]\rho)&\Vert_{L^1(0,T;W^{1,\infty}(\RR^N))}\\
&\leq C\Vert u\Vert_{L^1(0,T;W^{1,\infty}(\RR^N))}\Vert \rho\Vert_{L^\infty(0,T;L^{p_1}(\RR^N))}^{(1/{p_2'}-\frac{2\lambda}{N})/(1/{p_1}-1/{p_2})}\Vert \rho\Vert_{L^\infty(0,T;L^{p_2}(\RR^N))}^{(\frac{2\lambda}{N}-1/{p_1'})/(1/{p_1}-1/{p_2})}.
\end{align*}
for every $\rho\in L^\infty(0,T;L^{p_1}(\RR^N))\cap L^\infty(0,T;L^{p_2}(\RR^N))$ and every $u\in L^1(0,T;W^{1,\infty}(\RR^N))$.
\end{theo}
\begin{proof}
Write the commutator of the multiplier $u$ with the Riesz potential in a more explicit way as follows
$$([u,I_{N-2\lambda}]\rho)(t,x)=\int_{\RR^N}\phi_0(\vert x-y\vert)(u(t,x)-u(t,y))\rho(t,y)\,dy.$$
On the one hand, taking $L^\infty$ norms with respect to space yields
$$\Vert ([u,I_{N-2\lambda}]\rho)(t,\cdot)\Vert_{L^\infty(\RR^N)}\leq 2\Vert u(t,\cdot)\Vert_{L^\infty(\RR^N)}\Vert \phi_0*\rho(t,\cdot)\Vert_{L^\infty(\RR^N)}.$$
The assumptions on $\rho$ together with the Hardy--Littlewood--Sobolev theorem of fractional integrals then leads to
\begin{align*}
\Vert ([u,I_{N-2\lambda}]\rho)&\Vert_{L^1(0,T;L^\infty(\RR^N))}\\
&\leq C\Vert u\Vert_{L^1(0,T;L^\infty(\RR^N))}\Vert \rho\Vert_{L^\infty(0,T;L^{p_1}(\RR^N))}^{(1/{p_2'}-\frac{2\lambda}{N})/(1/{p_1}-1/{p_2})}\Vert \rho\Vert_{L^\infty(0,T;L^{p_2}(\RR^N))}^{(\frac{2\lambda}{N}-1/{p_1'})/(1/{p_1}-1/{p_2})}.
\end{align*}
Now, taking derivatives with respect to space in the commutator, we have
$$\frac{\partial}{\partial x_i}[u,I_{N-2\lambda}]\rho=[u,I_{N-(2\lambda+1)}]\rho+\frac{\partial u}{\partial x_i} \phi_0*\rho.$$
Regarding the first term, one can cancel the extra degree of singularity thanks to the Lipschitz continuity with respect to space of $u$, leading to a similar estimate
\begin{align*}
\Vert ([u,I_{N-(2\lambda+1)}]\rho)&\Vert_{L^1(0,T;L^\infty(\RR^N))}\\
&\leq C\Vert u\Vert_{L^1(0,T;W^{1,\infty}(\RR^N))}\Vert \rho\Vert_{L^\infty(0,T;L^{p_1}(\RR^N))}^{(1/{p_2'}-\frac{2\lambda}{N})/(1/{p_1}-1/{p_2})}\Vert \rho\Vert_{L^\infty(0,T;L^{p_2}(\RR^N))}^{(\frac{2\lambda}{N}-1/{p_1'})/(1/{p_1}-1/{p_2})}.
\end{align*}
The second term is easier to estimate in the same manner as follows
\begin{align*}
& \hspace{-1.8cm} \left\Vert\frac{\partial u}{\partial x_i}\phi_0*\rho\right\Vert_{L^1(0,T;L^\infty(\RR^N))}\\
&\leq C\Vert u\Vert_{L^1(0,T;W^{1,\infty}(\RR^N))}\Vert \rho\Vert_{L^\infty(0,T;L^{p_1}(\RR^N))}^{(1/{p_2'}-\frac{2\lambda}{N})/(1/{p_1}-1/{p_2})}\Vert \rho\Vert_{L^\infty(0,T;L^{p_2}(\RR^N))}^{(\frac{2\lambda}{N}-1/{p_1'})/(1/{p_1}-1/{p_2})}.
\end{align*}
Arranging all the terms we are led to the desired estimate.
\end{proof}

Apart from the preceding regularity results, that leads to $W^{1,\infty}$ vector fields $u$, we can indeed obtain some extra integrability result by virtue of the Hardy-Littlewood-Sobolev theorem of fractional integrals. Note that such integrability cannot directly follow from the $W^{1,\infty}$ regularity from the Sobolev embedding theorem.

\begin{theo}\label{Commutator.integrability.estimate.theo}
Consider $1\leq p_1<p_2\leq \infty$ and $1\leq s<\infty$ such that
$$\frac{1}{p_2}<1-\frac{2\lambda}{N}<\frac{1}{p_1}\hspace{0.25cm}\mbox{ and }\hspace{0.25cm}s>\frac{N}{2\lambda}.$$
Then, there exists some positive constant $C=C(p_1,p_2,s,\lambda,N)$ such that
\begin{align*}
\Vert ([u,I_{N-2\lambda}]\rho)&\Vert_{L^1(0,T;L^s(\RR^N))}\\
&\leq C\Vert u\Vert_{L^1(0,T;L^s(\RR^N))}\Vert \rho\Vert_{L^\infty(0,T;L^{p_1}(\RR^N))}^{(1/{p_2'}-\frac{2\lambda}{N})/(1/{p_1}-1/{p_2})}\Vert \rho\Vert_{L^\infty(0,T;L^{p_2}(\RR^N))}^{(\frac{2\lambda}{N}-1/{p_1'})/(1/{p_1}-1/{p_2})}.
\end{align*}
for every $\rho\in L^\infty(0,T;L^{p_1}(\RR^N))\cap L^\infty(0,T;L^{p_2}(\RR^N))$ and every $u\in L^1(0,T;L^s(\RR^N))$.
\end{theo}
\begin{proof}
First, split the commutator into two parts as follows
$$[u,I_{N-2\lambda}]\rho=F+G,$$
where the functions $F$ and $G$ takes the form
$$F=(\phi_0*\rho)u,\hspace{0.25cm}\mbox{ and }\hspace{0.25cm}G=\phi_0*(\rho u).$$
Note that one can apply the above reasoning to estimate $F$ in any Lebesgue space as follows
$$\Vert F\Vert_{L^1(0,T;L^s(\RR^N))}\leq C\Vert u\Vert_{L^1(0,T;L^s(\RR^N))}\Vert \rho\Vert_{L^\infty(0,T;L^{p_1}(\RR^N))}^{(1/{p_2'}-\frac{2\lambda}{N})/(1/{p_1}-1/{p_2})}\Vert \rho\Vert_{L^\infty(0,T;L^{p_2}(\RR^N))}^{(\frac{2\lambda}{N}-1/{p_1'})/(1/{p_1}-1/{p_2})}.$$
Regarding the second term, define the exponent 
$$p:=\frac{N}{N-2\lambda},$$
and note that our hypothesis leads to $p_1<p<p_2$. Then, the interpolation inequality of the Lebesgue spaces shows that $\rho\in L^\infty(0,T;L^p(\RR^N))$ and
$$\Vert \rho\Vert_{L^\infty(0,T;L^p(\RR^N))}\leq \Vert \rho\Vert_{L^\infty(0,T;L^{p_1}(\RR^N))}^\theta\Vert \rho\Vert_{L^\infty(0,T;L^{p_2}(\RR^N))}^{1-\theta},$$
for some exponent $\theta\in (0,1)$ given by
$$\frac{1}{p}=\frac{\theta}{p_1}+\frac{(1-\theta)}{p_2}.$$
By inspection, it is straightforward to check that
$$\theta=\frac{\frac{1}{p_2'}-\frac{2\lambda}{N}}{\frac{1}{p_1}-\frac{1}{p_2}}.$$
Hence,
$$\Vert \rho\Vert_{L^\infty(0,T;L^p(\RR^N))}\leq \Vert \rho\Vert_{L^\infty(0,T;L^{p_1}(\RR^N))}^{(1/{p_2'}-\frac{2\lambda}{N})/(1/{p_1}-1/{p_2})}\Vert \rho\Vert_{L^\infty(0,T;L^{p_2}(\RR^N))}^{(\frac{2\lambda}{N}-1/{p_1'})/(1/{p_1}-1/{p_2})}.$$
Now, the H\"{o}lder inequality shows that $\rho\,u\in L^1(0,T;L^r(\RR^N))$, where
$$\frac{1}{r}=\frac{1}{p}+\frac{1}{s}.$$
Since we are assuming $s>N/2\lambda$, then $r>1$. Moreover, the next identity holds
$$\frac{1}{s}=\frac{1}{r}+\frac{2\lambda}{N}-1.$$
Consequently, the Hardy--Littlewood--Sobolev inequality entails
\begin{align*}
\Vert G\Vert_{L^1(0,T;L^s(\RR^N))}&\leq C\Vert \rho u\Vert_{L^1(0,T;L^r(\RR^N))}\\
&\leq C\Vert u\Vert_{L^1(0,T;L^s(\RR^N))}\Vert \rho\Vert_{L^\infty(0,T;L^p(\RR^N))}\\
&\leq C\Vert u\Vert_{L^1(0,T;L^s(\RR^N))}\Vert \rho\Vert_{L^\infty(0,T;L^{p_1}(\RR^N))}^{(1/{p_2'}-\frac{2\lambda}{N})/(1/{p_1}-1/{p_2})}\Vert \rho\Vert_{L^\infty(0,T;L^{p_2}(\RR^N))}^{(\frac{2\lambda}{N}-1/{p_1'})/(1/{p_1}-1/{p_2})},
\end{align*}
and this end the proof of the theorem.
\end{proof}

In the next part, some particular choices of $s$ will be needed. Specifically, $s$ will equal some entire multiples of the upper and lower exponents $p_1$ and $p_2$ of $\rho$.

\begin{cor}\label{Commutator.integrability.estimate.cor}
Let $\lambda$ be any exponent in $(0,N/2)$ and $1\leq p_1<p_2\leq \infty$ such that
$$\frac{1}{p_2}<1-\frac{2\lambda}{N}<\frac{1}{p_1}.$$
Consider any (large enough) positive integer $k$ so that
$$k>\max\left\{\frac{N}{2\lambda}-1,\frac{N}{2\lambda p_1}\right\},$$
Then, there exists some positive constant $C=C(p_1,p_2,s,\lambda,N)$ such that
\begin{align*}
\Vert ([u,I_{N-2\lambda}]\rho)&\Vert_{L^1(0,T;L^{kp_i}(\RR^N))}\\
&\leq C\Vert u\Vert_{L^1(0,T;L^{kp_i}(\RR^N))}\Vert \rho\Vert_{L^\infty(0,T;L^{p_1}(\RR^N))}^{(1/{p_2'}-\frac{2\lambda}{N})/(1/{p_1}-1/{p_2})}\Vert \rho\Vert_{L^\infty(0,T;L^{p_2}(\RR^N))}^{(\frac{2\lambda}{N}-1/{p_1'})/(1/{p_1}-1/{p_2})}.
\end{align*}
for every $i\in \{1,2\}$ and each $\rho\in L^\infty(0,T;L^{p_1}(\RR^N))\cap L^\infty(0,T;L^{p_2}(\RR^N))$, $u\in L^1(0,T;L^{kp_1}(\RR^N))\cap L^1(0,T;L^{kp_2}(\RR^N))$.
\end{cor}

All the above results amounts to the necessary tools that we will need to construct a solution to (\ref{Euler.limit.eq}) by the Banach contraction principle in the space $W^{1,kp_1,kp_2}$ (see (\ref{normed.space.W.form}) for the definition of such Banach space). This is the content of Theorem \ref{Euler.limit.eq.existence.intro.theo} in the Introduction that we prove next.

\begin{proof}[Proof of Theorem \ref{Euler.limit.eq.existence.intro.theo}]

Let us define two operators $\mathcal{D}$ and $\mathcal{C}$ given by:
\begin{align*}
\mathcal{D}[u]&:=\left.\left(\frac{\rho^0(y)}{J^u(t;0,y)}\right)\right\vert_{y=X^u(0;t,\cdot)},\\
\mathcal{C}[\rho,u]&:=-\frac{1}{c_\lambda^\lambda}[u,I_{N-2\lambda}]\rho,
\end{align*}
for every $u\in L^1(0,T;W^{1,kp_1,kp_2}(\RR^N,\RR^N))$ and $\rho\in L^\infty(0,T;L^{p_1}(\RR^N)\cap L^{p_2}(\RR^N))$, i.e., $\mathcal{D}[u]$ is the solution to the Cauchy problem obtained by means of Theorem \ref{Scalar.conservation.law.theo} and $\mathcal{C}[\rho,u]$ is the commutator. Note that the existence and uniqueness of solution to (\ref{Euler.limit.eq}) amounts to finding a solution to the equation
$$\mathcal{C}[\mathcal{D}[u],u]-\nabla\psi=u,\hspace{0.25cm}u\in L^1(0,T;W^{1,kp_1,kp_2}(\RR^N,\RR^N)).$$
Naturally, it can be restated as a fixed point equation
$$\Phi[u]=u, \hspace{0.25cm}u\in L^1(0,T;W^{1,kp_1,kp_2}(\RR^N,\RR^N)),$$
for the operator $\Phi[u]=\mathcal{C}[\mathcal{D}[u],u]-\nabla\psi$. First, notice that such fixed point problem is well posed because 
$$\Phi(L^1(0,T;W^{1,kp_1,kp_2}(\RR^N,\RR^N)))\subseteq L^1(0,T;W^{1,kp_1,kp_2}(\RR^N,\RR^N)),$$
as a consequence of Theorem \ref{Scalar.conservation.law.theo} and Corollary \ref{Commutator.integrability.estimate.cor}. Indeed,
\begin{align*}
\Vert \mathcal{D}[u]\Vert_{L^\infty(0,T;L^{p_i}(\RR^N))}&\leq \exp\left(\frac{1}{p_i'}\Vert u\Vert_{L^1(0,T;W^{1,\infty}(\RR^N))}\right)\Vert \rho^0\Vert_{L^{p_i}(\RR^N)},\\
\Vert \mathcal{C}[\rho,u]\Vert_{L^1(0,T;W^{1,kp_1,kp_2}(\RR^N,\RR^N))}&\leq C\Vert u\Vert_{L^1(0,T;W^{1,kp_1,kp_2}(\RR^N,\RR^N))}\\
&\hspace{1.5cm}\times\Vert \rho\Vert_{L^\infty(0,T;L^{p_1}(\RR^N))}^{(1/{p_2'}-\frac{2\lambda}{N})/(1/{p_1}-1/{p_2})}\Vert \rho\Vert_{L^\infty(0,T;L^{p_2}(\RR^N))}^{(\frac{2\lambda}{N}-1/{p_1'})/(1/{p_1}-1/{p_2})}.
\end{align*}
Let us show that it is indeed a contraction. To this end, consider $u_1,u_2\in L^1(0,T;W^{1,kp_1,kp_2}(\RR^N,\RR^N))$ and split
$$\Phi[u_1]-\Phi[u_2]=\mathcal{C}[\mathcal{D}[u_1],u_1-u_2]+\mathcal{C}[\mathcal{D}[u_1]-\mathcal{D}[u_2],u_2]=:U+V.$$
The first term can be bounded as follows
\begin{multline*}
\Vert U\Vert_{L^1(0,T;W^{1,kp_1,kp_2}(\RR^N,\RR^N))}\leq C\Vert u_1-u_2\Vert_{L^1(0,T;W^{1,kp_1,kp_2}(\RR^N,\RR^N))}\\
\times\exp\left(\frac{2\lambda}{N}\Vert u_1\Vert_{L^1(0,T;W^{1,\infty}(\RR^N))}\right)\Vert \rho^0\Vert_{L^{p_1}(\RR^N)}^{(1/{p_2'}-\frac{2\lambda}{N})/(1/{p_1}-1/{p_2})}\Vert \rho^0\Vert_{L^{p_2}(\RR^N)}^{(\frac{2\lambda}{N}-1/{p_1'})/(1/{p_1}-1/{p_2})}.
\end{multline*}
Regarding the second term one arrives at the next slightly different estimate
\begin{multline*}
\Vert V\Vert_{L^1(0,T;W^{1,kp_1,kp_2}(\RR^N,\RR^N))}\leq C\Vert u_2\Vert_{L^1(0,T;W^{1,kp_1,kp_2}(\RR^N,\RR^N))}\\
\times \Vert \mathcal{D}[u_1]-\mathcal{D}[u_2]\Vert_{L^\infty(0,T;L^{p_1}(\RR^N))}^{(1/{p_2'}-\frac{2\lambda}{N})/(1/{p_1}-1/{p_2})}\Vert \mathcal{D}[u_1]-\mathcal{D}[u_2]\Vert_{L^\infty(0,T;L^{p_2}(\RR^N))}^{(\frac{2\lambda}{N}-1/{p_1'})/(1/{p_1}-1/{p_2})}.
\end{multline*}
The next step is to estimate $\mathcal{D}[u_1]-\mathcal{D}[u_2]$. To this end, the $L^{kp_i}$ estimate of the commutator in Theorem \ref{Commutator.estimate.theo} and Corollary \ref{Commutator.integrability.estimate.cor} will be essential. As it will be checked, it will not not directly follow from the $W^{1,\infty}$ estimate in Theorem \ref{Commutator.estimate.theo}. Since we are assuming smooth $\rho^0$, there is no problem in considering the pull-back pointwise. Thus,
\begin{multline*}
\mathcal{D}[u_1](t,x)-\mathcal{D}[u_2](t,x)\\
=\rho^0(X^{u_1}(0;t,x))(J^{u_1}(0;t,x)-J^{u_2}(0;t,x))+(\rho^0(X^{u_1}(0;t,x))-\rho^0(X^{u_2}(0;t,x)))J^{u_2}(0;t,x).
\end{multline*}
On the one hand, the first term can easily bounded by virtue of a straightforward change of variables, the Jacobi--Louville formula and the mean value theorem
\begin{align*}
\Vert \rho^0(X^{u_1}(0;t,\cdot))&(J^{u_1}(0;t,\cdot)-J^{u_2}(0;t,\cdot))\Vert_{L^{p_i}(\RR^N)}\\
&\leq \Vert \rho^0(X^{u_1}(0;t,\cdot))\Vert_{L^{p_i}(\RR^N)}\Vert J^{u_1}(0;t,\cdot)-J^{u_2}(0;t,\cdot)\Vert_{L^\infty(\RR^N)}\\
&\leq\Vert\rho^0\Vert_{L^{p_i}(\RR^N)}\exp\left(\frac{1}{p_i}\Vert u_1\Vert_{L^1(0,T;W^{1,\infty}(\RR^N))}\right)\\
&\hspace{1.3cm}\times\exp\left(\Vert u_1\Vert_{L^1(0,T;W^{1,\infty}(\RR^N))}+\Vert u_2\Vert_{L^1(0,T;W^{1,\infty}(\RR^N))}\right)\Vert u_1-u_2\Vert_{L^1(0,T;W^{1,\infty}(\RR^N))}.
\end{align*}
On the other hand, the second term has to be studied separately as follows. First, the Jacobian determinant can be bounded in $L^\infty$ as usual through (\ref{J.upper.lower.estimate.ineq}). Second, the difference of the evaluation of $\rho^0$ along the flow of $u_1$ and $u_2$ can be split by means of the integral remainder version of Taylor's theorem, specifically
\begin{align*}
\rho^0(X^{u_1}&(0;t,x))-\rho^0(X^{u_2}(0;t,x))\\
=&\sum_{0<\vert \gamma\vert\leq k-1}\frac{1}{\vert \gamma\vert !}D^\gamma\rho^0(X^{u_2}(0;t,x))\,(X^{u_1}(0;t,x)-X^{u_2}(0;t,x))^\gamma\\
&+\sum_{\vert\gamma\vert= k}\left(\int_0^1\frac{(1-\theta)^k}{k!}D^\gamma\rho^0(X^{u_2}(0;t,x)+\theta(X^{u_1}(0;t,x)-X^{u_2}(0;t,x)))\,d\theta\right)\\
&\hspace{8cm}\times (X^{u_1}(0;t,x)-X^{u_2}(0;t,x))^\gamma.
\end{align*}
As a consequence, one can obtain the next bound
\begin{align*}
\Vert (&\rho^0(X^{u_1}(0;t,\cdot))-\rho^0(X^{u_2}(0;t,\cdot)))J^{u_2}(0;t,\cdot)\Vert_{L^{p_i}(\RR^N)}\\
&\leq\sum_{0<\vert \gamma\vert\leq k-1}\frac{1}{\vert \gamma\vert !}\Vert D^\gamma\rho^0\Vert_{L^{p_i}(\RR^N)}\exp\left(\frac{1}{p_i}\Vert u_2\Vert_{L^1(0;T;W^{1,\infty}(\RR^N))}\right)\Vert X^{u_1}(0;t,\cdot)-X^{u_2}(0;t,\cdot)\Vert_{L^\infty(\RR^N)}^{\vert \gamma\vert}\\
&+\frac{1}{(k+1)!}\sum_{\vert \gamma\vert=k}\Vert D^\gamma\rho^0\Vert_{L^\infty(\RR^N)}\exp\left(\frac{1}{p_i}\Vert u_2\Vert_{L^1(0;T;W^{1,\infty}(\RR^N))}\right)\Vert X^{u_1}(0;t,\cdot)-X^{u_2}(0;t,\cdot)\Vert_{L^{kp_i}(\RR^N)}^k
\end{align*}

Since $u_i(t,\cdot)$ belongs both to $W^{1,\infty}(\RR^N)$ and $L^{kp_i}$, then a straightforward application of Gronwall's Lemma yields the next upper bounds
\begin{align*}
\Vert X^{u_1}(0;t,\cdot)-&X^{u_2}(0;t,\cdot)\Vert_{L^\infty(\RR^N)}\\
&\leq \exp\left(\Vert u_1\Vert_{L^1(0,T;W^{1,\infty}(\RR^N))}\right)\Vert u_1-u_2\Vert_{L^1(0,T;L^\infty(\RR^N))},\\
\Vert X^{u_1}(0;t,\cdot)-&X^{u_2}(0;t,\cdot)\Vert_{L^{kp_i}(\RR^N)}\\
&\leq \exp\left(\frac{1}{kp_i}\Vert u_1\Vert_{L^1(0,T;W^{1,\infty}(\RR^N))}\right)\exp(\Vert u_2\Vert_{L^1(0,T;W^{1,\infty}(\RR^N))})\Vert u_1-u_2\Vert_{L^1(0,T;L^{kp_i}(\RR^N))}.
\end{align*}
To sum up, there exists some separately increasing function $\kappa_2:\RR_0^+\times\RR_0^+\longrightarrow\RR_0^+$ that does not depend on $u_1$, $u_2$ or $\rho^0$ such that
\begin{multline*}
\Vert \Phi[u_1]-\Phi[u_2]\Vert_{L^1(0,T;W^{1,kp_1,kp_2}(\RR^N,\RR^N))}\leq \Vert \rho^0\Vert_{W^{k,p_1,p_2}(\RR^N)}\\
\times\kappa_2(\Vert u_1\Vert_{L^1(0,T;W^{1,kp_1,kp_2}(\RR^N,\RR^N))},\Vert u_2\Vert_{L^1(0,T;W^{1,kp_1,kp_2}(\RR^N,\RR^N))})\Vert u_1-u_2\Vert_{L^1(0,T;W^{1,kp_1,kp_2}(\RR^N,\RR^N))}.
\end{multline*}
Similarly, one obtains the next estimate of $\Phi[u]$
\begin{multline*}
\Vert \Phi[u]\Vert_{L^1(0,T;W^{1,kp_1,kp_2}(\RR^N,\RR^N))}\\
\leq \Vert \rho^0\Vert_{L^{p_1}(\RR^N)}^{(1/{p_2'}-\frac{2\lambda}{N})/(1/{p_1}-1/{p_2})}\Vert \rho^0\Vert_{L^{p_2}(\RR^N)}^{(\frac{2\lambda}{N}-1/{p_1'})/(1/{p_1}-1/{p_2})}\kappa_1(\Vert u\Vert_{L^1(0,T;W^{1,kp_1,kp_2}(\RR^N,\RR^N))})\\
+\Vert \nabla\psi\Vert_{L^1(0,T;W^{1,kp_1,kp_2}(\RR^N,\RR^N))},
\end{multline*}
for some increasing function $\kappa_1:\RR_0^+\longrightarrow\RR_0^+$ which does not depend on $u$ or $\rho^0$. Consider any radius $R>\Vert \nabla\psi\Vert_{L^1(0,T;W^{1,kp_1,kp_2}(\RR^N,\RR^N))}$ and define the unit ball of $L^1(0,T;W^{1,kp_1,kp_2}(\RR^N,\RR^N))$ centered at the origin
$$\mathcal{B}_R:=\{u\in L^1(0,T;W^{1,kp_1,kp_2}(\RR^N,\RR^N)):\,\Vert u\Vert_{L^1(0,T;W^{1,kp_1,kp_2}(\RR^N,\RR^N))}\leq R\}.$$
Assume that $\rho^0$ is ``small enough'' so that
\begin{align*}
\Vert \rho^0\Vert_{L^{p_1}(\RR^N)}^{(1/{p_2'}-\frac{2\lambda}{N})/(1/{p_1}-1/{p_2})}\Vert \rho^0\Vert_{L^{p_2}(\RR^N)}^{(\frac{2\lambda}{N}-1/{p_1'})/(1/{p_1}-1/{p_2})}\kappa_1(R)&\leq R-\Vert \nabla\psi\Vert_{L^1(0,T;W^{1,kp_1,kp_2}(\RR^N,\RR^N))},\\
\Vert \rho^0\Vert_{W^{k,p_1,p_2}(\RR^N)}\kappa_2(R,R)&<1.
\end{align*}
Then, $\Phi(\mathcal{B}_R)\subseteq \mathcal{B}_R$, $\Phi$ is a contraction for the norm $\Vert \cdot\Vert_{L^1(0,T;W^{1,kp_1,kp_2}(\RR^N,\RR^N))}$ and the Banach contraction principle shows the existence of a unique solution $u$ to (\ref{Euler.limit.eq}) in $\mathcal{B}_R$.
\end{proof}

\begin{rem}
Let us note that the above Theorem does not only provides us with a local existence result in $L^1(0,T;W^{1,kp_1,kp_2}(\RR^N))^N$. Indeed, since $T$ is allowed to take the value $T=+\infty$ in the preceding proof, then a global existence result holds for $L^1(0,+\infty;W^{1,kp_1,kp_2}(\RR^N))^N$ velocity fields. Regarding $L^q(0,T;W^{1,kp_1,kp_2}(\RR^N))^N$ velocity fields with $q>1$ (e.g., $q=\infty$), the preceding estimates in Theorems \ref{Commutator.estimate.theo} and \ref{Commutator.integrability.estimate.theo} can be extended to such framework through similar ideas. Specifically,
\begin{align*}
\Vert \mathcal{D}[u]\Vert_{L^\infty(0,T;L^{p_i}(\RR^N))}&\leq \exp\left(\frac{T^{1/q'}}{p_i'}\Vert u\Vert_{L^q(0,T;W^{1,\infty}(\RR^N))}\right)\Vert \rho^0\Vert_{L^{p_i}(\RR^N)},\\
\Vert \mathcal{C}[\rho,u]\Vert_{L^q(0,T;W^{1,kp_1,kp_2}(\RR^N,\RR^N))}&\leq C\Vert u\Vert_{L^q(0,T;W^{1,kp_1,kp_2}(\RR^N,\RR^N))}\\
&\hspace{1.3cm}\times\Vert \rho\Vert_{L^\infty(0,T;L^{p_1}(\RR^N))}^{(1/{p_2'}-\frac{2\lambda}{N})/(1/{p_1}-1/{p_2})}\Vert \rho\Vert_{L^\infty(0,T;L^{p_2}(\RR^N))}^{(\frac{2\lambda}{N}-1/{p_1'})/(1/{p_1}-1/{p_2})}.
\end{align*}
Consequently, a local in time existence result for (\ref{Euler.limit.eq}) with $L^q(0,T;W^{1,kp_1,kp_2}(\RR^N))$ velocity fields holds for small enough initial datum $\rho^0$ (compared to $T$ and the external force $-\nabla\psi$).
\end{rem}

\section{Other relevant hydrodynamic limits}\label{Section-OtherLimits}

This section focuses on showing that the techniques in Section \ref{Section-Hyperbolic} remains valid for other relevant scalings of the system. Specifically, regarding the frictional case we will be concerned with an intermediate scaling (see Equation (\ref{VPFPCS-dimensionless-semihyperbolicscaling-appendix.form}) in Appendix \ref{Appendix-Scaling}). Note that this scaled system includes the velocity diffusion at a low order of $\varepsilon$. Hence, we cannot expect the system to converge to a Maxwellian as it was expected in the hyperbolic scaling in the preceding section. Nevertheless, such scaling in the diffusion term is compulsory in order to get some estimate on the scaled current (coming from the friction term) that allows passing to the limit $\varepsilon\rightarrow 0$. In the frictionless case, we will introduce another hyperbolic scaling (see Equation (\ref{VPFPCS-dimensionless-hyperbolicscaling-frictionless-appendix.form}) in Appendix \ref{Appendix-Scaling}) where the velocity diffusion is again of a low order of $\varepsilon$ and similar a priori bounds for the current can be obtained from the inertial terms, not from the friction term. 

The final goal of such section will be to address the well known Rayleigh--Helmholtz friction, that has been considered of great help in the modeling of flocking, swarming and self-propelling phenomena through the recent years \cite{CarrilloFornasierToscaniVecil,Romanczuk}. Specifically, we will introduce a new hydrodynamic limit where not only hyperbolic, intermediate scalings and singular influence functions can be considered, but also the Rayleigh--Helmholtz friction can be assumed to approach the classical linear friction when $\varepsilon\rightarrow 0$.

\subsection{Hydrodynamic limit with intermediate scalings in the frictional case}\label{Subsection-SemihyperbolicFrictional}
The intermediate hyperbolic scaling (\ref{VPFPCS-dimensionless-semihyperbolicscaling-appendix.form}) was derived in \ref{Appendix-Scaling} and takes the form
\begin{equation}\label{VPFPCS-dimensionless-semihyperbolicscaling.form}
\varepsilon^{1+\gamma}\frac{\partial f_\varepsilon}{\partial t}+\varepsilon v\cdot\nabla_x f_\varepsilon-\varepsilon^\gamma\nabla_x\psi_\varepsilon\cdot \nabla_v f_\varepsilon=\divop_v\left(f_\varepsilon v+\varepsilon^{2\gamma}\nabla_v f_\varepsilon+Q_{CS}^{\phi_\varepsilon}(f_\varepsilon,f_\varepsilon)\right),
\end{equation}
for some parameter $\gamma\in [0,1]$. Note that when $\gamma=0$, it agrees with the hyperbolic scaling in Section \ref{Section-Hyperbolic} and the choice $\gamma=1$ reminds us of a parabolic scaling. In this case, the hierarchy of velocity moments now takes the from
\subsubsection*{Mass conservation}
\begin{equation}\label{MomentumEquation-mass2.form}
\frac{\partial \rho_\varepsilon}{\partial t}+\divop_x \left(\frac{j_\varepsilon}{\varepsilon^\gamma}\right)=0.
\end{equation}
\subsubsection*{Current balance}
\begin{equation}\label{MomentumEquation-current2.form}
\varepsilon^{1+\gamma}\frac{\partial j_\varepsilon}{\partial t}+\varepsilon\divop_x \mathcal{S}_\varepsilon+\varepsilon^\gamma\rho_\varepsilon\nabla_x\psi_\varepsilon+\left(1+\phi_\varepsilon*\rho_\varepsilon\right)j_\varepsilon-\left(\phi_\varepsilon*j_\varepsilon\right)\rho_\varepsilon=0.
\end{equation}
\subsubsection*{Stress tensor balance}
\begin{equation}\label{MomentumEquation-energy2.form}
\varepsilon^{1+\gamma}\frac{\partial \mathcal{S_\varepsilon}}{\partial t}+\varepsilon\divop_x\mathcal{T}_\varepsilon+2\varepsilon^\gamma\Sym(j_\varepsilon\otimes \nabla_x\psi_\varepsilon)+2\left(\left(1+\phi_\varepsilon*\rho_\varepsilon\right)\mathcal{S}_\varepsilon-\varepsilon^{2\gamma}\rho_\varepsilon I\right)-2\Sym((\phi_\varepsilon*j_\varepsilon)\otimes j_\varepsilon)=0.
\end{equation}

Thanks to our choice for the scaling of the velocity diffusion term, one arrives at an analogue of Corollary \ref{Moments_bounds.cor}. Specifically,

\begin{cor}\label{Moments_bounds.cor2}
Let the initial distribution functions $f_\varepsilon^0$ verify (\ref{hypothesis.initial.data.f.form}), the external forces $-\nabla\psi_\varepsilon$ fulfil (\ref{hypothesis.initial.data.F.form}) and consider the strong global in time solution $f_\varepsilon$ to (\ref{VPFPCS-dimensionless-semihyperbolicscaling.form}) with initial data $f_\varepsilon^0$ and $\lambda\in (0,N/2)$. Then, for any nonnegative integer $k$ the next bound holds true
\begin{multline*}
k\left\Vert \vert v\vert^k \frac{f_\varepsilon}{\varepsilon^{2\gamma}}\right\Vert_{L^1(0,T;L^1(\RR^{2N}))}\\
+\frac{k}{2}\frac{1}{\varepsilon^{2\gamma}}\int_0^T\int_{\RR^{4N}}(\vert v\vert^{k-2}v-\vert w\vert^{k-2}w)\cdot (v-w)\phi_\varepsilon(\vert x-y\vert)f_\varepsilon(t,x,v)f_\varepsilon(t,y,w)\,dx\,dt\,dv\,dw\,dt\\
\leq \varepsilon^{1-\gamma}\Vert\vert v\vert^k f_\varepsilon(0)\Vert_{L^1(\RR^{2N})}+k(N+k-2)\Vert \vert v\vert^{k-2}f_\varepsilon\Vert_{L^1(0,T;L^1(\RR^{2N}))}+k\left\Vert \vert v\vert^{k-1}\frac{f_\varepsilon}{\varepsilon^\gamma} \nabla_x\psi_\varepsilon\right\Vert_{L^1(0,T;L^1(\RR^{2N}))}.
\end{multline*}
\end{cor}

Again, the choice $k=2$ leads to the next estimates for the scaled first and second order velocity moments.

\begin{cor}\label{Moments_bounds_k=2.cor2}
Let the initial distribution functions $f_\varepsilon^0$ verify (\ref{hypothesis.initial.data.f.form}), the external forces $-\nabla\psi_\varepsilon$ fulfil (\ref{hypothesis.initial.data.F.form}) and consider the strong global in time solution $f_\varepsilon$ to (\ref{VPFPCS-dimensionless-semihyperbolicscaling.form}) with initial data $f_\varepsilon^0$ and $\lambda\in (0,N/2)$. Then,
\begin{align*}
\left\Vert \vert v\vert \frac{f_\varepsilon}{\varepsilon^\gamma}\right\Vert_{L^2(0,T;L^1(\RR^{2N}))}&\leq M_0^{1/2}\left(\left\Vert \vert v\vert^2 \frac{f_\varepsilon}{\varepsilon^{2\gamma}}\right\Vert_{L^1(0,T;L^1(\RR^{2N}))}\right)^{1/2},\\
\left\Vert \vert v\vert^2 \frac{f_\varepsilon}{\varepsilon^{2\gamma}}\right\Vert_{L^1(0,T;L^1(\RR^{2N}))}&\leq 2\varepsilon^{1-\gamma}E_0+\left(2NT+F_0^2\right)M_0.
\end{align*}
In addition, the next bound also holds
$$
\frac{1}{\varepsilon^{2\gamma}}\int_0^T\int_{\RR^{4N}}\phi_\varepsilon(\vert x-y\vert)\vert v-w\vert^2f_\varepsilon(t,x,v)f_\varepsilon(t,y,w)\,dx\,dy\,dv\,dw\,dt\leq 2\varepsilon^{1-\gamma}E_0+\left(2NT+F_0^2\right)M_0.
$$
\end{cor}
By virtue of the preceding Corollary \ref{Moments_bounds_k=2.cor2}, we can pass to the limit again in the equation of balance of current (\ref{MomentumEquation-current2.form}) and we obtain
$$\lim_{\varepsilon\rightarrow 0}\frac{j_\varepsilon}{\varepsilon^\gamma}=\lim_{\varepsilon\rightarrow 0}\left\{\left(\phi_\varepsilon*\frac{j_\varepsilon}{\varepsilon^\gamma}\right)-(\phi_\varepsilon*\rho_\varepsilon)\frac{j_\varepsilon}{\varepsilon^\gamma}-\rho_\varepsilon\nabla\psi_\varepsilon\right\},$$
in the sense of distributions. The next step is to identify the limit of the nonlinear term in the above commutator. The same ideas as above shows that
\begin{align*}
\rho_\varepsilon&\overset{*}{\rightharpoonup}\rho, \hspace{0.25cm} \mbox{in }L^\infty(0,T;\mathcal{M}(\RR^N)),\\
\frac{j_\varepsilon}{\varepsilon^\gamma}&\overset{*}{\rightharpoonup} j, \hspace{0.25cm} \mbox{in }L^2(0,T;\mathcal{M}(\RR^N))^N.
\end{align*}
Again, the continuity equation provides us with some extra compactness with respect to time (see Theorem \ref{Mass_convergence_strong_weak_star.theo}), namely,
$$\rho_\varepsilon\overset{*}{\rightarrow} \rho\hspace{0.25cm}\mbox{in }C([0,T];\mathcal{M}(\RR^N)-weak\,*)$$
that ensure the convergence of the tensor product of both measures
$$\rho_\varepsilon\otimes \frac{j_\varepsilon}{\varepsilon^\gamma}\overset{*}{\rightharpoonup} \rho\otimes j\hspace{0.25cm}\mbox{in }L^2(0,T;\mathcal{M}(\RR^{2N}))^N.$$
Consequently, we recover an analogous convergence result to Corollary \ref{Commutator_forces_limits.cor} for any exponent $\lambda\in (0,1/2]$ and the same limiting macroscopic system (\ref{Euler-limit-j}), hence (\ref{Euler-limit.intro}).

\begin{rem}
Recall that the intermediate scaling in Equation (\ref{VPFPCS-dimensionless-semihyperbolicscaling.form}) with $\gamma\in (0,1]$ differs from that in the Vlassov--Poisson--Fokker--Planck system \cite{BellouquidCalvoNietoSoler-Intermedio,PoupaudSoler-Parabolico}. Specifically, in this paper the velocity diffusion is of low order of $\varepsilon$ and each term in the Fokker--Planck differential operator is scaled in a different way (recall the choice $\mathcal{V}=\varepsilon^\gamma$ in Appendix \ref{Appendix-Scaling}). 

In particular, the parabolic case $\gamma=1$ in \cite{PoupaudSoler-Parabolico} does not enjoy any bound for the scaled kinetic energy in the spirit of Corollary \ref{Moments_bounds_k=2.cor2}. Consequently, dividing the analogue to the first order moment equation (\ref{VPFPCS-dimensionless-semihyperbolicscaling.form}) by $\varepsilon$ one arrives at a term $\varepsilon^{1-\gamma}\divop_x\mathcal{S}_\varepsilon$ that cannot be shown to converge towards zero if $\gamma=1$ in contrast with our scaling, where we do enjoy the above-mentioned bound of the scaled kinetic energy. Indeed, it is shown in \cite{PoupaudSoler-Parabolico} that
$$\divop_x\mathcal{S}_\varepsilon\rightharpoonup \nabla_x\rho\hspace{0.25cm}\mbox{in }\mathcal{D}'((0,T)\times\RR^N).$$
As a consequence, a diffusion term appears in the continuity equation in such scaling of the Vlasov--Poisson--Fokker--Planck system and this does no longer happen in our particular parabolic scaling (\ref{VPFPCS-dimensionless-semihyperbolicscaling.form}) of the kinetic Cucker--Smale model.
\end{rem}

\subsection{Hydrodynamic limit in the frictionless case}\label{Subsection-HyperbolicFrictionless}
In the frictionless case, the same technique cannot provide us with the desired a priori estimates since the friction terms is lacking in the system. Consequently, one has to rely on estimates that arise from the inertial terms in the equations. To this end, the appropriate choice of the scaling is some hyperbolic scale where the velocity diffusion and external force are of low order of $\varepsilon$. Specifically, we will consider the scaled system (\ref{VPFPCS-dimensionless-hyperbolicscaling-frictionless-appendix.form}) in Appendix \ref{Appendix-Scaling}, i.e.,
\begin{equation}\label{VPFPCS-dimensionless-hyperbolicscaling-frictionless.form}
\varepsilon \frac{\partial f_\varepsilon}{\partial t}+\varepsilon v\cdot \nabla_x f_\varepsilon-\varepsilon\nabla_x\psi_\varepsilon\cdot \nabla_v f_\varepsilon=\divop_v(\varepsilon \nabla_v f_\varepsilon+f_\varepsilon \phi_\varepsilon*\rho_\varepsilon v-f_\varepsilon\phi_\varepsilon*j_\varepsilon).
\end{equation}
The associated hierarchy of velocity moments then takes the form
\subsubsection*{Mass conservation}
\begin{equation}\label{MomentumEquation-mass.frictionless.form}
\frac{\partial \rho_\varepsilon}{\partial t}+\divop_x j_\varepsilon=0.
\end{equation}
\subsubsection*{Current balance}
\begin{equation}\label{MomentumEquation-current.frictionless.form}
\varepsilon\frac{\partial j_\varepsilon}{\partial t}+\varepsilon\divop_x \mathcal{S}_\varepsilon+\varepsilon\rho_\varepsilon\nabla_x\psi_\varepsilon+\left(\phi_\varepsilon*\rho_\varepsilon\right)j_\varepsilon-\left(\phi_\varepsilon*j_\varepsilon\right)\rho_\varepsilon=0.
\end{equation}
\subsubsection*{Energy balance}
\begin{equation}\label{MomentumEquation-energy.frictionless.form}
\varepsilon\frac{\partial \mathcal{S_\varepsilon}}{\partial t}+\varepsilon\divop_x\mathcal{T}_\varepsilon+2\varepsilon \Sym(j_\varepsilon\otimes \nabla_x\psi_\varepsilon)+2\left(\left(\phi_\varepsilon*\rho_\varepsilon\right)\mathcal{S}_\varepsilon-\varepsilon\rho_\varepsilon I\right)-2\Sym((\phi_\varepsilon*j_\varepsilon)\otimes j_\varepsilon)=0.
\end{equation}
First, let us obtain some estimates for the moments of $f_\varepsilon$ in the same spirit as Proposition \ref{Momenta_equations.pro}. Straightforward computations that are identical to that in the above-mentioned Propositions yields the next analogue.
\begin{pro}\label{Momenta_equations.pro3}
Let the initial distribution functions $f_\varepsilon^0$ verify (\ref{hypothesis.initial.data.f.form}), the external forces $-\nabla\psi_\varepsilon$ be bounded in $L^2(0,T;L^\infty(\RR^N))^N$ and consider the strong global in time solution $f_\varepsilon$ to (\ref{VPFPCS-dimensionless-semihyperbolicscaling.form}) with initial data $f_\varepsilon^0$ and $\lambda\in (0,N/2)$. Consider any nonnegative integer $k$. Then, the $k$-th order moments in $x$ and $v$ of $f_\varepsilon$ obey the following equations,
\begin{align*}
\varepsilon\frac{\partial}{\partial t}\int_{\RR^N}\int_{\RR^N}\vert v\vert^k f_\varepsilon\,dx\,dv=&-k\varepsilon\int_{\RR^N}\int_{\RR^N}\vert v\vert^{k-2}v\cdot\nabla_x\psi_\varepsilon f_\varepsilon\,dx\,dv\nonumber\\
&+k(N+k-2)\varepsilon\int_{\RR^N}\int_{\RR^N}\vert v\vert^{k-2}f_\varepsilon\,dx\,dv\nonumber\\
&-k\int_{\RR^N}\int_{\RR^N}\vert v\vert^{k-1}\left((\phi_\varepsilon*\rho_\varepsilon) \vert v\vert-(\phi_\varepsilon*j_\varepsilon)\cdot\frac{v}{\vert v\vert}\right)f_\varepsilon\,dx\,dv\\
\frac{\partial}{\partial t}\int_{\RR^N}\int_{\RR^N}\vert x\vert^k f_\varepsilon\,dx\,dv=&\hspace{0.35cm}k\int_{\RR^N}\int_{\RR^N}\vert x\vert^{k-2}x\cdot v f_\varepsilon\,dx\,dv
\end{align*}
Note that in particular the total mass remains constant, i.e.,
$$
\frac{d}{dt}\int_{\RR^N} \rho_\varepsilon(t,x)\,dx=0.
$$
\end{pro}
As a direct consequence, the case $k=2$ entails the next analogue of Corollary \ref{Moments_bounds.cor}.
\begin{cor}\label{Moments_bounds_k=2.cor3}
Let the initial distribution functions $f_\varepsilon^0$ verify (\ref{hypothesis.initial.data.f.form}), the external forces be bounded in $L^2(0,T;L^\infty(\RR^N))^N$ and consider the strong global in time solution $f_\varepsilon$ to (\ref{VPFPCS-dimensionless-semihyperbolicscaling.form}) with initial data $f_\varepsilon^0$ and $\lambda\in (0,N/2)$. Then,
\begin{align*}
\left\Vert \vert v\vert f_\varepsilon\right\Vert_{L^\infty(0,T;L^1(\RR^{2N}))}&\leq M_0^{1/2}\left(\left\Vert \vert v\vert^2 f_\varepsilon\right\Vert_{L^\infty(0,T;L^1(\RR^{2N}))}\right)^{1/2},\\
\frac{1}{2}\left\Vert \vert v\vert^2 f_\varepsilon\right\Vert_{L^\infty(0,T;L^1(\RR^{2N}))}&\leq 2E_0+\left(2NT+2F_0^2\right)M_0.
\end{align*}
In addition, the next bound also holds
\begin{align*}
\frac{1}{\varepsilon^{2\lambda}}\int_0^T\int_{\RR^{4N}}\phi_\varepsilon(\vert x-y\vert)\vert v-w\vert^2f_\varepsilon(t,x,v)f_\varepsilon(t,y,w)&\,dx\,dy\,dv\,dw\,dt\\
&\leq 2\varepsilon^{1-2\lambda}E_0+\varepsilon^{1-2\lambda}\left(2NT+2F_0^2\right)M_0.
\end{align*}
\end{cor}

Again, the preceding estimates can be arranged to show that
$$\rho_\varepsilon\otimes j_\varepsilon\overset{*}{\rightharpoonup} \rho\otimes j\ \mbox{ in }\ L^2(0,T;\mathcal{M}(\RR^{2N}))^N.$$
The same reasoning as in Section \ref{Section-Hyperbolic}, where one distinguishes again the regimes $\lambda\in (0,1/2)$ and $\lambda=1/2$ because of the concentration issues, yields the next result.

\begin{cor}\label{Current_limit_equation.frictionless.cor}
Let the initial distribution functions $f_\varepsilon^0$ verify (\ref{hypothesis.initial.data.f.form}), the external forces $-\nabla\psi_\varepsilon$ be bounded in $L^2(0,T;L^\infty(\RR^N))$ and consider the strong global in time solution $f_\varepsilon$ to (\ref{VPFPCS-dimensionless-semihyperbolicscaling.form}) with initial data $f_\varepsilon^0$ and $\lambda\in (0,1/2]$. Then, the limiting measures $\rho$ and $j$ of the approximate densities and currents $\rho_\varepsilon$ adn $j_\varepsilon$ verify the next system in the sense of distribution
\begin{equation}\label{Euler-limit-j-frictionless.intro}
\left\{
\begin{array}{ll}
\displaystyle \partial_t\rho +\divop j=0, & x\in\RR^N,\,t\in [0,T),\\
\displaystyle 0=(\phi_0*j)\rho-(\phi_0*\rho)j, & x\in\RR^N,\, t\in [0,T)\\
\displaystyle \rho(0,\cdot)=\rho^0, & x\in \RR^N.
\end{array}
\right.
\end{equation}
\end{cor}

\begin{rem}\label{StrongConvergenceCommutator.frictionless.rem}
In this particular case, the last estimate in Corollary \ref{Moments_bounds_k=2.cor3} provides us with an improved bound that lets quantifying a better convergence towards zero of the commutator, namely, strong convergence (not only in the sense of distributions) for the range of parameters $\lambda\in (0,1/2)$. Specifically, note that the Cauchy--Schwartz inequality, the obvious inequality $\phi_\varepsilon\leq
\frac{1}{\varepsilon^\lambda}\phi_\varepsilon^{1/2}$ and Corollary \ref{Moments_bounds_k=2.cor3} entail
\begin{align*}
\Vert (&\phi_\varepsilon*\rho_\varepsilon)j_\varepsilon-(\phi_\varepsilon*j_\varepsilon)\rho_\varepsilon\Vert_{L^2(0,T;L^1(\RR^{2N}))}\\
&\leq \left(\frac{1}{\varepsilon^{2\lambda}}\int_0^T\phi_\varepsilon(\vert x-y\vert)\vert v-w\vert^2 f_\varepsilon(t,x,v)f_\varepsilon(t,y,w)\,dx\,dy\,dv\,dw\,dt\right)^{1/2}\Vert f_\varepsilon(0)\Vert_{L^1(\RR^{2N})}\\
&\leq \varepsilon^{1/2-\lambda}\left(2E_0+\left(2NT+2F_0^2\right)M_0\right)^{1/2}M_0.
\end{align*}
In particular, since $\lambda<1/2$ then
$$\lim_{\varepsilon\rightarrow 0}\Vert(\phi_\varepsilon*\rho_\varepsilon)j_\varepsilon-(\phi_\varepsilon*j_\varepsilon)\rho_\varepsilon \Vert_{L^2(0,T;L^1(\RR^{2N}))}=0.$$
\end{rem}

\subsection{Rayleigh--Helmholtz friction towards linear friction}\label{Subsection-RayleighHelmholtz}

As it has be seen in Section \ref{Section-AnalisisLimit} in the analysis of the limiting equation (\ref{Euler-limit.intro}), the classical friction in the Fokker--Planck differential operator prevents the individuals from the desired self-propelled behavior since (at the microscopic level) it aims at reducing the particles' velocity to zero. Then, classical linear friction tends to halt the dynamics of the individuals unless we include an external force given by a potential $\psi=\psi(t,x)$ (see Section \ref{Section-Hyperbolic} and Subsection \ref{Subsection-SemihyperbolicFrictional}) or we neglect friction effects (see Subsection \ref{Subsection-HyperbolicFrictionless}). Depending on what are we modeling it might (or not) make sense. Actually, it is the nature of the environment where agents live that determines the kind of friction to be considered. For instance, assume that individuals live inside some viscous material so that its velocity decreases proportionally to the velocity itself at a constant rate $\mu$. In such case, the microscopic Langevin equation reads
$$\left\{
\begin{array}{ll}
\displaystyle\hspace{0.30cm}\dot{x_i}=v_i, & \displaystyle t\geq 0,\\
\displaystyle m\dot{v_i}=-\mu v+\frac{1}{n}\sum_{j\neq i}\phi(\vert x_i-x_j\vert)(v_j-v_i)+\sqrt{2D}\xi_i(t), & \displaystyle t\geq 0,
\end{array}
\right.$$
and the mesoscopic description takes the form
$$\frac{\partial f}{\partial t}+v\cdot\nabla_xf=\divop_v(\mu vf+D\nabla_v f+Q_{CS}^\phi(f,f)),$$
which is the kinetic model that we have been focused on so far. Naturally, such friction has being imposed by the medium. However, the individuals might slow down its velocity according to any $v$-dependent friction coefficient $\mu=\mu(v)$. A particularly interesting choice by its consequences in the modeling of self-propelled behavior is the \textit{Rayleigh--Helmholtz friction}, see \cite{CarrilloFornasierToscaniVecil,Romanczuk}. It arose from the theory of sound developed by Rayleigh and Helmholtz and takes the form
$$\mu(v):=\beta \vert v\vert^2-\alpha.$$
Note that, it consists of a decrease term, $-\beta \vert v\vert^2 v$, in the same spirit as the classical friction and an increase term, $+\alpha v$, that can be understood as an intrinsic self-propulsion of individuals to surpass the medium natural friction. Such competition leads to a natural asymptotic velocity $\sqrt{\alpha/\beta}$. To understand it, let us forget about the interactions and stochastic effects and restrict ourselves to one single particle and one spacial dimension. This gives rise to the next first order scalar ODEs
$$\dot{v}=-\mu v, \hspace{0.5cm}\dot{v}=(\alpha -\beta v^2)v.$$
On the one hand, the former has an only velocity equilibrium, namely $v=0$, that is asymptotically stable. On the second hand, the latter enjoys three different equilibria, namely,
$$
v=-\sqrt{\alpha/\beta},\hspace{0.25cm}v=0\hspace{0.2cm}\mbox{and}\hspace{0.2cm}v=\sqrt{\alpha/\beta}.
$$
Here, $v=0$ is unstable whilst the remaining two equilibria are asymptotically stable. Consequently, the velocities evolve as depicted in Figure \ref{fig:Fig3}. Note that when the initial velocity is under the threshold one, the self-propulsion term helps to achieve such asymptotic velocity. Similarly, when the initial velocity is over the threshold one, the medium friction slow down agent's velocity until it relaxes towards the asymptotic one. In particular, note that whenever $\alpha\neq 0$ the asymptotic velocity of the particle is no longer zero, which is a much more realistic property to model flocking behavior of birds under the influence of some friction arising from the medium.
\begin{figure}
\centering
\includegraphics[scale=0.7]{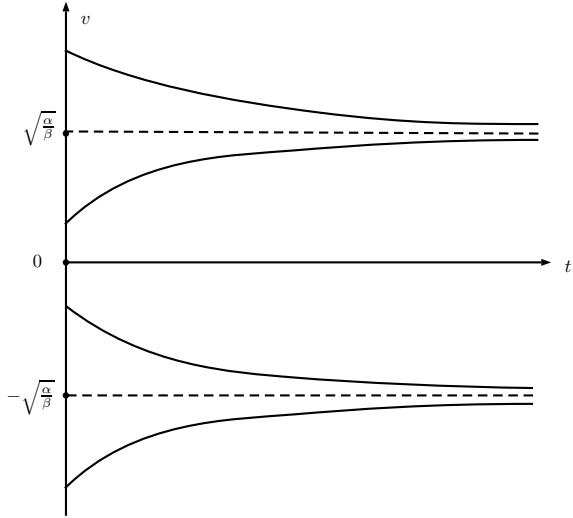}
\centering
\caption{Velocity tendency under the Rayleigh--Helmholtz friction equation}\label{fig:Fig3}
\end{figure}
Hence, why don't we study the same kind of hydrodynamic limit as above when one considers alignment interactions along with Rayleigh--Helmholtz friction, e.g.,
$$\varepsilon^{1+\gamma}\frac{\partial f_\varepsilon}{\partial t}+\varepsilon v\cdot \nabla_x f-\varepsilon^\gamma\nabla_x\psi_\varepsilon\cdot \nabla_v f_\varepsilon=\divop_v((\beta \vert v\vert^2-\alpha)vf+\varepsilon^{2\gamma}\nabla_v f_\varepsilon+Q_{CS}^{\phi_\varepsilon}(f_\varepsilon,f_\varepsilon))\mbox{?}$$ 
The reason is apparent as long as one computes the hierarchy of moments and try to obtain the same type of estimates as in the preceding sections, that validated the rigorous convergence result when $\varepsilon\rightarrow 0$. It is clear that one can obtain bounds for velocity moments of high enough order. However, it is not clear at all how the hydrodynamic limit might help on closing the hierarchy of macroscopic equations of moments since in the equation of any moment (say the i-th velocity moment) some twice higher order moment appears and it cannot be said to converge to zero, but only (at most) to remain bounded. 

\begin{rem}
Another possibility (that we will no study here) might be to consider a first order nonlinear friction enjoying the same sort of self-propulsion effects. For instance, one might consider the next one
$$-\mu(v)v=-\gamma_0(\vert v\vert-v_0)\frac{v}{\vert v\vert},$$
for $\gamma_0,v_0>0$, that is known as the Schienbein--Gruler friction (see \cite{Romanczuk} and the references therein, where its is proposed as a linearization of the above-mentioned Rayleigh--Helmholtz friction in the theory of \textit{active Brownian motions}). It also arises in the modeling of certain type of cell of granulocyte type. As it is apparent, the velocity $v_0$ is the asymptotic one that the system tends to achieve. However, from a mathematical point of view, such friction term is less well behaved due to the obvious discontinuity at $v=0$. Nevertheless, this friction term can also be compared with the classical linear friction in the limit $\vert v\vert\rightarrow +\infty$.
\end{rem}

In this subsection we will compare the first two types of friction (classical and Rayleigh--Helmholtz) by considering a family of intermediate frictions that approach, in the limit $\varepsilon\rightarrow 0$, the classical friction case. Such intermediate frictions take the form
$$\mu(v)=\beta \vert v\vert^k-\alpha,$$
for $0\leq k\leq 2$ and $\alpha,\beta\geq 0$. Note that $k=0$ and $\alpha=0$ yields the classical linear friction. To summarize, the kind of kinetic equations that we will deal with takes the form
$$\frac{\partial f}{\partial t}+v\cdot \nabla_x f-\frac{1}{m}\nabla_x\psi\cdot \nabla_v f=\divop_v((\beta\vert v\vert^k-\alpha)vf+D\nabla_v f+Q_{CS}^\phi(f,f)).$$
First note that the asymptotic velocity of the friction term is now $(\alpha/\beta)^{1/k}$ and it converges towards zero in the coupled limit $\alpha\rightarrow 0$ and $k\rightarrow 0$. A dimensional analysis like in Appendix \ref{Appendix-Scaling} allows introducing the next intermediate scaling of the system
\begin{equation}\label{P_epsilon_frictionlimit.form}
\varepsilon^{1+\gamma}\partial_t f_\varepsilon+\varepsilon v\cdot \nabla_x f_\varepsilon-\varepsilon^\gamma\nabla_x\psi_\varepsilon\cdot \nabla_vf_\varepsilon=\divop_v\left((\vert v\vert^{k(\varepsilon)}-\alpha(\varepsilon))v f_\varepsilon+\varepsilon^{2\gamma}\nabla_v f_\varepsilon+Q_{CS}^{\phi_\varepsilon}(f_\varepsilon,f_\varepsilon)\right),
\end{equation}
for some functions $k=k(\varepsilon)$ and $\alpha=\alpha(\varepsilon)$ such that $k(\varepsilon)\searrow 0$ and $\alpha(\varepsilon)\searrow 0$ when $\varepsilon\searrow 0$. In the formal limit $\varepsilon\rightarrow 0$, we expect to recover the same hydrodynamic limit as with a fixed classical friction,  i.e., Equation (\ref{Euler-limit.intro}). As in the previous case, we refer to  \cite{B,CS,HaTadmor}  for the combination of ideas to deal with the analysis of existence for solutions to this system.

Let us sketch how does the hierarchy of velocity moments looks like and how can we obtain similar estimates to rigorously passing to the limit.

\subsubsection{Hierarchy of moments}
In addition to the velocity moments $\rho_\varepsilon,j_\varepsilon,\mathcal{S}_\varepsilon,\mathcal{T}_\varepsilon$, let us define
\begin{align*}
q_\varepsilon^{k(\varepsilon)+1}&:=\int_{\RR^N}\vert v\vert^{k(\varepsilon)}v f_\varepsilon\,dv,\\
\mathcal{Q}_\varepsilon^{k(\varepsilon)+2}&:=\int_{\RR^N}\vert v\vert^{k(\varepsilon)}v\otimes v f_\varepsilon\,dv.
\end{align*}
Then, the first three velocity moments read as follows.

\subsubsection*{Mass conservation}
\begin{equation}\label{MomentumEquation-mass.limitfriction.form}
\frac{\partial \rho_\varepsilon}{\partial t}+\divop_x\left(\frac{j_\varepsilon}{\varepsilon^\gamma}\right)=0.
\end{equation}

\subsubsection*{Current balance}

\begin{equation}\label{MomentumEquation-current.limitfriction.form}
\varepsilon^{1+\gamma}\frac{\partial j_\varepsilon}{\partial t}+\varepsilon\divop_x \mathcal{S}_\varepsilon+\varepsilon^\gamma\rho_\varepsilon\nabla_x\psi_\varepsilon+(q^{k(\varepsilon)+1}_\varepsilon-\alpha(\varepsilon) j_\varepsilon)+\left(\phi_\varepsilon*\rho_\varepsilon\right)j_\varepsilon-\left(\phi_\varepsilon*j_\varepsilon\right)\rho_\varepsilon=0.
\end{equation}

\subsubsection*{Stress tensor balance}

\begin{multline}\label{MomentumEquation-energy.limitfriction.form}
\varepsilon^{1+\gamma}\frac{\partial \mathcal{S_\varepsilon}}{\partial t}+\varepsilon\divop_x\mathcal{T}_\varepsilon+2\varepsilon^\gamma\Sym(j_\varepsilon\otimes \nabla_x\psi_\varepsilon)+2(\mathcal{Q}_\varepsilon^{k(\varepsilon)+2}-\alpha(\varepsilon)\mathcal{S}_\varepsilon)\\
+2\left(\left(\phi_\varepsilon*\rho_\varepsilon\right)\mathcal{S}_\varepsilon-\varepsilon^{2\gamma}\rho_\varepsilon I\right)-2\Sym((\phi_\varepsilon*j_\varepsilon)\otimes j_\varepsilon)=0.
\end{multline}

\subsubsection{A priori bounds}
We will need some hypothesis on the coefficients $k(\varepsilon)$ and $\alpha(\varepsilon)$ in order to obtain appropriate a priori bounds. On the one hand, in the hyperbolic case, i.e., $\gamma=0$ we will assume that
\begin{equation}\label{Hypothesis.k.alpha.hyperbolic}\tag{$H_{\gamma=0}$}
k(\varepsilon)=o(1)\hspace{0.2cm}\mbox{ and }\hspace{0.2cm}\alpha(\varepsilon)=o(1)\hspace{0.5cm}\mbox{ when }\hspace{0.2cm}\varepsilon\rightarrow 0.
\end{equation}
On the other hand in the purely intermediate or parabolic cases $\gamma\in (0,1]$ we will assume
\begin{equation}\label{Hypothesis.k.alpha.intermediate}\tag{$H_{\gamma\neq 0}$}
k(\varepsilon)=O(\varepsilon^{2\gamma}),\hspace{0.2cm} k'(\varepsilon)=O(\varepsilon^{-(1-\gamma)})\hspace{0.2cm}\mbox{ and }\hspace{0.2cm}\alpha(\varepsilon)=o(1)\hspace{0.5cm}\mbox{ when }\hspace{0.2cm}\varepsilon\rightarrow 0.
\end{equation}
For the sake of simplicity, we first present the next lemma that will help us to control the velocity moments of order $1$, $2$ and $k(\varepsilon)+1$ in terms of that of order $k(\varepsilon)+2$.

\begin{lem}\label{interpolation.moments.frictionlimit.lem}
Let the initial distribution functions $f_\varepsilon^0$ verify (\ref{hypothesis.initial.data.f.form}), the external forces $-\nabla\psi_\varepsilon$ fulfill (\ref{hypothesis.initial.data.F.form}) and consider the strong global in time solution $f_\varepsilon$ to (\ref{P_epsilon_frictionlimit.form}) with initial data $f_\varepsilon^0$ and $\lambda\in (0,N/2)$. Let us also define the exponents
$$p_\varepsilon:=1+\frac{k(\varepsilon)}{2},\hspace{0.25cm}q_\varepsilon:=2+k(\varepsilon),\hspace{0.25cm}r_\varepsilon:=2-\frac{k(\varepsilon)}{1+k(\varepsilon)}$$
Then, the next estimates hold true
\begin{align*}
\Vert \vert v\vert^2 f_\varepsilon\Vert_{L^{p_\varepsilon}(0,T;L^1(\RR^{2N}))}&\leq \Vert \vert v\vert^{k(\varepsilon)+2}f_\varepsilon\Vert_{L^1(0,T;L^1(\RR^{2N}))}^{1/{p_\varepsilon}}M_0^{1/{p_\varepsilon'}},\\
\Vert \vert v\vert f_\varepsilon\Vert_{L^{q_\varepsilon}(0,T;L^1(\RR^{2N}))}&\leq \Vert \vert v\vert^{k(\varepsilon)+2}f_\varepsilon\Vert_{L^1(0,T;L^1(\RR^{2N}))}^{1/{q_\varepsilon}}M_0^{1/{q_\varepsilon'}},\\
\Vert \vert v\vert^{k(\varepsilon)+1}f_\varepsilon\Vert_{L^{r_\varepsilon}(0,T;L^1(\RR^{2N}))}&\leq \Vert \vert v\vert^{k(\varepsilon)+2}f_\varepsilon\Vert_{L^1(0,T;L^1(\RR^{2N}))}^{1/{r_\varepsilon}}M_0^{1/{r_\varepsilon'}}.
\end{align*}
In particular, H\"{o}lder's inequality together with the Young's inequality for real numbers entail the next estimates
\begin{align*}
\Vert \vert v\vert^2 f_\varepsilon\Vert_{L^1(0,T;L^1(\RR^{2N}))}&\leq \Vert \vert v\vert^{k(\varepsilon)+2}f_\varepsilon\Vert_{L^1(0,T;L^1(\RR^{2N}))}^{1/{p_\varepsilon}}\left(TM_0\right)^{1/{p_\varepsilon'}}\\
&\leq \frac{2}{2+k(\varepsilon)}\Vert \vert v\vert^{k(\varepsilon)+2}f_\varepsilon\Vert_{L^1(0,T;L^1(\RR^{2N}))}+\frac{k(\varepsilon)}{2+k(\varepsilon)}TM_0,\\
\Vert \vert v\vert f_\varepsilon\Vert_{L^2(0,T;L^1(\RR^{2N}))}&\leq \Vert \vert v\vert^{k(\varepsilon)+2}f_\varepsilon\Vert_{L^1(0,T;L^1(\RR^{2N}))}^{1/{q_\varepsilon}}T^{\frac{k(\varepsilon)}{2(k(\varepsilon)+2)}}M_0^{1/{q_\varepsilon'}}\\
&=\left(\Vert \vert v\vert^{k(\varepsilon)+2}f_\varepsilon\Vert_{L^1(0,T;L^1(\RR^{2N}))}^{1/{p_\varepsilon}}(TM_0)^{1/{p_\varepsilon'}}\right)^{1/2}M_0\\
&\leq \left(\frac{2}{2+k(\varepsilon)}\Vert \vert v\vert^{k(\varepsilon)+2}f_\varepsilon\Vert_{L^1(0,T;L^1(\RR^{2N}))}+\frac{k(\varepsilon)}{2+k(\varepsilon)}TM_0\right)^{1/2}M_0^{1/2}. 
\end{align*}
\end{lem}

\begin{rem}
Note that we have not considered neither a $L^1$ norm nor a $L^2$ norm with respect to time for the last term $\Vert \vert v\vert^{k(\varepsilon)+1}f_\varepsilon\Vert_{L^{r_\varepsilon}(0,T;L^1(\RR^{2N}))}$. The reason is twofold. 
\begin{enumerate}
\item First, the exponent $r_\varepsilon\nearrow 2$ as $\varepsilon\searrow 0$. Consequently, the natural $L^2$ norm cannot be achieved for each fixed $\varepsilon$ but in the limit $\varepsilon\rightarrow 0$. 
\item Second, one could have obtained a $L^1$ norm with respect to time in the same spirit as in the second order velocity moment since $r_\varepsilon>1$ for small enough $\varepsilon$. However, if we do so, then the coefficient in the Young inequality that comes before $\Vert f_\varepsilon(0)\Vert_{L^1(\RR^{2N})}$ would be $1/(2+k(\varepsilon))$. Although it is obviously bounded with respect to $\varepsilon$, it does not vanish in the limit $\varepsilon\rightarrow 0$ and it is a serious obstruction in order to obtain bounds for the needed scaled velocity moments.
\end{enumerate}
\end{rem}

\begin{pro}\label{Moment_estimates_frictionless.pro}
Let the initial distribution functions $f_\varepsilon^0$ verify (\ref{hypothesis.initial.data.f.form}), the external forces $-\nabla\psi_\varepsilon$ fulfil (\ref{hypothesis.initial.data.F.form}) and consider the strong global in time solution $f_\varepsilon$ to (\ref{P_epsilon_frictionlimit.form}) with initial data $f_\varepsilon^0$ and $\lambda\in (0,N/2)$. Then,
\begin{align*}
\frac{1}{\varepsilon^{2\gamma}}\int_0^T\int_{\RR^{2N}}\vert v\vert^{k(\varepsilon)+2}&f_\varepsilon\,dx\,dv\,dt+\frac{1}{\varepsilon^{2\gamma}}\int_0^T\int_{\RR^{4N}}\phi_\varepsilon(\vert x-y\vert)\vert v-w\vert^2f_\varepsilon(t,x,v)f_\varepsilon(t,y,w)\,dx\,dy\,dv\,dw\\
&\leq \frac{\alpha(\varepsilon)}{\varepsilon^{2\gamma}}\int_0^T\int_{\RR^{2N}}\vert v\vert^2 f_\varepsilon\,dx\,dv\,dt+\varepsilon^{1-\gamma}\int_{\RR^{2N}}\vert v\vert^2 f_\varepsilon(0)\,dx\,dv\\
&\hspace{0.3cm}-\frac{1}{\varepsilon^\gamma}\int_0^T\int_{\RR^{2N}}v\cdot\nabla_x\psi_\varepsilon\,f_\varepsilon\,dx\,dv\,dt+T\int_{\RR^{2N}}f_\varepsilon(0)\,dx\,dv.
\end{align*}
\end{pro}

Now, one can apply apply the estimates of the $L^1$ and $L^2$ norms of the second and first order velocity moments above to obtain the next inequality.

\begin{cor}\label{estimate.higher.moment.frictionlimit.cor}
Under the assumptions of Proposition \ref{Moment_estimates_frictionless.pro}, the next property holds
\begin{multline*}
\left[1-\left(\alpha(\varepsilon)+\frac{1}{2}\right)\frac{2}{2+k(\varepsilon)}\right]\frac{1}{\varepsilon^{2\gamma}}\Vert \vert v\vert^{k(\varepsilon)+2}f_\varepsilon\Vert_{L^1(0,T;L^1(\RR^{2N}))}\\
+\frac{1}{\varepsilon^{2\gamma}}\int_0^T\int_{\RR^{4N}}\phi_\varepsilon(\vert x-y\vert)\vert v-w\vert^2f_\varepsilon(t,x,v)f_\varepsilon(t,y,w)\,dx\,dy\,dv\,dw\\
\leq 2\varepsilon^{1-\gamma}E_0+\left[\frac{1}{\varepsilon^{2\gamma}}\left(\alpha(\varepsilon)+\frac{1}{2}\right)\frac{k(\varepsilon)}{2+k(\varepsilon)}+1\right]TM_0+\frac{1}{2}M_0F_0^2.
\end{multline*}
\end{cor}

For $\varepsilon$ small enough one can obtain a lower estimate of the first factor as follows
$$\left[1-\left(\alpha(\varepsilon)+\frac{1}{2}\right)\frac{2}{2+k(\varepsilon)}\right]\geq \frac{1}{4}.$$
Consequently, the preceding results yield the next list of estimates for the scaled velocity moments under consideration.
\begin{cor}\label{estimates.all.moments.fritionlimit.cor}
Let the initial distribution functions $f_\varepsilon^0$ verify (\ref{hypothesis.initial.data.f.form}), the external forces $-\nabla\psi_\varepsilon$ fulfill (\ref{hypothesis.initial.data.F.form}), consider the strong global in time solution $f_\varepsilon$ to (\ref{P_epsilon_frictionlimit.form}) with initial data $f_\varepsilon^0$ and $\lambda\in (0,N/2)$ and assume the hypothesis (\ref{Hypothesis.k.alpha.hyperbolic})--(\ref{Hypothesis.k.alpha.intermediate}).  Then, there exists some constant $C>0$ that does nor depend on $\varepsilon$ such that
\begin{align*}
\frac{1}{\varepsilon^{2\gamma}}\Vert \vert v\vert^{k(\varepsilon)+2} f_\varepsilon\Vert_{L^1(0,T;L^1(\RR^{2N}))}&\leq C,\\
\frac{1}{\varepsilon^\gamma}\Vert \vert v\vert^{k(\varepsilon)+1} f_\varepsilon\Vert_{L^{r_\varepsilon}(0,T;L^1(\RR^{2N})}&\leq C,\\
\frac{1}{\varepsilon^{2\gamma}}\Vert \vert v\vert^2 f_\varepsilon\Vert_{L^1(0,T;L^1(\RR^{2N}))}&\leq C,\\
\frac{1}{\varepsilon^\gamma}\Vert \vert v\vert f_\varepsilon\Vert_{L^2(0,T;L^1(\RR^{2N}))}&\leq C,
\end{align*}
and 
$$\frac{1}{\varepsilon^{2\gamma}}\int_0^T\int_{\RR^{4N}}\phi_\varepsilon(\vert x-y\vert)\vert v-w\vert^2f_\varepsilon(t,x,v)f_\varepsilon(t,y,w)\,dx\,dy\,dv\,dw\leq C,$$
where $r_\varepsilon$ agrees with the same exponent in Lemma \ref{interpolation.moments.frictionlimit.lem}.
\end{cor}
\begin{proof}
All the estimates obviously follows from Lemma \ref{interpolation.moments.frictionlimit.lem} and Corollary \ref{estimate.higher.moment.frictionlimit.cor}. Let us just sketch the proof of the second one which is less apparent according to the moment relations in Lemma \ref{interpolation.moments.frictionlimit.lem}. Such result shows that
$$\frac{1}{\varepsilon^\gamma}\Vert \vert v\vert^{k(\varepsilon)+1}f_\varepsilon\Vert_{L^{r_\varepsilon}(0,T;L^1(\RR^{2N})}\leq\frac{\varepsilon^{2\gamma/r_\varepsilon}}{\varepsilon^\gamma}\left(\frac{1}{\varepsilon^{2\gamma}}\Vert \vert v\vert^{k(\varepsilon)+2}f_\varepsilon\Vert_{L^1(0,T;L^1(\RR^{2N})}\right)^{1/{r_\varepsilon}}M_0^{1/{r_\varepsilon'}}.$$
Now, one can apply the Young inequality for real numbers once more to obtain the next bound
$$
\frac{1}{\varepsilon^\gamma}\Vert \vert v\vert^{k(\varepsilon)+1}f_\varepsilon\Vert_{L^{r_\varepsilon}(0,T;L^1(\RR^{2N})}\leq \varepsilon^{\gamma\frac{k(\varepsilon)}{1+k(\varepsilon)}}\left\{\frac{1+k(\varepsilon)}{2+k(\varepsilon)}\frac{1}{\varepsilon^{2\gamma}}\Vert \vert v\vert^{k(\varepsilon)+2}f_\varepsilon\Vert_{L^1(0,T;L^1(\RR^{2N}))}+\frac{1}{2+k(\varepsilon)}M_0\right\}.
$$
Note now that
$$\varepsilon^{\gamma\frac{k(\varepsilon)}{1+k(\varepsilon)}}=\exp\left({\frac{\gamma}{1+k(\varepsilon)}k(\varepsilon)\log\varepsilon}\right)\rightarrow 1\ \mbox{ when }\varepsilon\rightarrow 0,$$
because we are assuming $k=O(\varepsilon^{2\gamma})$. In particular, note that one cannot say that such scaled momentum converges to zero as $\varepsilon\rightarrow 0$ because the above factor, although bounded, does not converges to zero.
\end{proof}
The above results show that one can consider weak-star limits for $\rho_\varepsilon$ and $\frac{j_\varepsilon}{\varepsilon^\gamma}$ in the next sense:
\begin{align*}
\rho_\varepsilon&\overset{*}{\rightharpoonup} \rho\hspace{0.25cm}\mbox{in }L^\infty(0,T;\mathcal{M}(\RR^N)),\\
\frac{j_\varepsilon}{\varepsilon^\gamma}&\overset{*}{\rightharpoonup} j \hspace{0.25cm} \mbox{in }L^2(0,T;\mathcal{M}(\RR^N))^N,
\end{align*}
thus leading to the limiting continuity equation in the sense of distributions
$$\frac{\partial \rho}{\partial t}+\divop j=0.$$
Indeed, one can repeat the same idea as in the preceding sections to ensure (thanks to the continuity equation (\ref{MomentumEquation-mass.limitfriction.form})) that
$$\rho_\varepsilon\rightarrow \rho\hspace{0.25cm}\mbox{in }C([0,T],\mathcal{M}(\RR^N)-weak*),$$
and, in particular, $\rho(0,\cdot)=\rho^0$. Recall that the balance law for the scaled current now takes the form
$$\varepsilon^{1+\gamma}\frac{\partial}{\partial t}\left(\frac{j_\varepsilon}{\varepsilon^\gamma}\right)+\varepsilon^{1-\gamma}\divop_x\mathcal{S}_\varepsilon+\rho_\varepsilon \nabla_x\psi_\varepsilon -\left(\frac{1}{\varepsilon^\gamma}q_\varepsilon^{k(\varepsilon)+1}-\alpha(\varepsilon)\frac{j_\varepsilon}{\varepsilon^\gamma}\right)+(\phi_\varepsilon*\rho_\varepsilon)\frac{j_\varepsilon}{\varepsilon^\gamma}-\left(\phi_\varepsilon*\frac{j_\varepsilon}{\varepsilon^\gamma}\right)\rho_\varepsilon=0,$$
or, in weak form,
\begin{multline*}
-\varepsilon^{1+\gamma}\int_0^T\int_{\RR^N}\frac{\partial \varphi}{\partial t}\cdot \frac{j_\varepsilon}{\varepsilon^\gamma}\,dx\,dt-\varepsilon^{1-\gamma}\int_0^T\int_{\RR^N}\jac \varphi:\mathcal{S}_\varepsilon\,dx\,dt\\
+\int_0^T\int_{\RR^N}\rho_\varepsilon\nabla_x\psi_\varepsilon\cdot \varphi\,dx\,dt-\int_0^T\int_{\RR^N}\frac{q_\varepsilon^{k(\varepsilon)+1}}{\varepsilon^\gamma}\cdot \varphi\,dx\,dt+\alpha(\varepsilon)\int_0^T\int_{\RR^N}\frac{j_\varepsilon}{\varepsilon^\gamma}\cdot \varphi\,dx\,dt\\
-\frac{1}{2}\int_0^T\int_{\RR^N}\int_{\RR^N}H_\varphi^{\lambda,\varepsilon}(t,x,y)\cdot\left(\rho_\varepsilon(t,x)\frac{j_\varepsilon}{\varepsilon^\gamma}(t,y)-\rho_\varepsilon(t,y)\frac{j_\varepsilon}{\varepsilon^\gamma}(t,y)\right)\,dx\,dy\,dt=0,
\end{multline*}
for any $\varphi\in C^1_0([0,T];C^1_0(\RR^N))^N$. In all the intermediate hyperbolic cases one can pass to the limit in each term (including the commutator) by virtue of Corollary \ref{estimates.all.moments.fritionlimit.cor}. Recall that it holds for all the parameters in the range $\lambda\in (0,1/2]$. Again, the endpoint case $\lambda=1/2$ has to be considered separately because of the concentration issues of the term $\rho\otimes j-j\otimes \rho$ and the only term that requires a special analysis is the forth term one, i.e.,
$$\int_0^T\int_{\RR^N}\frac{q_\varepsilon^{k(\varepsilon)+1}}{\varepsilon^\gamma}\cdot \varphi\,dx\,dt.$$
Moreover, the first, second and fifth terms vanishes when $\varepsilon\rightarrow 0$. The next result shows that one can indeed pass to the limit in the forth term and identify it in terms of the limit of $j_\varepsilon/\varepsilon^\gamma$, i.e., $j$.

\begin{theo}\label{convergence.friction.frictionlimit.theo}
Let the initial distribution functions $f_\varepsilon^0$ verify (\ref{hypothesis.initial.data.f.form}), the external forces $-\nabla\psi_\varepsilon$ fulfil (\ref{hypothesis.initial.data.F.form}), consider the strong global in time solution $f_\varepsilon$ to (\ref{P_epsilon_frictionlimit.form}) with initial data $f_\varepsilon^0$ and $\lambda\in (0,N/2)$ and assume the hypothesis (\ref{Hypothesis.k.alpha.hyperbolic})--(\ref{Hypothesis.k.alpha.intermediate}). Then,
$$\lim_{\varepsilon\rightarrow 0}\frac{1}{\varepsilon^\gamma}\left\Vert \vert \vert v\vert^{k(\varepsilon)}-1\vert\,\vert v\vert f_\varepsilon\right\Vert_{L^p(0,T;L^1(\RR^{2N}))}=0,$$
for every exponent $1\leq p<2$.
\end{theo}
\begin{proof}
We start by restating the difference $\vert v\vert^{k(\varepsilon)}-1$ though the integral version of the mean value theorem as follows
$$\vert v\vert^{k(\varepsilon)}-1=\int_0^{k(\varepsilon)}\frac{d}{d\theta}\vert v\vert^\theta\,d\theta=\int_0^{k(\varepsilon)}\log\vert v\vert\,\vert v\vert^\theta\,d\theta.$$
As a consequence,
$$\frac{1}{\varepsilon^\gamma}\int_{\RR^{2N}}\vert \vert v\vert^{k(\varepsilon)}-1\vert\,\vert v\vert f_\varepsilon\,dx\,dv\leq \frac{1}{\varepsilon^\gamma}\int_0^{k(\varepsilon)}\int_{\RR^{2N}}\left\vert \log\vert v\vert\right\vert\,\vert v\vert^{\theta+1} f_\varepsilon\,dx\,dv\,d\theta.$$
Then, all our efforts must be conducted to bound logarithmic velocity moments of the particle distribution. This will be done by controlling them in terms of the already known bounds for the velocity moments in Corollary \ref{estimates.all.moments.fritionlimit.cor}. Then, let us recall that for any couple of positive exponent $a,b$, the asymptotic behavior of the logarithm near $r=0$ and $r=\infty$ is explicitly given by
$$
\vert\log r\vert\leq \left\{
\begin{array}{ll}
\displaystyle\frac{1}{a e}r^{-a}, & \displaystyle r\in (0,1],\\
\displaystyle\frac{1}{b e} r^b, & \displaystyle r\in [1,+\infty).
\end{array}
\right.
$$
In a more compact (although less sharp) way
$$\vert \log r\vert\,r^{\theta +1}\leq \frac{1}{a e}r^{\theta+1-a}+\frac{1}{b e}r^{\theta+1+b},$$
for every $\theta\in (0,k(\varepsilon))$ and every $a,b>0$. Fix now any couple of exponents $\alpha\in [0,1)$ and $\beta\in (0,1]$. Note that when $\varepsilon$ is small enough, then $0<k(\varepsilon)<\beta$ and this allows choosing $a=a(\theta)$ and $b=b(\theta)$ as follows
$$a=(1-\alpha)+\theta,\hspace{0.5cm}b=\beta-\theta.$$
Such choice amounts to
\begin{equation}\label{logarithm.bounds.eq}
\vert \log r\vert r^{\theta+1}\leq \frac{1}{((1-\alpha)+\theta)e}r^\alpha+\frac{1}{(\beta-\theta)e}r^{1+\beta},
\end{equation}
for every $r>0$. Thus, our integral can be split into two terms
$$\frac{1}{\varepsilon^\gamma}\int_{\RR^{2N}}\vert \vert v\vert^{k(\varepsilon)}-1\vert\,\vert v\vert f_\varepsilon\,dx\,dv\leq F_\varepsilon(t)+G_\varepsilon(t),$$
where
\begin{align*}
F_\varepsilon(t)&:=\frac{1}{\varepsilon^\gamma}\int_0^{k(\varepsilon)}\frac{1}{((1-\alpha)+\theta)e}\int_{\RR^{2N}}\vert v\vert^\alpha f_\varepsilon\,dx\,dv\,d\theta,\\
&\ =\frac{1}{\varepsilon^\gamma}\frac{1}{e}\log\left(1+\frac{1}{1-\alpha}k(\varepsilon)\right)\Vert \vert v\vert^{\alpha}f_\varepsilon\Vert_{L^1(\RR^{2N})}\\
G_\varepsilon(t)&:=\frac{1}{\varepsilon^\gamma}\int_0^{k(\varepsilon)}\frac{1}{(\beta-\theta)e}\int_{\RR^{2N}}\vert v\vert^{1+\beta}f_\varepsilon\,dx\,dv\,d\theta,\\
&\ =-\frac{1}{\varepsilon^\gamma}\frac{1}{e}\log\left(1-\frac{1}{\beta}k(\varepsilon)\right)\Vert \vert v\vert^{1+\beta} f_\varepsilon\Vert_{L^1(\RR^{2N})}.
\end{align*}
Regarding the first term, let us use once more the H\"{o}lder inequality to obtain
$$\Vert F_\varepsilon\Vert_{L^{2/\alpha}(0,T)}\leq \frac{1}{e}\frac{\log\left(1+\frac{1}{1-\alpha}k(\varepsilon)\right)}{\varepsilon^{\gamma(1-\alpha)}}\left(\frac{1}{\varepsilon^\gamma}\Vert \vert v\vert f_\varepsilon\Vert_{L^1(0,T;L^1(\RR^{2N}))}\right)^{\alpha}\Vert f_\varepsilon(0)\Vert_{L^1(\RR^{2N})}^{1-\alpha}.$$
In this case, Corollary \ref{estimates.all.moments.fritionlimit.cor}, our choice of $k(\varepsilon)$ and L'H\^opital's rule show that
$$\lim_{\varepsilon\rightarrow 0}\Vert F_\varepsilon\Vert_{L^{2/\alpha}(0,T)}=0.$$
Note that the endpoint case $\alpha=0$ is also allowed, giving rise to $L^\infty$ norms with respect to times. Regarding the second term,
$$\Vert G_\varepsilon\Vert_{L^{\frac{2}{1+\beta}}(0,T)}\leq -\frac{1}{e}\log\left(1-\frac{1}{\beta}k(\varepsilon)\right)\varepsilon^{\gamma\beta}\left(\frac{1}{\varepsilon^{2\gamma}}\Vert \vert v\vert^2 f_\varepsilon\Vert_{L^1(0,T;L^1(\RR^{2N}))}\right)^{\frac{1+\beta}{2}}\left(\int_{\RR^{2N}}f_\varepsilon(0)\right)^{\frac{1-\beta}{2}}.$$
The same reasoning as above also shows that
$$\lim_{\varepsilon\rightarrow 0}\Vert G_\varepsilon\Vert_{L^{\frac{2}{1+\beta}}(0,T)}=0.$$
Since both exponents are ordered, namely
$$\frac{2}{1+\beta}<\frac{2}{\alpha},$$
then, one can conclude that
$$\lim_{\varepsilon\rightarrow 0}\frac{1}{\varepsilon^\gamma}\left\Vert\vert\vert v\vert^{k(\varepsilon)}-1\vert\,\vert v\vert f_\varepsilon\right\Vert_{L^{\frac{2}{1+\beta}}(0,T;L^1(\RR^{2N}))}=0,$$
for every $\beta\in (0,1]$. Since $\beta=0$ is not allowed in the splitting (\ref{logarithm.bounds.eq}) (to obtain convergent $\theta$-integrals), then we are led to the desired convergence result for each exponent $1\leq p<2$.
\end{proof}

\begin{cor}\label{convergence.friction.frictionlimit.cor}
Let the initial distribution functions $f_\varepsilon^0$ verify (\ref{hypothesis.initial.data.f.form}), the external forces $-\nabla\psi_\varepsilon$ fulfil (\ref{hypothesis.initial.data.F.form}), consider the strong global in time solution $f_\varepsilon$ to (\ref{P_epsilon_frictionlimit.form}) with initial data $f_\varepsilon^0$ and $\lambda\in (0,N/2)$ and assume the hypothesis (\ref{Hypothesis.k.alpha.hyperbolic})--(\ref{Hypothesis.k.alpha.intermediate}). Then,
$$\frac{1}{\varepsilon^\gamma}q_\varepsilon^{k(\varepsilon)+1}\overset{*}{\rightharpoonup} j$$
when $\varepsilon\rightarrow 0$ both in $L^p(0,T;\mathcal{M}(\RR^N))$, for every $1<p<2$, and in $\mathcal{M}([0,T]\times \RR^N)$.
\end{cor}
\begin{proof}
For simplicity, we restrict to the case $1<p<2$, although the other case can be proved through similar arguments. First, recall that
$$\frac{j_\varepsilon}{\varepsilon^\gamma}\overset{*}{\rightharpoonup} j\hspace{0.25cm}\mbox{in }L^2(0,T;\mathcal{M}(\RR^N)).$$
Since we are assuming $p<2$, then
$$\frac{j_\varepsilon}{\varepsilon^\gamma}\overset{*}{\rightharpoonup} j\hspace{0.25cm}\mbox{in }L^p(0,T;\mathcal{M}(\RR^N)).$$
By virtue of Theorem \ref{convergence.friction.frictionlimit.theo} one has that $\frac{1}{\varepsilon^\gamma}q_\varepsilon^{k(\varepsilon)+1}$ is bounded in $L^p(0,T;L^1(\RR^{2N}))$. Then, some common subsequence (that we do not distinguish from the original for the sake of simplicity) weakly-star converges to some limit $q\in L^p(0,T;\mathcal{M}(\RR^N))$, i.e.,
$$\frac{1}{\varepsilon^\gamma}q_\varepsilon^{k(\varepsilon)+1}\overset{*}{\rightharpoonup} q\hspace{0.25cm}\mbox{in }L^p(0,T;\mathcal{M}(\RR^N)).$$
Hence,
$$\frac{q_\varepsilon^{k(\varepsilon)+1}}{\varepsilon^\gamma}-\frac{j_\varepsilon}{\varepsilon^\gamma}\overset{*}{\rightharpoonup} q-j\hspace{0.25cm}\mbox{in }L^p(0,T;\mathcal{M}(\RR^N)).$$
Then, the weak-star lower semicontinuity of dual norms yields the estimate
\begin{align*}
\Vert q-j\Vert_{L^p(0,T;\mathcal{M}(\RR^N))}&\leq \liminf_{\varepsilon\rightarrow 0}\left\Vert\frac{q_\varepsilon^{k(\varepsilon)+1}}{\varepsilon^\gamma}-\frac{j_\varepsilon}{\varepsilon^\gamma}\right\Vert_{L^p(0,T;L^1(\RR^{2N}))}\\
&\leq \lim_{\varepsilon\rightarrow 0}\frac{1}{\varepsilon^\gamma}\left\Vert \vert \vert v\vert^{k(\varepsilon)+1}-1\vert\,\vert v\vert f_\varepsilon\right\Vert_{L^p(0,T;L^1(\RR^{2N}))}=0.
\end{align*}
Hence, $q=j$ and this ends the proof.
\end{proof}

\appendix

\section{Physical constants, dimensionless model and scalings}\label{Appendix-Scaling}
In this Appendix we introduce the kinetic Cucker--Smale model of interest with all its physical constants. After a dimensionless model, we will identify some dimensionless parameters that will be useful to justify the proposed scalings of hyperbolic and semi-hyperbolic type. 

Let us begin with the Vlasov equation for the distribution function of particles, $f=f(t,x,v)$, taking into account alignment effects of Cucker--Smale type along with some sort of thermal bath modeled by the classical Fokker--Planck term. We may also consider the action of some exterior or self-generated conservative force $F=-\nabla_x\psi$ described by some potential function $\psi=\psi(t,x)$.
$$\frac{\partial f}{\partial t}+v\cdot \nabla_x f-\nabla_x\psi\cdot \nabla_v f=L_{FP}(f)+\divop_v(Q_{CS}(f,f)).$$
The mechanical variables $x,v\in\RR^N$ stand for the position and velocity in any dimension $N$. Here, $L_{FP}(f)$  denote the linear Fokker--Plank operator and $Q_{CS}(f,f)$ is an integro--differential operator of Cucker--Smale type, i.e.,
$$
L_{FP}(f)=\divop_v(vf+\nabla_v f),\hspace{0.7cm} Q_{CS}(f)=f(t,x,v)\int_{\RR^N}\int_{\RR^N} \phi(\vert x-y\vert)(v-w)f(t,y,w)\,dy\,dw.
$$
The influence function $\phi(r)$ is assumed to take the classical form attributed to Cucker and Smale
$$\phi(r)=\frac{1}{(1+r^2)^\lambda},\ r\geq 0.$$
A dimensional analysis amounts to consider the next physical constants
\begin{align*}
m:&\hspace{0.3cm}\mbox{mass of each particle,}\\
\sqrt{\mu}:&\hspace{0.3cm}\mbox{mean thermal velocity,}\\
\tau:&\hspace{0.3cm}\mbox{relaxation time under the thermal bath,}\\
\sigma:&\hspace{0.3cm}\mbox{``range of the effective interactions'',}\\
K:&\hspace{0.3cm}\mbox{``maximum strength of the interactions'',}
\end{align*}
and the associated dimensional model
$$\frac{\partial f}{\partial t}+v\cdot \nabla_x f-\frac{1}{m}\nabla_x\psi\cdot \nabla_v f=\divop_v\left(\frac{\mu}{\tau}\nabla_v f+\frac{1}{\tau}vf+\frac{1}{m}Q_{CS}^{\lambda,K,\sigma}(f,f)\right),$$
where the Cucker--Smale operator is
$$Q_{CS}^{\lambda,K,\delta}(f,f)=f(t,x,v)\int_{\RR^N}\int_{\RR^N}\phi^{\lambda,K,\delta}(\vert x-y\vert)(v-w)f(t,y,w)\,dy\,dw.$$
The corresponding interaction potential takes the form
$$\phi^{\lambda,K,\sigma}(r)=Ka^{\lambda,\sigma}(r):=K\frac{\sigma^{2\lambda}}{(\sigma^2+c_\lambda r^2)^\lambda},\ r\geq 0,$$
for the nonnegative constant $c_\lambda=2^{1/\lambda}-1$. The constants $\mu$ and $\tau$ are well known in the Boltzmann equation of kinetic theory of gases. See \cite[Appendix]{PoupaudSoler-Parabolico} for a comprehensive dimensional analysis of the Vlasov--Poisson--Fokker--Planck system. Let us give some sense to the parameters $\sigma$ and $K$ and explain the above quotation marks.
\begin{itemize}
\item First, notice that the $N$-particles system of Cucker--Smale type simply amounts to the next law for the modification of the velocity of the $i$-th particle:
$$v_i(t+\Delta t)=v_i(t)+\frac{K\Delta t}{N}\sum_{i=1}^Na^{\lambda,\sigma}(\vert x_i(t)-x_j(t)\vert)(v_j(t)-v_i(t)),$$
for small time increments $\Delta t>0$. Since $a^{\lambda,\sigma}(r)$ is dimensionless, then $\Delta t K$ measures the amount of times that one has to add the weighted mean relative velocities of each particle
$$\frac{1}{N}\sum_{i=1}^Na^{\lambda,\sigma}(\vert x_i(t)-x_j(t)\vert)(v_j(t)-v_i(t)),$$
to the initial velocity $v_i(t)$ to obtain the correction $v_i(t+\Delta t)$ of the final velocity. In some sense, $K$ measures the ``strength of the interactions''.

\item Second, our weight function $a^{\lambda,\sigma}(r)$ takes values on $[0,1]$ and the interaction kernel $\phi^{\lambda,K,\sigma}$ achieves its maximum $K$ at $r=0$. Furthermore, $\sigma$ has been considered to characterize the ``range of the effective interactions''. Specifically, particles at distance $r\geq \sigma$ amount to interactions with strength
$$\phi^{\lambda,K,\sigma}(r)\leq \frac{K}{2}.$$
\end{itemize}

Let us now adimensionalize the above model by means of some characteristic units for time, space, velocity, density of particles and potential
$$\overline{t}=\frac{t}{T},\hspace{0.5cm}\overline{x}=\frac{x}{R},\hspace{0.5cm}\overline{v}=\frac{v}{V},\hspace{0.5cm}\overline{f}(\overline{t},\overline{x},\overline{v})=\frac{f(t,x,v)}{f_0}, \hspace{0.5cm}\overline{\psi}(\overline{t},\overline{x})=\frac{\psi(t,x)}{\psi_0}.$$
Straightforward computations on the original model lead to the next dimensionless model, where we have removed bars for the sake of simplicity in the notation
$$
\frac{\partial f}{\partial t}+\frac{VT}{R}v\cdot \nabla_x f-\frac{1}{m}\frac{\psi_0 T}{RV}\nabla_x\psi\cdot \nabla_v f=\divop_v\left(\frac{T}{\tau}fv+\frac{T}{\tau}\frac{\mu}{V^2}\nabla_v f+\frac{V^N R^N f_0T}{m}Q_{CS}^{\lambda,K,\sigma/R}(f,f)\right).
$$
Now, let us introduce a scaling where the characteristic units and the physical constants are linked thought the next formulas:
$$\frac{R}{T}=\frac{1}{m}\frac{\tau}{R}\psi_0,\hspace{0.5cm} M=V^NR^Nf_0,$$
where $M$ is some characteristic mass of the system of particles. The above choices are considered in order for the quotient of the characteristic length over the characteristic time to be the drift velocity associated with the potential $\psi(t,x)$ and the total mass of the rescaled particle distribution function $\overline{f}$ not to depend on the physical parameters in the rescaling (thus remaining constant). Finally, consider the next dimensionless parameters for the thermal mean velocity, the thermal mean free path, the scaled mass of the particles, the scaled range of the effective interactions and the scaled maximum strength of the interactions
$$\alpha:=\frac{\sqrt{\mu}}{R/T}, \hspace{0.5cm}\beta:=\frac{\sqrt{\mu}\tau}{R},\hspace{0.5cm} \mathcal{V}:=\frac{\sqrt{\mu}}{V},\hspace{0.5cm}\mathcal{M}:=\frac{m}{M},\hspace{0.5cm}\delta:=\frac{\sigma}{R},\hspace{0.5cm}\mathcal{K}:=\tau K.$$
Consequently, the corresponding dimensionless model reads
$$
\frac{\partial f}{\partial t}+\frac{\alpha}{\mathcal{V}} v\cdot\nabla_x f-\frac{\mathcal{V}}{\beta}\nabla_x\psi\cdot \nabla_v f=\frac{\alpha}{\beta}\divop_v\left(fv+\mathcal{V}^2\nabla_v f+\frac{1}{\mathcal{M}}Q_{CS}^{\lambda,\mathcal{K},\delta}(f,f)\right).
$$
In this paper, we are interested in the next scalings of the system for a parameter $\varepsilon\rightarrow 0$:
\begin{enumerate}
\item First, a hyperbolic scaling can be obtained by choosing the next order of the parameters:
$$\alpha=1,\hspace{0.5cm}\beta=\varepsilon,\hspace{0.5cm}\mathcal{V}=1,\hspace{0.5cm}\mathcal{M}=1,\hspace{0.50cm}\delta=\varepsilon,\hspace{0.5cm}\mathcal{K}=\varepsilon^{-2\gamma}.$$
Note that, in particular, we are assuming the characteristic velocity of the system to agree with the mean thermal velocity in the thermal bath. In this case the system takes the form
\begin{equation}\label{VPFPCS-dimensionless-hyperbolicscaling-appendix.form}
\varepsilon\frac{\partial f_\varepsilon}{\partial t}+\varepsilon v\cdot\nabla_x f_\varepsilon-\nabla_x\psi_\varepsilon\cdot \nabla_v f_\varepsilon=\divop_v\left(f_\varepsilon v+\nabla_v f_\varepsilon+Q_{CS}^{\phi_\varepsilon}(f_\varepsilon,f_\varepsilon)\right).
\end{equation}
\item Second, one can suppose that the characteristic velocity of the system is much larger that the mean thermal velocity. An appropriate choice of the scaling is
$$\alpha=1,\hspace{0.5cm}\beta=\varepsilon^{1+\gamma},\hspace{0.5cm}\mathcal{V}=\varepsilon^\gamma,\hspace{0.5cm}\mathcal{M}=1,\hspace{0.50cm}\delta=\varepsilon,\hspace{0.5cm}\mathcal{K}=\varepsilon^{-2\gamma},$$
for some parameter $\gamma\in [0,1]$. We will call such scaling, an intermediate hyperbolic scaling and it reads
\begin{equation}\label{VPFPCS-dimensionless-semihyperbolicscaling-appendix.form}
\varepsilon^{1+\gamma}\frac{\partial f_\varepsilon}{\partial t}+\varepsilon v\cdot\nabla_x f_\varepsilon-\varepsilon^\gamma\nabla_x\psi_\varepsilon\cdot \nabla_v f_\varepsilon=\divop_v\left(f_\varepsilon v+\varepsilon^{2\gamma}\nabla_v f_\varepsilon+Q_{CS}^{\phi_\varepsilon}(f_\varepsilon,f_\varepsilon)\right).
\end{equation}
\item Finally, one might be interested on neglecting friction effects. In this case, the factor $fv$ disappears in the above equation. We are interested in a hyperbolic scaling that stands for the next choice of the dimensionless parameters
$$\alpha=1,\hspace{0.5cm}\beta=\varepsilon,\hspace{0.5cm}\mathcal{V}=\varepsilon,\hspace{0.5cm}\mathcal{M}=1,\hspace{0.50cm}\delta=\varepsilon,\hspace{0.5cm}\mathcal{K}=\varepsilon^{-2\gamma}.$$
Such system takes the form
\begin{equation}\label{VPFPCS-dimensionless-hyperbolicscaling-frictionless-appendix.form}
\varepsilon\frac{\partial f_\varepsilon}{\partial t}+\varepsilon v\cdot\nabla_x f_\varepsilon-\varepsilon\nabla_x\psi_\varepsilon\cdot \nabla_v f_\varepsilon=\divop_v\left(\varepsilon\nabla_v f_\varepsilon+Q_{CS}^{\phi_\varepsilon}(f_\varepsilon,f_\varepsilon)\right).
\end{equation}
\end{enumerate}
Here, $\phi_\varepsilon$ is just the local interaction kernel given by
$$\phi_\varepsilon(r):=\frac{1}{(\varepsilon^2+c_\lambda r^2)^\lambda},\ \ r>0.$$
When $\varepsilon\rightarrow 0$, we recover singular kernels of Newtonian type as opposed to the classical regular setting in the original works of Cucker and Smale. Thus, we are interested on taking limits as $\varepsilon\rightarrow 0$, that amounts to consider a coupled hydrodynamic and singular limit in the kinetic model.


\end{document}